\documentclass[times,sort&compress,3p,12pt]{elsarticle}
\usepackage{natbib}
\usepackage{setspace}
\usepackage[labelfont=bf]{caption}

\usepackage{amssymb}
\usepackage[utf8]{inputenc}
\usepackage{amssymb,amsmath,mathrsfs}
\usepackage{amsthm}
\usepackage[colorlinks,
            linkcolor=blue,
            anchorcolor=red,
            citecolor=green
            ]{hyperref}
\numberwithin{equation}{section}
\usepackage{graphicx}
\usepackage{amsmath}
\usepackage{amsmath,amsfonts,amssymb,amsthm,booktabs,color,epsfig,graphicx,hyperref,url}
\usepackage{enumitem}
\usepackage{array}
\usepackage{graphicx}
\usepackage{tikz}
\usetikzlibrary{backgrounds}
\usepackage{xcolor}
\usepackage{yhmath} 
\usepackage{dsfont} 

\theoremstyle{plain}

\definecolor{mycustompurple}{RGB}{154, 36, 79}
\usepackage[utf8]{inputenc}

\usepackage{hyperref} %

\hypersetup{
	colorlinks=true,            %
	linkcolor=blue,              %
	citecolor=blue,            %
	filecolor=mycustompurple,   %
	urlcolor=magenta            %
}

\newtheorem{theorem}{Theorem}[section]
\newtheorem{proposition}[theorem]{Proposition}
\newtheorem{corollary}[theorem]{Corollary}
\newtheorem{lemma}[theorem]{Lemma}
\newtheorem{definition}[theorem]{Definition}

\newtheorem{remark}[theorem]{Remark}

\numberwithin{equation}{section}

\def\C{\mathbb C}
\def\R{\mathbb R}
\def\Re{{\rm Re\hspace{0.05cm}}}
\def\d{\hspace{0.05cm}\mathrm{d}}

\begin{document}
	
	\begin{frontmatter}
		
		\title	{On the direct scattering theory for the
			Calogero-Moser derivative nonlinear Schr\"{o}dinger equation}
		
		\author[1]{Ruoci Sun}
		\author[1]{Deng-Shan Wang}
		\author[2]{Yi Zhao\corref{mycorrespondingauthor}}

		\address[1]{Laboratory of Mathematics and Complex Systems (Ministry of Education), 
			School of Mathematical Sciences, 
			Beijing Normal University, 
			Beijing 100875, China.}
		
		\address[2]{College of Science, Minzu University of China, Beijing 100081, China.}
		
		\cortext[mycorrespondingauthor]{Corresponding author: \url{zhaoyi202406@163.com}}  
		
		\begin{abstract}
 This paper establishes  the direct scattering transform for the Calogero–Moser derivative nonlinear Schrödinger (CMDNLS) equation on the real line. For potentials in a weighted Hardy–Sobolev space, we prove the existence and uniqueness of the Jost functions associated with the Lax operator, and construct the corresponding scattering coefficients including the transmission coefficient and the reflection coefficient. Under the CMDNLS flow, we determine the time evolution of these Jost functions and scattering data, revealing a simple dynamics: the transmission coefficient is invariant, and the reflection coefficient acquires a purely rotational phase. Furthermore, we establish the invariance of a suitable weighted Sobolev space under the flow, which guarantees the persistence of the required regularity for the solutions.

		\end{abstract}

		
	\end{frontmatter}
\tableofcontents

\section{Introduction}\label{sec1}

The  Calogero-Moser derivative nonlinear Schr\"{o}dinger (CMDNLS) equation on the real line reads as 
\begin{equation}\label{21.1}
	i\partial_tu+\partial_x^2u+u(\mathbb{H}-i)\partial_x(|u|^2)=0, \quad (t,x)\in \R\times \R,
	\end{equation}
where $u=u(t,x)$ is a complex-valued function, and
$\mathbb{H}$ is the Hilbert transform
\begin{equation}
	 \mathbb{H} f(x)=\mathrm{P. V.}\frac 1 {\pi} \int_{\R}\frac{f(y)}{x-y}\ \mathrm{d} y.
\end{equation}

 On the one hand, the CMDNLS equation \eqref{21.1} can be regarded as the thermodynamic continuum limit of the classical Calogero-Moser model:
\begin{equation}\label{CMHamil}
 		\mathcal{H}_{CM} = \frac{1}{2} \sum_{j=1}^{N} p_j^2 + \sum_{1 \leq j \neq k \leq N} \frac{1}{2(x_j - x_k)^2},
 	\end{equation}by  Abanov-Bettelheim-Wiegmann  \cite{1}.  
  On the other hand,  
P. G\'erard and E. Lenzmann have completely classified all traveling solitary waves and multisoliton solutions of equation \eqref{21.1} in \cite{2}. Every  multisoliton solution is global in time. For any $s >0$, the $H^s$-norm of each multisoliton solution has polynomial growth, as the time goes to infinity. When  $|t|$ is sufficiently large, every pure multisoliton solution $u$ of \eqref{21.1} has distinct $N$ poles $z_1(t), z_2(t), \cdots, z_N(t) $ in the lower half-plane $\mathbb{C}^-=\{a+b i \in \mathbb{C}: b <0\}$. So $u$ admits the following partial fraction expansion
\begin{equation}
        u(t,x)=\sum_{j=1}^N \frac{a_j(t)}{x-z_j(t)} , \qquad \forall x \in \mathbb{R} , 
    \end{equation}where $a_j(t) \ne 0$, if $   |t| \gg 1$. Moreover, the poles satisfy the following system  
    \begin{equation}\label{CMSys}
    \partial_t^2 z_j = - \sum_{k \ne j} \frac{8}{(z_k-z_j)^3}, \qquad |t| \gg 1,
\end{equation}for any $j=1,2, \cdots, N$. System \eqref{CMSys} is the complexified version of classical Calogero--Moser system. This is the reason why equation \eqref{21.1} is called the Calogero-Moser derivative nonlinear Schr\"odinger equation.

A key feature of  equation \eqref{21.1} is its set of symmetries: it is invariant under phase rotation, scaling, and spatial translation, corresponding to the transformation  
\begin{equation}
    u(t,x) \mapsto e^{i\theta}\lambda^{\frac12}u(\lambda^2t,\lambda x+x_0),\quad x_0\in\mathbb{R},  \quad \theta\in\R, \quad \lambda>0. 
\end{equation}
Equation \eqref{21.1} is  $L^2$-critical. Additionally, it is invariant under the Galilean   transformation
\begin{equation}
    u(t,x)\mapsto e^{i\eta x-it\eta^2} u(t,x-2t\eta),\quad \eta\in \R.
\end{equation}According to G\'erard-Lenzmann  \cite{2}, the CMDNLS equation \eqref{21.1} is locally well-posed in the Hardy-Sobolev space
\begin{equation}\label{HardySobolevHs+}
     H^s_+:= \{f \in H^s(\mathbb{R}; \mathbb{C}) : \mathrm{supp}\hat{f} \subset [0, +\infty) \},
 \end{equation}for any  $s>\frac 12$. Here $\hat{f}$ denotes the Fourier transform of $f$:
 $$\hat{f}(\xi) : = \int_{\mathbb{R}} f(x) e^{-ix \xi} \mathrm{d}x, \quad \forall f \in L^1(\mathbb{R}; \mathbb{C}).$$
Let $L^2_+=L^2_+(\R; \C) = H^0_+$ denote the Hilbert space that consists of all $L^2$ functions whose Fourier transform is   supported on $\mathbb{R}_+ = [0, +\infty)$. Let 
$\Pi=\frac{\mathrm{id}+ i \mathbb{H} }{2} : L^2 (\mathbb{R}; \mathbb{C}) \to L^2 (\mathbb{R}; \mathbb{C})$ denote the Cauchy-Szeg\H{o}  projector. It cancels all negative Fourier modes and preserves all positive Fourier modes.  In other words, for any $f \in L^2= L^2 (\mathbb{R}; \mathbb{C})$, we have
\[
\widehat{\Pi f}(\xi)=\mathds{1}_{\xi\geq 0}\hat f(\xi)=\begin{cases} \hat{f}(\xi), & \xi \geq 0, \\ 0, & \xi < 0. \end{cases}
\]

Equation \eqref{21.1} is  completely integrable. It  admits infinitely many conservation laws, but none of them controls the high regularity Sobolev  norms. P. G\'erard and E. Lenzmann introduced the Lax pair structure of equation \eqref{21.1} in  \cite{2}, i.e. there exist two operators $(\mathbb{L}_u, \mathbb{B}_u)$, depending on $u$, such that if $u : t \mapsto u(t)$ solves the CMDNLS equation \eqref{21.1}, then we have
\begin{equation}\label{LaxEq}
    \partial_t \mathbb{L}_{u}=[\mathbb{B}_u,\mathbb{L}_u].
\end{equation}For any $u \in H^1_+$, the Lax operator {$\mathbb{L}_u$ is defined as follows,}
\begin{equation} \label{21.4}
    \mathbb{L}_u (h):= -i \partial_x h - u \Pi (\overline{u}h) \in L^2_+, \quad \forall h \in H^1_+.
\end{equation}The Lax operator $\left(\mathbb{L}_u, \; \mathrm{Dom}(\mathbb{L}_u) = H^1_+\right)$ is an unbounded self-adjoint operator on the Hilbert space $L^2_+$. When $u \in H^2_+$,  the operator $ \mathbb{B}_u$ is a  bounded   anti-symmetric operator on $L^2_+$.

Matsuno also derived  that  the CMDNLS equation \eqref{21.1} is the compatible condition of the following system  \cite{3}
\begin{align}
	&\quad i\partial_xf + \lambda f + ug^+ = 0, \label{21.6}\\
	&\quad g^+ - g^- - \overline uf = 0, \label{21.7}\\
	&\quad \partial_tf +2\lambda\partial_xf + i\partial_{xx}f +2  g^+\partial_x u  = 0, \\
	&\quad \partial_tg^\pm + 2\lambda \partial_xg^\pm + i\partial_{xx}g^\pm + [ (\pm 1 + i\mathbb{H})\partial_x(|u|^2)  ]g^\pm = 0\label{21.9},
\end{align}
where $\lambda$ is the spectral parameter. The function $g^+(t,x)$ ($g^-(t,x)$) denotes the boundary value of the analytic function in the upper (lower) half complex $z$ plane for $z=x+iy$ and $y>0$ ($y<0$). 
The system \eqref{21.6}-\eqref{21.9} is equivalent to the Lax pair given by \eqref{LaxEq}. The Lax operator $\mathbb{L}_u$ is also called the direct scattering operator.

The direct and inverse scattering transforms (DST and IST) provide a powerful framework for analyzing integrable nonlinear PDEs. The DST maps the initial datum of a nonlinear system into scattering data associated with a corresponding linear wave equation.  The time evolution determines the scattering data at an arbitrary time. Conversely, the IST reconstructs the solution  from the time-evolved scattering data, effectively solving the Cauchy problem via this nonlinear  Fourier transform. This is achieved by solving a linear integral equation, the Gelfand–Levitan–Marchenko equation, or equivalently, by formulating and solving the corresponding Riemann–Hilbert problem.

Beyond solving initial value problems, these methods enable the explicit construction of wide classes of solutions, including multisoliton solutions.  Moreover, the IST serves as a key tool for investigating the long‑time asymptotics of integrable systems. For instance, it plays an important role in addressing the soliton resolution conjecture for integrable equations such as the cubic nonlinear Schrödinger (NLS), derivative NLS,   KdV equation, etc.

We recall some properties about eigenvalues and eigenvectors of $\mathbb{L}_u$.
By Proposition 5.1 of Gérard--Lenzmann  \cite{2}, the self-adjoint operator $\mathbb{L}_u$ has only finitely many eigenvalues and each eigenvalue is simple. We denote its pure point spectrum, with eigenvalues arranged in strictly increasing order, by
\begin{equation}\label{ppspecLu}
\sigma_{pp}(\mathbb{L}_u) = \{\lambda_1^u, \lambda_2^u, \ldots, \lambda_N^u\} \subset \mathbb{R}, \quad \text{with } \; \lambda_j^u < \lambda_{j+1}^u \;\text{ for } \; j = 1, \cdots, N-1.
\end{equation}
Furthermore, the number $N$ of eigenvalues satisfies the estimate
$N \leq \frac{1}{2\pi}\|u\|^2_{L^2}$.
The standard eigenvector $ \varphi_j^u  $ associated to the eigenvalue $\lambda_j^u \in \sigma_{pp}(\mathbb{L}_u)$ is defined as follows
\begin{equation}\label{defEigenVec}
  \mathbb{L}_u  \varphi_j^u  =\lambda_j^u \varphi_j^u , \quad \|\varphi_j^u\|_{L^2_+}=1 \quad \mathrm{and} \quad \langle u, \varphi_j^u \rangle_{L^2_+} = \sqrt{2 \pi}>0.
\end{equation}
\subsection{Main results}

The goal of this paper is to rigorously establish the direct scattering transform for equation \eqref{21.1}. To this end, we construct the Jost functions and scattering data using the Lax operator \eqref{21.4}, the Hardy--Sobolev space $H^s_+$  defined in \eqref{HardySobolevHs+}, and the following weighted $L^2$-space 
\begin{equation}\label{WeightedL2Sp}
    L_s^2 : = \left\{ f : \mathbb{R} \to \mathbb{C} \, : \, \| f \|_{L_s^2}^2 = \int_{\mathbb{R}} |f(x)|^2 (1+x^2)^{s} \d x < +\infty \right\}, \qquad \forall s\geq 0.
\end{equation}Let $\mathbb{C}^+= - \mathbb{C}^- = \{a+bi \in \mathbb{C} : b>0\}$ denote the upper half-plane. The first result concerns the existence and uniqueness of Jost functions of the Lax operator $\mathbb{L}_u$ given by \eqref{21.4}.

\begin{theorem}\label{1.4}
  Assume that $u\in H^1_+\cap L^2_s$ for some $s> \frac 12$. 
\begin{enumerate}
\item 
For every $k\in \C^+\cup \C^-\cup  \big( (-\infty, 0)\setminus \sigma_{pp}(\mathbb{L}_u)\big)$,  there exists a  unique Jost function $\phi_k\in H^1_+ \cap L_{-s}^{\infty}$ such that 
\begin{equation}
\mathbb{L}_u\phi_{k}=k\phi_{k}+u.
\end{equation}
\item 
For every $\lambda\in (0,+\infty)\setminus \sigma_{pp}(\mathbb{L}_u)$, there exist   unique Jost functions $\left( \phi_\lambda^+, \phi_\lambda^- \right)\in L_{-s}^{\infty}$ such that 
\begin{equation}
\begin{split}
&\mathbb{L}_u\phi_{\lambda}^\pm=\lambda\phi_{\lambda}^\pm+u,\\
&\lim_{x\to-\infty}\phi_\lambda^+(x)=0,\\
&\lim_{x\to+\infty}\phi_\lambda^-(x)= 0.
\end{split}
\end{equation}
\end{enumerate}
\end{theorem}
 
For spectral parameter $k \in  \left( \mathbb{C}\backslash [0,+\infty) \right) \backslash \sigma_{pp}(\mathbb{L}_u) $, the Jost function $\phi_k$ belongs to  $H^1_+$ and therefore decays to zero as  $x \to \pm\infty$. In contrast, when $\lambda \in  (0, +\infty)\backslash \sigma_{pp}(\mathbb{L}_u) $, the Jost functions  $\phi_{\lambda}^+$ and $\phi_{\lambda}^-$ vanish only at one end: $\phi_{\lambda}^+(x)\to 0$ as $x\to-\infty$ and $\phi_{\lambda}^-(x)\to 0$ as $x\to+\infty$.  To capture the jump relation between $\phi_\lambda^+$ and $\phi_\lambda^-$ across the continuous spectrum
  and to set up the inverse scattering transform via some specific Riemann–Hilbert problem, it is convenient to introduce another pair of Jost functions, $\psi_\lambda^\pm$, which behave like the exponential $e^{i\lambda x}$ at infinity.

\begin{theorem}\label{thm2}
 	 Assume that $u\in H^1_+\cap L^2_s$ for some $s> \frac 12$. 
     For every $\lambda\in (0,+\infty)\setminus \sigma_{pp}(\mathbb{L}_u)$,
    there exists   unique Jost functions {$(\psi_\lambda^+,\psi_\lambda^-)\in L_{-s}^{\infty} $} such that  
 	\begin{equation} 
    \begin{split}
&\mathbb{L}_u\psi_{\lambda}^\pm=\lambda\psi_{\lambda}^\pm,\\
        &\lim_{x\to -\infty}\left(\psi_\lambda^+-e^{ix\lambda}\right)=  0,\\
    &\lim_{x\to +\infty}\left(\psi_\lambda^--e^{ix\lambda} \right)= 0.
    \end{split}
 	\end{equation}
 \end{theorem}
\begin{remark}
    Here the Jost functions $\phi_k$, $\phi_{\lambda}^{\pm}$ and $\psi_{\lambda}^{\pm}$ all depend on $u \in H^1_+ \cap L^2_s$.
\end{remark}

To establish both direct and inverse scattering transform, we now introduce the scattering coefficients $\Gamma(\lambda)$ and $\beta(\lambda)$. These quantities arise naturally from the relations among the Jost functions: $\Gamma(\lambda)$ connects $\psi_\lambda^+$ and $\psi_\lambda^-$, while $\beta(\lambda)$ characterizes the jump between $\phi_\lambda^+$ and $\phi_\lambda^-$ in terms of $\psi_\lambda^-$. Together with an additional relation linking $\phi_\lambda^-$ and $\psi_\lambda^-$, they yield a nonlocal jump condition that no longer involves $\psi_\lambda^\pm$. 
Their precise definitions  are given in the following theorem.

\begin{theorem}\label{Thm3}    
    Assume that $u\in H^1_+\cap L_s^2$ for some $s>\frac12$. For any $k\in \C^+\cup \C^-\cup  \big( (-\infty, 0)\setminus \sigma_{pp}(\mathbb{L}_u)\big)$ and   $\lambda\in (0,+\infty)\setminus \sigma_{pp}(\mathbb{L}_u)$, let  $\phi_\lambda^\pm$ and $\psi_\lambda^\pm$ be defined in Theorem \ref{1.4} and \ref{thm2}. Set
    \begin{align}
&\Gamma(\lambda):=1 + i\int_\R u\Pi(\overline{u}	\psi_\lambda^+)e^{-i\lambda y}
\mathrm{d} y,\label{introGamma}\\
&	
\beta(\lambda):=i\int_{\mathbb{R}}(u+u\Pi(\overline{u}\phi_{\lambda }^+))e^{-i\lambda y}         \ \mathrm{d} y\label{introbeta},
\end{align}
where the bar denotes  complex conjugate. Then the following identities hold: 
        \begin{align} 
			&\psi_\lambda^+=\Gamma(\lambda)	\psi_{\lambda}^-,\\
&\phi_{\lambda}^+-\phi_{\lambda}^-=\beta(\lambda)\psi_{\lambda}^-,
\\
&
		e_\lambda \partial_{\lambda} \left( e_{-\lambda} \psi_{\lambda}^- \right) = \left( \frac{1}{2\pi} \int_{\mathbb{R}} \overline{u} \psi_{\lambda}^-      \ \mathrm{d} x \right) \phi_{\lambda}^-.
	\end{align}
\end{theorem}

In analogy with the standard scattering theory, the function $\Gamma(\lambda)$ is referred to as the transmission coefficient, and $\beta(\lambda)$ as the reflection coefficient.
These scattering coefficients are fundamental for formulating the associated Riemann–Hilbert problem and completing the inverse scattering transform.

In previous results, the potential $u$ depends only on the spatial variable $x \in \mathbb{R}$. 
Since the Jost functions  $\phi_k$, $\phi_{\lambda}^{\pm}$ and $\psi_{\lambda}^{\pm}$  in Theorem \ref{1.4} and \ref{thm2} depend on the potential $u$, we write
\begin{equation}\label{Jostdeponu}
    \boldsymbol{\phi}_k^u := \phi_k , \quad   \boldsymbol{\phi}_{\lambda}^{\pm, u} :=  \phi_{\lambda}^{\pm} \quad \mathrm{and} \quad \boldsymbol{\psi}_{\lambda}^{\pm, u} :=  \psi_{\lambda}^{\pm},  
\end{equation}for any    $k\in \C^+\cup \C^- \cup  \big((-\infty,0)\setminus \sigma_{pp}(\mathbb{L}_{u})\big) $ and  $\lambda\in (0,+\infty)\setminus \sigma_{pp}(\mathbb{L}_{u})$.   In Theorem \ref{Thm3},  the scattering coefficients $\Gamma $ and $\beta$ also depend on $u$. So we write
\begin{equation}\label{boldGamma}
    \boldsymbol{\Gamma}_u (\lambda): = \Gamma (\lambda)  \quad \mathrm{and} \quad      \boldsymbol{\beta}_u (\lambda): = \beta (\lambda) ,
\end{equation}for any $\lambda\in (0,+\infty)\setminus \sigma_{pp}(\mathbb{L}_{u})$.
When the potential $u$ depends on the time variable $t$, the Jost functions in \eqref{Jostdeponu} and scattering coefficients in \eqref{boldGamma} also become time-dependent. To make this dependence explicit, we introduce the following notations:
 
\begin{equation}\label{JostFctdepont}
\begin{cases}
  \phi_k (t) : =   \boldsymbol{\phi}_k^{u(t)}  , \quad    \phi_{\lambda}^{\pm} (t)  : = \boldsymbol{\phi}_{\lambda}^{\pm, u(t) },    \quad   \psi_{\lambda}^{\pm}(t) : = \boldsymbol{\psi}_{\lambda}^{\pm, u(t)}  ,  \\\Gamma(\lambda, t) : = \boldsymbol{\Gamma}_{u(t)} (\lambda), \quad  \mathrm{and} \quad \beta(\lambda, t) : = \boldsymbol{\beta}_{u(t)} (\lambda),
  \end{cases}
\end{equation}for any    $k\in \C^+\cup \C^- \cup  \big((-\infty,0)\setminus \sigma_{pp}(\mathbb{L}_{u})\big) $ and  $\lambda\in (0,+\infty)\setminus \sigma_{pp}(\mathbb{L}_{u})$.  
The following theorem describes how these Jost functions and scattering coefficients evolve in time when $u$ is an  $H^3_+$-solution of \eqref{21.1} such that $u(t)$ belongs to some weighted Sobolev space. If $n\geq 1$ is an integer,  the weighted Sobolev space $H^{n}_s$ is defined as follows:    \begin{equation}\label{WeightSoboH2s}
            H^{n}_s =H^{n}_s(\mathbb{R}; \mathbb{C}): =\{f \in H^{n} (\mathbb{R}; \mathbb{C}): f, \partial_x f, \partial_x^2 f, \cdots, \partial_x^n f\in L^2_s\}, \qquad \forall s \geq 0.
    \end{equation}

\begin{theorem}\label{ThmTimEvoJostScatt}
Assume that
    $u \in C^0\left( (T^-, T^+); H^2_s\cap H^3_+\right)$ solves the CMDNLS equation \eqref{21.1} with initial datum $u(0)=u_0 \in H^2_s \cap H^3_+$, 
     for some $T^- < 0 <T^+$ and  \( s > \frac 12 \).   For any $t \in (T^-, T^+)$,  $$k\in \C^+\cup \C^- \cup  \big((-\infty,0)\setminus \sigma_{pp}(\mathbb{L}_{u_0})\big),\quad \lambda\in (0,+\infty)\setminus \sigma_{pp}(\mathbb{L}_{u_0}),$$	 the Jost functions $\phi_k(t)$, $\phi_\lambda^\pm(t)$ , $\psi_\lambda^\pm(t)$ and the scattering coefficients $\Gamma(\lambda,t)$,  $\beta(\lambda,t)$ are defined in  \eqref{JostFctdepont}.  
    Then the following identities hold:  
    \begin{equation}\label{TimEvphik}
        \partial_t \phi_k = i \partial_x^2 \phi_k + 2 u \partial_x \Pi (\overline{u}\phi_k ),
    \end{equation}
\begin{equation}\label{TimEvphilambda}
        \partial_t \phi_\lambda^\pm = i \partial_x^2 \phi_\lambda^\pm + 2 u \partial_x \Pi (\overline{u}\phi_\lambda^\pm ),
    \end{equation}
\begin{equation}\label{TimEvPsiLambda}
        \partial_t \psi_\lambda^\pm = i \partial_x^2 \psi_\lambda^\pm + 2 u \partial_x \Pi (\overline{u}\psi_\lambda^\pm )+i \lambda^2 \psi_\lambda^\pm,
    \end{equation}
\begin{equation}\label{TimEvScattCoeff}
        \partial_t\Gamma(\lambda, t)=0,\quad \partial_t\beta(\lambda,t)=-i\lambda^2\beta(\lambda,t),\quad \forall t\in (T^-,T^+).
    \end{equation}
\end{theorem}
Equations \eqref{TimEvphik}, \eqref{TimEvphilambda}, and \eqref{TimEvPsiLambda} govern the time evolution of the Jost functions, while the evolution of the scattering coefficients is given by \eqref{TimEvScattCoeff}. Let $u : t \in (T^-, T^+) \mapsto u(t) \in H^1_+$ be a solution to \eqref{21.1}. Then  the Lax pair structure \eqref{LaxEq} implies that $\mathbb{L}_{u(t)}$ is unitarily equivalent to $\mathbb{L}_{u(0)}$, for any $t \in (T^-,  T^+)$. Consequently, the pure point spectrum $\sigma_{pp}(\mathbb{L}_u)$ is invariant under the flow of \eqref{21.1}. The time evolution of the standard eigenvectors   is given by the following theorem.

\begin{theorem}
    For any $T^- < 0  $ and $T^+>0$, let $u\in C^0((T^-, T^+),H^1_+)$ be the unique solution to   equation \eqref{21.1} with initial datum $u(0)=u_0\in H^1_+$.  Then for any    $j\in \{1,2, \cdots, N\}$,   the standard eigenvector $\varphi_j^u : t \in (T^-, T^+) \mapsto \varphi_j^{u(t)}$ given by \eqref{defEigenVec} satisfies the following equation: 
\begin{equation}
     \partial_t \varphi_j^u = i \partial_x^2 \varphi_j^u + 2 u \partial_x \Pi (\overline{u}\varphi_j^u ),\quad \forall t\in (T^-,T^+).
\end{equation}
\end{theorem}

Recall that $H^s_+$ is defined in \eqref{HardySobolevHs+}, $\forall s \geq 0$. The smooth Hardy-Sobolev space is denoted by $H^{\infty}_+: = \cap_{s \geq 0} H^s_+$. 
Theorem \ref{ThmTimEvoJostScatt} requires the $H^3_+$-solution $u(t)$ to remain in the weighted Sobolev space $H^2_s$ (with $s>\frac12$) for all time. A natural question is whether this condition can be weakened. Ideally, we would like to find a weighted Sobolev space $H^2_s$   that is invariant under the $H^{\infty}_+$-flow of \eqref{21.1}; that is, if the initial datum $u_0=u(0) $ belongs to $H^\infty_+ \cap H^2_s$, then the corresponding solution $u(t)$ stays in this space for all $t \in (T^-, T^+)$. The following theorem shows that such an invariant space exists, namely $H^2_1=\{f \in H^2 : f , \; \partial_x f, \;  \partial_x^2 f \in L^2_1\}$ (the case $s=1$). 
Indeed, we can prove that $H^4_+ \cap H^2_1$ is invariant under the flow of equation \eqref{21.1}.

\begin{theorem}\label{ThmInvariantSpace}
Let \( u \in C^0((T^-, T^+); H_+^4) \)  be the unique solution to the 
CMDNLS equation \eqref{21.1}
with initial datum \( u(0) = u_0 \in H_+^4 \). 
If \( u_0, \partial_x u_0, \partial_x^2 u_0 \in L^2_1 \), then 
we have 
\begin{equation}
u(t), \partial_x u(t), \partial_x^2 u(t) \in L_1^2,\quad \forall t\in (T^-, T^+).
\end{equation}
\end{theorem}

The  following result is a direct consequence of Theorem \ref{ThmTimEvoJostScatt} and \ref{ThmInvariantSpace}.

\begin{corollary}
    For any 
    $u_0 \in H^2_1 \cap H^4_+$, let $u \in C^0 \left( (T^-, T^+); H^4_+ \right) $ be the unique solution to   equation \eqref{21.1} with initial datum $u(0)=u_0$.     Then identities \eqref{TimEvphik}, \eqref{TimEvphilambda}, \eqref{TimEvPsiLambda} and  \eqref{TimEvScattCoeff} hold.
\end{corollary}

\subsection{Comparison with  DST of the Benjamin-Ono equation}
The direct and inverse scattering transform for the Benjamin-Ono (BO) equation has been  investigated by \cite{12, 13, 14, 15, 16}. 
In the present paper, we rigorously  establish  a direct scattering framework for the CMDNLS equation \eqref{21.1}.
This paper differs fundamentally from the direct scattering framework of the BO equation in several key aspects.

\begin{enumerate}

  \item 
    The spectrum of the Lax operator for the BO equation differs significantly from that of $\mathbb{L}_u$ given by \eqref{21.4}. The Lax operator for the BO equation does not allow embedded eigenvalues, meaning that any eigenvalue must be  discrete, whereas embedded eigenvalues are permitted for the Lax operator $\mathbb{L}_u$ of the CMDNLS equation \eqref{21.1}.

\item The construction of Jost functions for the CMDNLS equation \eqref{21.1} is fundamentally different from the BO case. In particular, imposing nonzero constant boundary conditions yields only trivial Jost functions, as shown in the following proposition.
\begin{proposition}\label{remark1}
Assume that $u\in H^1_+\cap L^2_s$ for some $s> \frac 12$. Then for every $k\in \C^+\cup \C^-\cup \big( (-\infty, 0)\setminus \sigma_{pp}(\mathbb{L}_u)\big)$, there exists a unique solution to
\begin{equation}\label{2.22}
\phi(x,k)=1+G_k*u\Pi( \overline u \phi),
\end{equation}
and this solution is given by $\phi \equiv 1$.
\end{proposition}
Hence the Jost function of \eqref{21.1} obtained in this manner is identically constant, and therefore insufficient for establishing a direct and inverse scattering transform that links the potential $u$ to its scattering data and spectral parameters. Instead, we must rely on the nontrivial Jost functions $\phi_k$, $\phi_{\lambda}^\pm$ and $\psi_\lambda^\pm$ constructed in Theorem \ref{1.4} and Theorem \ref{thm2}, which will be used to formulate the direct scattering transform.

 This marks the central difference from the direct scattering transform for the BO equation. In the case of the BO equation, the Jost functions also approach constant boundary values (typically normalized to $1$ as $|x|\to +\infty$), yet they yield nontrivial solutions that are effective in the direct scattering framework.

\item Compared with the  direct scattering framework of the BO equation, we provide a 
detailed proof for the time evolution of the Jost functions and the scattering coefficients. In particular, we exploit the unitary equivalence of the Lax operator to analyze the time evolution of the eigenvectors. And we prove the invariant subspace property, which enables us to weaken the conditions imposed on the solution $u$ in the time evolution argument.

\end{enumerate}

\subsection{Related work}

The CMDNLS equation \eqref{21.1} arises as the Hamiltonian equation associated with the energy functional
\begin{equation}\label{Hinvarphi}
{\mathcal{H}}=\tfrac{1}{2}\int_{\mathbb{R}} \big|\partial_x \varphi  +  \varphi\mathbb{H}(|\varphi|^2))\big|^2 \ \mathrm{d}x,
\end{equation}
which can be derived from the Calogero–Moser model \eqref{CMSys} in the thermodynamic limit via the work  \cite{1}. Under the gauge transformation 
\[
u(x):=\varphi(x)\exp\Bigl(\frac{i}{2}\int_{-\infty}^x |\varphi(y)|^2 \mathrm{d}y\Bigr),
\]
this Hamiltonian PDE reduces to the CMDNLS equation \eqref{21.1} on the real line. For further details on the physical derivation, we refer the reader to \cite{1}.

When the initial datum has sufficiently small $L^2$-mass, G\'erard and  Lenzmann proved that the $H^1_+$-solution to \eqref{21.1} will not blow up in finite time in \cite{2}.
More recently, Kim-Kim-Kwon   \cite{7} constructed the finite-time blow-up solutions with mass arbitrarily close to ground state and initial datum belongs to Schwartz space. Kim-Kwon   \cite{8} solved  the soliton resolution conjecture for both finite-time blow-up solutions and global solutions.

 The defocusing version of  equation \eqref{21.1}  is given by  \begin{equation}\label{21.2}
	 i\partial_tu+\partial_x^2u-u(\mathbb{H}-i)\partial_x(|u|^2)=0, \quad (t,x)\in \R\times \R,
	 \end{equation}
 which traces back to the seminal work of Pelinovsky and Grimshaw   \cite{9,10}.
   Equation \eqref{21.2} is the deep-water limit of the intermediate nonlinear Schr\"{o}dinger (INLS) equation, introduced in Pelinovsky  \cite{9}, 
 \begin{equation}\label{iNLS}
 	i\partial_tu-\partial_x^2u-u(i+T_\delta)\partial_x(|u|^2)=0,
 	\end{equation}
 where the integral operator $T_\delta$ is given by 
 \[
 T_\delta(f)(x):=\mathrm{P.V.} \int_{\R} \coth\left(  \frac{\pi}{2\delta} (y-x)
 \right) f(y) \frac{         \ \mathrm{d} y}{2\delta},        \quad   \forall \delta>0.
 \]

 In Killip-Laurens-Vi\c{s}an \cite{6}, the flow map  was continuously extended from the high-regularity Hardy-Sobolev space $H^{\infty}_+ \cap L^2_1$ to the space $L^2_+$. 
  According to Proposition 2.3 of \cite{6}, the space $H^{\infty}_+ \cap L^2_1$ is invariant under the flow of \eqref{21.1}.
Theorem $\ref{ThmInvariantSpace}$ of this paper shows that the space $H^4_+ \cap H^2_1$ is also invariant under the flow of \eqref{21.1}, 
which follows from the unitary equivalence of the Lax operator $\mathbb{L}_u$ along the flow. Thus, while \cite{6} establishes the invariance of the weighted $L^2$-norm of the solution itself, our result further guarantees that the weighted $L^2$-norms of its first and second derivatives remain bounded as well, providing a refined regularity persistence property not previously established.

Frank and Read first derived the Jost solutions and scattering theory  for the CMDNLS equation under weaker assumptions  \cite{Frank}, and established the corresponding trace formula. Our work adopts a distinct technical route:
  we specifically explore the invariant subspace and the detailed time evolution of the eigenvectors and Jost functions. We employ different techniques, such as unitary equivalence arguments rather than direct differentiation with respect to time, to analyze the relations between  scattering coefficients and their asymptotic behaviors.

 In the periodic setting, Badreddine \cite{4} derived an explicit formula for general solutions and extended the flow map to low regularity Sobolev spaces for both the focusing and defocusing equations.  
   In
addition, the traveling wave and finite-gap potentials are studied in \cite{5}.
Sun   studied the two-variable extension for \eqref{21.1} and \eqref{21.2} in   \cite{11}.

\subsection{Notation}
We introduce the following notation.
\begin{enumerate}
\item Let $\langle x \rangle=\sqrt{1+x^2}$ for $x\in\R$. 

\item The upper and lower   half-complex planes are denoted by $\mathbb{C}^+$ and $\mathbb{C}^-$, i.e., $$\mathbb{C}^+=\{a+b i \in \mathbb{C}: b >0\} ,\quad \mathbb{C}^-=\{a+b i \in \mathbb{C}: b <0\}.$$

\item The operator $D$ is given by $ D=-i\partial_x$.

\item

For all  $f,g \in H^s$ with $s > \frac{1}{2}$, there exists $ C_s' > 0$ such that 
\begin{equation}\label{inequalityCspri}
    \|fg\|_{H^s} \leq C_s' \|f\|_{H^s} \|g\|_{H^s} .
\end{equation}
Due to 
 $H^1 \hookrightarrow L^\infty$, we define $\left(C_1''\right)^{-1} := \inf_{f \in H^1 \setminus \{0\}} \frac{\|f\|_{H^1}}{\|f\|_{L^\infty}}$ such that 
 \begin{equation}\label{ineqC1pripri}
     \|f\|_{L^\infty} \leq C_1'' \|f\|_{H^1}.
 \end{equation}
  
\item Let $\mathcal{B}(X)$ denote the space consists of  all bounded $\C$-linear transforms from $X$ to $X$.

\item For every real number $\mu$, define $e_\mathbf{\mu}(x):=e^{i\mathbf{\mu} x}$.

\item The $1$-dimension Schwartz space is denoted by $\mathcal{S}$, and the tempered distribution space is denoted by $\mathcal{S}'$.

\item $\mathcal{F} : \mathcal{S}' \to \mathcal{S}'$ and $\mathcal{F}^{-1}: \mathcal{S}' \to \mathcal{S}'$ denote Fourier transform and  the inverse Fourier transform.

\end{enumerate}

\subsection{Organization of the paper}
The paper is organized as follows.  In Sec.\ref{sec2},  we define the Jost functions and prove their existence and uniqueness. In Sec.\ref{sec3}, we construct the scattering coefficients, and show the relation of the Jost functions and the basic properties of scattering coefficients. In Sec.\ref{sec4}, we investigate the spectral asymptotic behavior near $k=0$ and $k=\infty$. 
In Sec.\ref{sec5}, we prove that certain subspaces are invariant under the flow of the CMDNLS equation \eqref{21.1}. Finally,  in Sec.\ref{sec6}, we explore the time evolution of the eigenfunctions, Jost functions and the scattering coefficents. 

\section{Jost functions}\label{sec2}
For any $s>0$,  the Hardy-Sobolev space $H^s_+$ contains all $H^s$-functions whose Fourier transforms are positively supported, i.e., 

\begin{equation}
   H^s_+  =H^s(\R; \C)\cap L^2_+=\Pi (H^s), \quad \forall s \geq 0.
\end{equation}
For any $v \in L^2$, the  Toeplitz operator of symbol $v$ is defined as follows
\begin{equation}
    \mathbb{T}_v(h):=\Pi(vh)\in L^2_+,\quad \forall h\in H^1_+,
\end{equation}where $\Pi=\frac{1}{2}(1 + i \mathbb{H}) : L^2 \to L^2$ denotes the Szeg\H{o} projector. If the symbol $b\in L^\infty$, the Toeplitz operator $\mathbb{T}_b$ is a bounded operator on $L^2_+$, whose $L^2_+$-adjoint is $\mathbb{T}_{\overline{b}}$. Then the Lax operator given by \eqref{21.4} can be expressed by
\begin{equation}
    \mathbb{L}_u = D - \mathbb{T}_u \mathbb{T}_{\overline{u}} :H^1_+ \to L^2_+, \quad \mathrm{with} \quad D = -i\partial_x.
\end{equation} 

Following the approach in  \cite{12,13,16}, we express the Jost functions as solutions to a convolution integral equation involving a Green function.  For any $x\in \mathbb{R}$, we define the Green functions as 
\begin{equation}\label{gk}
	G_k(x)
	:=\mathcal F^{-1}\left(\frac{\mathds{1}_{\xi\geq 0}}{\xi-k} \right)(x)
	=\frac{1}{2\pi}\int_0^{\infty} \frac{e^{ix\xi}}{\xi-k}d\xi,\quad k\in  \C\setminus [0,+\infty),
	\end{equation}
\begin{equation}\label{Gkt}
	\widetilde G_k(x):=\mathcal F^{-1}\left(\frac{\mathds{1}_{\xi\leq 0}}{\xi-k} \right)(x)=\frac{1}{2\pi}\int_{-\infty}^0 \frac{e^{ix\xi}}{\xi-k}d\xi, \quad k \in \C\setminus (-\infty, 0].
	\end{equation}
    Here $\mathcal{F}$ and $\mathcal{F}^{-1}$ denote Fourier transform and  the inverse Fourier transform. 
     Then we have 
    \begin{equation}\begin{split}\label{Gkc+c-}
&G_k(x)=ie^{ikx}\mathds{1}_{x\geq 0}-\widetilde G_k(x),\quad k\in\C^+,\\
    &G_k(x)=-ie^{ikx}\mathds{1}_{x\leq 0}-\widetilde G_k (x),\quad k\in\C^-.
    \end{split}\end{equation}
    For $\lambda>0$, by  Plancherel formula, we get 
    \begin{equation}\label{cvL2}
        \| \widetilde{G}_{\lambda+i\varepsilon}- \widetilde{G}_\lambda \|_{L^2}=\int_{-\infty}^0 \bigg|\frac{1}{\xi-(\lambda+i\varepsilon)}-\frac{1}{\xi-\lambda}\bigg|^2 \d \xi\to 0, \quad  \varepsilon\to 0^+.
    \end{equation}
    Then  define $G_\lambda^+$ as  $G_{\lambda+i\varepsilon}$ when $\varepsilon\to 0^+$,
    \begin{equation}\label{glambda}
        G_\lambda^+ (x):=\lim_{\varepsilon\to 0^+} G_{\lambda + i \varepsilon} (x)= ie^{i\lambda x}\mathds{1}_{x\geq 0}-\widetilde G_\lambda (x),
    \end{equation}
    where the first  term converges pointwisely, and the second term converges in $L^2$.
    Similarly, define $G_\lambda^-$ as 
    \begin{equation}\label{glambda-}
        G_\lambda^-(x):=\lim_{\varepsilon\to 0^+} G_{\lambda - i \varepsilon}(x)= -ie^{i\lambda x}\mathds{1}_{x\leq 0}-\widetilde G_\lambda(x).
    \end{equation}

  Using the Green functions, we will define  the convolution operators $T_k$ and $T_\lambda^\pm$. Applying the Fredholm alternative theorem, the  bijectivity of $\mathrm{id}-T_k$ and $\mathrm{id}-T_\lambda^\pm$ will be proved. Then we define the Jost functions by  the bijective  operator. 
  
    Assume that  $u\in H^1_+\cap L^2_s $ for some $s> 0$. For every $k\in \C\setminus [0, +\infty)$ and $\lambda>0$, define the  operators $T_k: L_{-s}^\infty \to L^\infty_{-s} $  and  $T_\lambda^\pm: L_{-s}^\infty \to L^\infty_{-s} $  as 
\begin{align}
	T_k\mathbf{f}_k:=G_k*(u\Pi( \overline u \mathbf{f}_k)), \quad T_\lambda^\pm \mathbf{f}_\lambda^\pm:=G_\lambda^\pm*(u\Pi( \overline u \mathbf{f}_\lambda^\pm)),\quad \forall \  \mathbf{f}_k, \mathbf{f}_\lambda^\pm\in L^{\infty}_{-s}.\label{Tlamf}
	\end{align}
    Since $G_k\in L^2$, the  operators $T_k$ and $T_\lambda^\pm$ are also well-defined  from $L^\infty$ to $L^\infty$.

   \begin{lemma}\label{lemma2.2}
Assume that $u\in H^1_+\cap L^2_s $ for some $s > 0$. Then for every $\lambda>0$, the   operators $T_\lambda^\pm$ are compact from $L_{-s}^{\infty} $ to $L^{\infty} $  .
	\end{lemma} 

        \begin{proof}
Let $(f_n)_{n\in \mathbb{N}}$ be a bounded sequence in  $L_{-s}^\infty$ and  $A=\sup\limits_{n\in \mathbb{N}}\|  f_n\|_{L_{-s}^\infty}<+\infty$.
Since $u\in H^1_+\cap L_s^2$, by Young's inequality, we deduce
\begin{equation}\begin{split}\label{Tlaminf}
\| T_\lambda^+ f_n\|_{L^\infty}&=\|G_\lambda^+*(u\Pi(\overline u f_n))\|_{L^\infty}=\|(ie_\lambda\mathds{1}_{x\geq 0})*(u\Pi(\overline u f_n))\|_{L^\infty}+\|\widetilde{G}_\lambda*(u\Pi(\overline u f_n))\|_{L^\infty}\\
&\leq \|ie_\lambda\mathds{1}_{x\geq 0}\|_{L^\infty}\|u\Pi(\overline u f_n)\|_{L^1}+\|\widetilde{G}_\lambda\|_{L^2}\|u\Pi(\overline u f_n)\|_{L^2}\\
&\leq  \|u\|_{L^2}\| u\|_{L^2_s}\| f_n\|_{L_{-s}^\infty}+\frac 1 {\sqrt{\lambda}} \|u\|_{L^\infty}\| u\|_{L^2_s}\| f_n\|_{L_{-s}^\infty}\leq A B(\lambda, \| u\|_{ H^{1}_s}),
\end{split}\end{equation}
where $B(\lambda, \| u\|_{H^{1}_s})$ depends on $\lambda$ and $\| u\|_{H^{1}_s}$, and  $B(\lambda, \| u\|_{H^{1}_s})=\|u\|_{L^2}\| u\|_{L^2_s}+\frac 1 {\sqrt{\lambda}} \|u\|_{L^\infty}\| u\|_{L^2_s}$.
Next we compute the derivative $D(T_\lambda^+f_n)$.
Taking the derivative with respect to $x$ for $ T_\lambda^+f_n$, we have
\[
D(T_\lambda^+f_n)(x)=i\lambda e^{i\lambda x}\int_{-\infty}^{x}e^{-i\lambda y}u(y)\Pi( \overline{u}f_n)(y)         \ \mathrm{d} y+u(x)\Pi( \overline{u}f_{n})(x)-\frac{1}{i}\partial_{x} \left(\widetilde{G}_{\lambda}*(u\Pi( \overline{u}f_n)) \right)(x).
\]
Since
\[
\begin{aligned}
	&\mathcal{F}\left(\frac{1}{i}\partial_{x}\left(\widetilde{G}_{\lambda}*(u\Pi( \overline{u}f_n))\right)\right)(\xi)= \xi\frac{\mathds{1}_{\xi\leq 0}}{\xi-\lambda}\widehat{u\Pi(\overline uf_n)}(\xi)\\
	&=\left(\mathds{1}_{\xi\leq 0}+\lambda\frac{\mathds{1}_{\xi\leq 0}}{\xi-\lambda}\right)\widehat{u\Pi(\overline uf_n)}(\xi)=\mathcal{F}\left(\Pi_{-}(u\Pi( \overline{u}f_n))+\lambda\widetilde{G}_{\lambda}*(u\Pi( \overline{u}f_n))\right)(\xi),
\end{aligned}
\]
with $\Pi_-=\mathrm{id}-\Pi$,
we obtain
\[
\begin{aligned}
D(T_\lambda^+f_n)&=i\lambda e^{i\lambda x}\int_{-\infty}^{x}e^{-i\lambda y}u(y)\Pi( \overline{u}f_n)(y,\lambda+)         \ \mathrm{d} y+u\Pi( \overline{u}f_n)-\Pi_{-}(u\Pi( \overline{u}f_n))-\lambda\widetilde{G}_{\lambda}*(u\Pi( \overline{u}f_n))
\\
&
=I_1+I_2+I_3+I_4.
\end{aligned}
\]
From $u\in H^1_+\cap  L_{s}^2$, we have $I_2\in L^2\cap L^1$ and  $I_3=-\Pi_-(I_2)\in L^2_- $. By Young's inequality, we get
\[
\| I_4\|_{L^{\infty}}\leq \lambda \|\widetilde G_\lambda \|_{L^2} \|u\Pi(\overline u f_n) \|_{L^2}\leq \frac{\sqrt{\lambda}}{\sqrt{2\pi}}\|u\|_{L^\infty}\| u\|_{L^2_s}\| f_n\|_{L_{-s}^\infty}\leq A \boldsymbol{C}_1(\lambda,\| u\|_{H^{1}_s}),
\]
where $\boldsymbol{C}_1(\lambda,\| u\|_{H^{1}_s})$ depends  on $\lambda$ and $\| u\|_{H^{1}_s}$.
Using H\"older inequality, we have
\[
| I_1 |\leq \int_{-\infty}^x |u\Pi(\overline u f_n)| \mathrm{d} y \leq \|u\Pi(\overline u f_n) \|_{L^1}\leq \|u\|_{L^2}\| u\|_{L^2_s}\| f_n\|_{L_{-s}^\infty}<A \boldsymbol{C}_2(\|u \|_{H^{1}_s}),
\]
where $\boldsymbol{C}_2(\|u \|_{H^{1}_s})$  depends on $\|u \|_{H^{1}_s}$.
Thus for every $N>0$, there exists  a constant $C_N>0$, depending on  $\lambda$ and $\|u \|_{H^{1}_s}$, such that $D(T_\lambda^+f_n)$ belongs to  a locally $L^2$ space, i.e
\begin{equation}\label{DTf}
\| DT_\lambda^+f_n\|_{L^2(-N,N)}\leq A C_N(\lambda, \|u \|_{H^{1}_s}).
\end{equation}
Therefore, by \eqref{Tlaminf} and \eqref{DTf}, 
 the sequence $( T_\lambda^+ f_n  )_{n\in \mathbb{N}}$ belongs to the Sobolev space $W^{1,2}[-N, N]$.  By the Rellich--Kondrachov theorem,   the imbedding from the Sobolev space $W^{1,2}[-N, N]$ to the  H\"older space \(C^{0,\frac{1}{2}}[-N,N]\) is compact.

Next we prove that there exists a subsequence converging  to a function \(f\) in $L^{\infty}$. 
The proof proceeds via a Cantor diagonal process with the following steps:  

For the interval $[-1,1]$,  $(T_\lambda^+f_{n})_{n\in \mathbb{N}}\in C^{0,\frac{1}{2}}[-1,1]$ is equicontinuous on $[-1,1]$, and $(T_\lambda^+f_{n}(x))_{n\in \mathbb{N}}$ is uniformly bounded on $[-1,1]$. Then according to Arzela-Ascoli theorem, there exists a increasing function $\kappa_1: \mathbb{N}\to  \mathbb{N}$ such that the subsequence $(T_\lambda^+f_{\kappa_1(n)})_{n\in \mathbb{N}}$ converges uniformly on $[-1,1]$. 

 For the interval $[-2,2]$, 
$(T_\lambda^+f_{\kappa_1(n)})_{n\in \mathbb{N}}$ is unifomly bounded in $W^{1,2}[-2,2]$. 
 There exists a increasing function $\kappa_2: \mathbb{N}\to  \mathbb{N}$ such that the subsequence $(T_\lambda^+f_{\kappa_1\circ \kappa_2(n)})_{n\in \mathbb{N}}$ converges uniformly on $[-2,2]$. 

For the interval $[-m-1,m+1]$, the subsequence $(T_\lambda^+f_{\kappa_1\circ \kappa_2\circ\cdots \circ \kappa_m(n)})_{n\in \mathbb{N}}$ is unifomly bounded in $W^{1,2}[-m-1,m+1]$. Thus there exists an increasing function $\kappa_{m+1}: \mathbb{N}\to  \mathbb{N}$ such that the subsequence $(T_\lambda^+f_{\kappa_1\circ \kappa_2\circ\cdots \circ \kappa_{m+1}(n)})_{n\in \mathbb{N}}$ converges uniformly on  $[-m-1,m+1]$. 

Define the diagonal subsequence $h_n:=T_\lambda^+f_{\kappa_1\circ \kappa_2\circ\cdots \circ \kappa_{n}(n)}$ for $n\geq 1$. For every $N\in \mathbb{N}_+$, there exists a function $f\in C^{0,\frac 1 2}$ such that $(h_n)_{n\geq 1}$ converges to $f$ uniformly on $[-N, N]$. 

Every compact subset of \(\mathbb{R}\) is contained in some closed interval \([-N, N]\) for \(N \in \mathbb{N}_+\). Thus $(h_n)_{n\geq 1}$ converges to $f$ uniformly on every compact subset of \(\mathbb{R}\). Since the uniform limit of sequence of continuous function is continuous, $f$ is continuous on $\R$. 

 For every $x\in\R$, there exists $N_x=[x]+1$ such that 
 \[
 |f(x)|\leq \| f\|_{L^{\infty}[-N_x,N_x]} \leq \sup\limits_{n\in\mathbb{N}_+} \|T_\lambda^+f_{\kappa_1\circ \kappa_2\circ\cdots \circ \kappa_{n}(n)} \|_{L^{\infty}} \leq \sup\limits_{n\in\mathbb{N}_+} \|T_\lambda^+f_n \|_{L^{\infty}}.
 \]
From \eqref{Tlaminf}, we obtain
\[
\begin{aligned}
	\|f\|_{L^{\infty}}\leq \sup_{n\in\mathbb{N}_+}\|T_{\lambda}^+f_n\|_{L^{\infty}}\leq  A B(\lambda, \| u\|_{ H^{1,s}}).
\end{aligned}
\]
Thus $T_\lambda^+: L^{\infty}_{-s}\to L^{\infty}$ is compact.

In a similar way, we prove that the   operator $T_\lambda^-$  is compact from $L_{-s}^\infty$ to $L^\infty$.
\end{proof}

\begin{corollary}\label{Tlamb-compact}
 Assume that $u\in H^1_+\cap L^2_s $ for some $s > 0$. Then for every $\lambda>0$, the operators $T_\lambda^\pm$ are compact from $L_{-s}^{\infty} $ to $L_{-s}^{\infty} $.
\end{corollary}
\begin{proof}
     Since the  injection $L^{\infty}\to  L^{\infty}_{-s}$ is continuous,  $T_\lambda^+: L^{\infty}_{-s}\to L^{\infty}_{-s}$ is compact. 
\end{proof}
   \begin{lemma}\label{T-lam-inf-inf}
Assume that $u\in H^1_+\cap L^2 $. Then for every $\lambda>0$, the   operators $T_\lambda^\pm$ are compact from $L^{\infty} $ to $L^{\infty} $  . 
	\end{lemma} 
    \begin{proof}
Let $(\mathfrak{f}_n)_{n\in \mathbb{N}}$ be a bounded sequence in $L^\infty$ and $\mathfrak{A}=\sup_{n\in \mathbb{N}}\| \mathfrak{f}_n \|<+\infty$. 
Since $u\in H^1_+\cap L^2$, by Young's inequality, we deduce
\begin{equation*}\begin{split}
\| T_\lambda^+ \mathfrak{f}_n\|_{L^\infty}&=\|G_\lambda^+*(u\Pi(\overline u \mathfrak{f}_n))\|_{L^\infty}=\|(ie_\lambda\mathds{1}_{x\geq 0})*(u\Pi(\overline u \mathfrak{f}_n))\|_{L^\infty}+\|\widetilde{G}_\lambda*(u\Pi(\overline u \mathfrak{f}_n))\|_{L^\infty}\\
&\leq \|ie_\lambda\mathds{1}_{x\geq 0}\|_{L^\infty}\|u\Pi(\overline u \mathfrak{f}_n)\|_{L^1}+\|\widetilde{G}_\lambda\|_{L^2}\|u\Pi(\overline u \mathfrak{f}_n)\|_{L^2}\\
&\leq  \|u\|_{L^2}\| u\|_{L^2}\| \mathfrak{f}_n\|_{L^\infty}+\frac 1 {\sqrt{\lambda}} \|u\|_{L^\infty}\| u\|_{L^2}\| \mathfrak{f}_n\|_{L^\infty}\leq \mathfrak{A} \mathfrak{B}(\lambda, \| u\|_{ H^{1}}),
\end{split}\end{equation*}
where $\mathfrak{B}(\lambda, \| u\|_{H^{1}})$ depends on $\lambda$ and $\| u\|_{H^{1}}$, and  $\mathfrak{B}(\lambda, \| u\|_{H^{s}})=\|u\|_{L^2}^2+\frac 1 {\sqrt{\lambda}} \|u\|_{L^\infty}\| u\|_{L^2}$.
Similar as Lemma \ref{lemma2.2}, we have 
\[
\begin{aligned}
D(T_\lambda^+ \mathfrak{f}_n)&=i\lambda e^{i\lambda x}\int_{-\infty}^{x}e^{-i\lambda y}u(y)\Pi( \overline{u}\mathfrak{f}_n)(y,\lambda+)         \ \mathrm{d} y+u\Pi( \overline{u}\mathfrak{f}_n)-\Pi_{-}(u\Pi( \overline{u}\mathfrak{f}_n))-\lambda\widetilde{G}_{\lambda}*(u\Pi( \overline{u}\mathfrak{f}_n)).
\end{aligned}
\]
Since $u\in H^1_+\cap  L^2$, we have $u\Pi( \overline{u}\mathfrak{f}_n)\in L^2\cap L^1$ and $\Pi_{-}(u\Pi( \overline{u}\mathfrak{f}_n))\in L^2_- $. By H\"{o}lder inequality and  Young's inequality, the first and second term belong to $L^\infty$.
Thus $D(T_\lambda^+\mathfrak{f}_n)$ belongs to  a locally $L^2$ space.  Then using the Sobolev imbedding theorem and Cantor diangonal process, we obtain the compactness of the operator $T_\lambda^+: L^\infty\to L^\infty$.
    \end{proof}
    
    For $\lambda>0$, the operators $T_\lambda^\pm: L^\infty_{-s}\to L^\infty_{-s}$ and $T_\lambda^\pm: L^\infty\to L^\infty$ are compact by Corollary \ref{Tlamb-compact} and  Lemma \ref{T-lam-inf-inf}. Now for $k\in\C\setminus [0,+\infty)$, we prove the compactness for $T_k:  L^\infty_{-s}\to L^\infty_{-s}$ and $T_k:  L^\infty\to L^\infty$.
 \begin{lemma}\label{Tkcomp}
 Assume that $u\in H^1_+\cap L^2_s $ for some $s > 0$. 	Then for every $k\in\C\setminus[0,+\infty)$,  the operator $T_k$ is compact from $L_{-s}^{\infty} $ to $L_{-s}^{\infty} $. In addtion, the operator $T_k$ is compact from $L^{\infty} $ to $L^{\infty} $.
	\end{lemma} 

    \begin{proof}
From \eqref{Tlamf}, we have
\[
T_kf_n= G_k*(u\Pi(\overline u f_n)).
\]
Since $G
_k\in L^2$ and $u\in H^1_+\cap L_s^2$, using Young's inequality, we deduce
\begin{equation}\small\begin{split}
&
\| T_k f_n\|_{L^\infty}=\|{G}_k*(u\Pi(\overline u f_n))\|_{L^\infty}
\leq \|{G}_k\|_{L^2}\|u\Pi(\overline u f_n)\|_{L^2}
\leq \boldsymbol{C}_3(k)\|u\|_{L^\infty}\| u\|_{L^2_s}\| f_n\|_{L_{-s}^\infty},
\\
&
\| T_k f_n\|_{L^\infty}=\|{G}_k*(u\Pi(\overline u f_n))\|_{L^\infty}
\leq \|{G}_k\|_{L^2}\|u\Pi(\overline u f_n)\|_{L^2}
\leq \boldsymbol{C}_4(k)\|u\|_{L^\infty}\| u\|_{L^2}\| f_n\|_{L^\infty},
\end{split}\end{equation}
where $\boldsymbol{C}_3(k)$ and $\boldsymbol{C}_4(k)$ depend on $k$.
Take the derivative with respect to $x$ for $ T_kf_n$ and use the Fourier transform,  we have
\[
D(T_kf_n)=\frac{1}{i}\partial_{x} \left({G}_k*(u\Pi( \overline{u}f_n)) \right)=u\Pi(\overline u f_n)+kG_k*(u\Pi(\overline u f_n)).
\]
Then similar as the imbedding and Cantor diagonal process in  the proof of Lemma \ref{lemma2.2}, we prove that $T_k$ is compact from $L_{-s}^{\infty} $ to $L_{-s}^{\infty} $. In addition, we obtain the compactness from $L^\infty$ to $L^\infty$.
    \end{proof}
        Next we prove that the operators $\mathrm{id}-T_k$ and $\mathrm{id}-T_\lambda^\pm$ are both injective for $L_{-s}^\infty \to L_{-s}^\infty$ and $L^\infty \to L^\infty$.
        \begin{lemma}\label{lemma2.3}
Assume that $u\in H^1_+\cap L^2_s $ for some $s> 0$.
	Then for $k\in   \C^+\cup \C^-  \cup \big( (-\infty,0)\setminus \sigma_{pp}(\mathbb{L}_u)\big)$, the operator $\mathrm{id}-T_k: L_{-s}^\infty \to L_{-s}^\infty$ is injective. In addition, the operator $\mathrm{id}-T_k: L^\infty \to L^\infty$ is injective.
	\end{lemma}

\begin{proof}
Let $g\in L_{-s}^\infty$ and $(\mathrm{id}-T_k)g=0$. Then $g=G_k*(u\Pi(\overline{u}g))$.	Since $G_k\in L^2$, $u\in L^2\cap L^2_s $ and $g\in L^\infty_{-s}$, by H\"{o}lder inequality and Young's inequality, we have $g\in L^2$. If $g\in L^\infty$, by H\"{o}lder inequality and Young's inequality, we also obtain $g\in L^2$. The Fourier transform yields 
\begin{equation}\label{gFour}
(\xi - k)\hat{g}=\mathds{1}_{\xi\geq 0} \widehat{u\Pi(\overline{u}g)}.
\end{equation}
 Using  $u\Pi(\overline u  g)\in L^2$, we have
$
    \int_\R |\xi-k|^2|\hat  g(\xi)|^2 d\xi =  \int_0^{+\infty}   |\widehat {u\Pi(\overline u  g)}|^2 d\xi< +\infty.
$
Let $k=a+ib$, then we get 
$
	\int_\R ((\xi-a)^2+b^2)  |\hat  g(\xi)|^2 d\xi< +\infty.
$
So we deduce that
\begin{equation}
\small
	\begin{split}
	\int_{0}^{+\infty} \vert \xi \hat{ g}(\xi) \vert^2 d\xi &\leq 2 \int_{0}^{+\infty} \left( (\xi - a)^2 + a^2 \right) \vert \hat{ g}(\xi) \vert^2 d\xi 
	\leq 2 \int_{0}^{+\infty} \vert \xi - k \vert^2 \vert \hat{ g}(\xi) \vert^2 d\xi + 2a^2 \cdot 2\pi \Vert  g \Vert_{L^2}^2 < +\infty.
	\end{split}
\end{equation}
Therefore we get $g\in H^1_+$.
Then  the inverse Fourier transform for \eqref{gFour} gives
	\begin{equation}\label{1.13}
		\mathbb{L}_ug=-i\partial_x g - u\Pi(\overline u g) =kg.
		\end{equation}
Since $g\in H^1_+= Dom (\mathbb{L}_u)$ and $k$ is not the eigenvalue, we obtain 
 $g(k)=0$ for $k\in   \C^+\cup \C^-  \cup \big( (-\infty,0)\setminus \sigma_{pp}(\mathbb{L}_u)\big)$.
\end{proof}   

    \begin{lemma}\label{Tlaminj}
Assume that $u\in H^1_+\cap L^2_s $ for some $s> \frac 12$.
	Then  for $\lambda\in (0,+\infty)\setminus  \sigma_{pp}(\mathbb{L}_u)$, the  operators $\mathrm{id}-T_\lambda^\pm: L_{-s}^\infty \to L_{-s}^\infty$ are injective. In addition, the  operators $\mathrm{id}-T_\lambda^\pm: L^\infty \to L^\infty$ are injective.
	\end{lemma}

    \begin{proof}
Let $h\in L_{-s}^\infty$ and $(T_\lambda^+-\mathrm{id})h=0$. Then $h=T_\lambda^+h=(ie^{i\lambda x}\mathds{1}_{\R^+})* (u\Pi(\overline u h))-\widetilde{G}_\lambda*(u\Pi(\overline u h))$. Since $\widetilde{G}_\lambda\in L^2_-$ and $u\Pi(\overline u h)\in L^2_+$, we obtain $\widetilde{G}_\lambda*(u\Pi(\overline u h))=0$. Then \begin{equation}\label{h}
h=ie^{i\lambda x} \int_{-\infty}^x e^{-i\lambda y}u(y)\Pi(\overline u h)(y)\d y,
\end{equation}
and $h\in L^\infty$. 
 For every $x<0$ and $y<x$, we have $|y|>|x|$. Utilizing $u\in L^2_s$ and $s>\frac 12$, we deduce that 
\[
|h(x)|\leq \int_{-\infty}^x |u\Pi(\overline u h)(y)|\langle y\rangle^{s} \langle y\rangle^{-s}\d y\leq \langle x\rangle^{-s}\| u\|_{L^2_s}\|\Pi(\overline u h) \|_{L^2}.
\]
Then we obtain $h\in L^2(-\infty,0)$. 

Set $F=u\Pi(\overline u h)\in L^1\cap L^2_+$. 
According to Lemma 4.3 in  \cite{16}, 
 we have
 \begin{equation}\label{Glambda^+*F,F}
     \begin{split}
         \langle h, u\Pi(\overline u h)\rangle_{L^\infty, L^1}
         &
         =\langle G_\lambda^+*(u\Pi(\overline u h)), u\Pi(\overline u h) \rangle_{L^\infty, L^1}=\langle G_\lambda^+*F, F\rangle_{L^\infty, L^1}
         \\
         &
         =i|\hat F(\lambda)|^2+\langle F, G_\lambda^+ *F\rangle_{L^1, L^\infty}=i|\hat F(\lambda)|^2+\langle u\Pi(\overline u h), h\rangle_{L^1, L^\infty}.
     \end{split}
 \end{equation}
Since $\langle h, u\Pi(\overline u h)\rangle_{L^\infty, L^1}=\| \Pi(\overline u h) \|_{L^2}^2$, $\langle h, u\Pi(\overline u h)\rangle_{L^\infty, L^1}$ is real. So we obtain 
\begin{equation}\label{fourier=0}
    \hat{F}(\lambda)=\int_\R u\Pi(\overline u h)e^{-i\lambda x} \d x=0.
\end{equation}
Then for $x>0$, we use $h(x)=i\int_{-\infty}^x e^{i\lambda(x-y)} u\Pi(\overline u h)\d y=-i\int_x^{+\infty}e^{i\lambda(x-y)} u\Pi(\overline u h)\d y$. By a similar manipulation as $x<0$, we obtain  $h\in L^2(0,
+\infty)$. As a consequence, $h\in L^2(\R)$.
By \eqref{h}, we have 
\begin{equation}
    D h=-i\partial_x h=i\lambda e^{i\lambda x} \int_{-\infty}^x e^{-i\lambda y}   u\Pi(\overline u h) (y) \ \mathrm{d} y+u\Pi(\overline u h)=\lambda h +u\Pi(\overline u h).
\end{equation}
That is 
\begin{equation}\label{L_uh}
    \mathbb{L}_uh=\lambda h.
\end{equation}
So  we have $h\in H^{1}_+ = Dom (\mathbb{L}_u)$ from \eqref{L_uh}. If $\lambda\in (0,+\infty)\setminus  \sigma_{pp}(\mathbb{L}_u)$,  then $h=0$.

In addition, if $h\in L^\infty$, we have $h=T_\lambda^+h=(ie^{i\lambda x}\mathds{1}_{\R^+})* (u\Pi(\overline u h))-\widetilde{G}_\lambda*(u\Pi(\overline u h))$. Since $\widetilde{G}_\lambda\in L^2_-$ and $u\Pi(\overline u h)\in L^2_+$, we also obtain $\widetilde{G}_\lambda*(u\Pi(\overline u h))=0$ and $
h=ie^{i\lambda x} \int_{-\infty}^x e^{-i\lambda y}u(y)\Pi(\overline u h)(y)\d y$. Since $u\in L^2_s$ and $s>\frac 12$, we get $h\in L^2$ and further $h\in H^1_+$. Finally $h=0$.
    \end{proof}

    \begin{corollary}\label{cor3.6}
        Assume that $u\in H^1_+\cap L^2_s$ for some $s> \frac 12$. For every $k\in \C^+\cup \C^- \cup \big((-\infty,0)\setminus \sigma_{pp}(\mathbb{L}_u)\big)$ and every $\lambda\in (0,+\infty)\setminus \sigma_{pp}(\mathbb{L}_u) $, the operators $\mathrm{id}-T_k: L_{-s}^\infty \to L_{-s}^\infty$ and $\mathrm{id}-T_\lambda^\pm: L_{-s}^\infty \to L_{-s}^\infty$ are bijective. In addition, the operators $\mathrm{id}-T_k: L^\infty \to L^\infty$ and $\mathrm{id}-T_\lambda^\pm: L^\infty \to L^\infty$ are bijective.
    \end{corollary}
    \begin{proof}
        By Fredholm alternative theorem and Lemma \ref{lemma2.2}-\ref{Tlaminj}, we obtain the bijectivity of $\mathrm{id}-T_k$ and $\mathrm{id}-T_\lambda^\pm$.
    \end{proof}
Now we define Jost functions $\phi_k$, $\phi_\lambda^\pm$ and $\psi_\lambda^\pm$ as follows:
       \begin{definition}\label{def3.7}
            Assume that $u\in H^1_+\cap L^2_s$ for some $s> \frac 12$. 
For every $k\in \C^+\cup \C^-\cup  \big( (-\infty, 0)\setminus \sigma_{pp}(\mathbb{L}_u)\big)$, we define 
\begin{equation}
\phi_k=(\mathrm{id}-T_k)^{-1}(G_k*u),
\end{equation}
i.e.,
\begin{equation}\label{1.5}
    \phi_k=G_k*\big(u+u\Pi(\overline u \phi_k)\big).
\end{equation}
       \end{definition}
\begin{lemma}\label{phikequi}
       Assume that $u\in H^1_+\cap L^2_s$ for some $s> \frac 12$. For every $k\in \C^+\cup \C^-\cup  \big( (-\infty, 0)\setminus \sigma_{pp}(\mathbb{L}_u)\big)$, 
     \eqref{1.5} is equivalent to 
    \begin{equation}\label{1.1}
\mathbb{L}_u\phi_k=k\phi_k+u,
    \end{equation}
    for $\phi_k\in L^\infty_{-s}$.
  In addition, if either \eqref{1.5} or \eqref{1.1} holds, we deduce $\phi_k\in H^1_+$.
\end{lemma}

\begin{proof}

For every $k\in \C^+\cup \C^-\cup  \big( (-\infty, 0)\setminus \sigma_{pp}(\mathbb{L}_u)\big)$, 
	taking the Fourier transform for (\ref{1.1}), we get
	\(
	\xi \hat  \phi_k- \mathds{1}_{\xi\geq 0}  \widehat{ u\Pi(\overline u  \phi_k)}=k\hat \phi_k+\hat u.
	\)
Then we have
	\begin{equation}\label{1.9}
	(\xi-k)\hat \phi_k ={\mathds{1}_{\xi\geq 0}}\widehat{u\Pi(\overline u  \phi_k)} +  \hat u.
	\end{equation}
	Since $u\in L^2_+$ and $u\Pi(\overline u \phi_k)\in L^2$,  take the  inverse Fourier transform for \eqref{1.9},  equation \eqref{1.5} follows.
    Taking the inverse process from  \eqref{1.5} gives \eqref{1.1}. 
    
Next we prove $\phi_k\in H^1_+$ from \eqref{1.5}. Since $G_k\in L^2$, $u\in L^1$ and $u\Pi(\overline u\phi_k)\in L^1$,  we get $\phi_k\in L^2$ by Young's inequality. Using \eqref{1.9}, $u\in L^2_+$ and $u\Pi(\overline u  \phi_k)\in L^2$, we have
\begin{equation}\label{phik01}
    \int_\R |\xi-k|^2|\hat  \phi_k(\xi)|^2 d\xi =  \int_0^{+\infty}   (  |\hat u| + |\widehat {u\Pi(\overline u  \phi_k)}|)^2 d\xi< +\infty.
\end{equation}
Let $k=a+ib$, \eqref{phik01} becomes 
\(
	\int_\R ((\xi-a)^2+b^2)  |\hat  \phi_k(\xi)|^2 d\xi< +\infty
\).
By Plancherel formula,  we have
\begin{equation}\label{22.14}
\small
	\begin{split}
	\int_{0}^{+\infty} \vert \xi \hat{ \phi}_k(\xi) \vert^2 d\xi 
    \leq 2 \int_{0}^{+\infty} \left( (\xi - a)^2 + a^2 \right) \vert \hat{ \phi}_k(\xi) \vert^2 d\xi 
	\leq 2 \int_{0}^{+\infty} \vert \xi - k \vert^2 \vert \hat{ \phi}_k(\xi) \vert^2 d\xi + 2a^2 \cdot 2\pi \Vert  \phi_k \Vert_{L^2}^2 
	< +\infty.
	\end{split}
\end{equation}
From the Fourier transform, we deduce $\partial_k \phi_k\in L^2$. Therefore, we obtain $\phi_k\in H^1_+$.
\end{proof}

   \begin{definition}\label{defphilam}
 Assume that $u\in H^1_+\cap L^2_s$ for some $s> \frac 12$.  For every $\lambda\in (0,+\infty)\setminus \sigma_{pp}(\mathbb{L}_u)$, define 
\begin{equation}
\phi_\lambda^\pm=(\mathrm{id}-T_\lambda^\pm)^{-1}(G_\lambda^\pm*u),
\end{equation}
i.e.,
\begin{align} 
&\phi_\lambda^\pm=G_\lambda^\pm*\big(u+u\Pi(\overline u \phi_\lambda^\pm)\big)\label{phi_lambda_conv}.
		 \end{align}
       \end{definition}
\begin{lemma}\label{philamequi}
 Assume that $u\in H^1_+\cap L^2_s$ for some $s> \frac 12$. For every $\lambda\in (0,+\infty)\setminus \sigma_{pp}(\mathbb{L}_u)$,  formula \eqref{phi_lambda_conv} is equivalent to 
 \begin{align}
&\mathbb{L}_u\phi_{\lambda}^\pm=\lambda\phi_{\lambda}^\pm+u,\label{Luphilam}\\
&\phi_\lambda^+(x)\to 0,\quad x\to-\infty,\label{phi_xto-inf}\\
&\phi_\lambda^-(x)\to 0,\quad x\to+\infty,\label{phi_xto+inf}
\end{align}
for every $\phi_\lambda^\pm\in L_{-s}^\infty$. 
\end{lemma}

\begin{proof}
We only study the Jost function $\phi_\lambda^+$. From \eqref{phi_xto-inf}, \eqref{phi_xto+inf} and $\phi_\lambda^\pm\in L_{-s}^\infty$, for every fixed $a\in\R$, we have $\phi_\lambda^+\in L^\infty(-\infty,a)$ and $\phi_\lambda^-\in L^\infty(a,+\infty)$.
 Take the Fourier transform for \eqref{Luphilam}, we get 
	\(
	\xi \widehat {\phi}_\lambda^+- \mathds{1}_{\xi\geq 0}  \widehat{ u\Pi(\overline u \phi_\lambda^+)}=\lambda\widehat{\phi_\lambda^+}+\widehat u\). Then we have 
\begin{equation}
    \widehat {\phi}_\lambda^+=\frac{\widehat u}{\xi-(\lambda+i\varepsilon)}+\frac{\mathds{1}_{\xi\geq 0}  }{\xi-(\lambda+i\varepsilon)}\widehat{ u\Pi(\overline u \phi_\lambda^+)}-\frac{i\varepsilon}{\xi-(\lambda+i\varepsilon)}\widehat{\phi}_\lambda^+.
    \end{equation}
    By the inverse Fourier transform, we have
    \begin{equation}\label{three_cv}
{\phi}_\lambda^+=G_{\lambda+i\varepsilon}*u+G_{\lambda+i\varepsilon}*(u\Pi(\overline u \phi_\lambda^+))
+\mathcal{F}^{-1}\bigg(-\frac{i\varepsilon}{\xi-(\lambda+i\varepsilon)}\widehat{\phi}_\lambda^+\bigg).
    \end{equation}
    For every $\mathfrak{F}\in L^1\cap L^2$, from \eqref{cvL2}, we get 
    \[
    \|  \widetilde{G}_{\lambda+i\varepsilon}*\mathfrak{F}-\widetilde{G}_\lambda*\mathfrak{F} \|_{L^\infty}\leq \|\widetilde{G}_{\lambda+i\varepsilon}-\widetilde{G}_\lambda\|_{L^2}\|\mathfrak{F} \|_{L^2}\to 0, \quad \varepsilon\to 0^+.
    \]
    By dominated convergence theorem, we have
    \(
(ie^{i(\lambda+i\varepsilon)x}\mathds{1}_{\R^+})*\mathfrak{F}(x)\to (ie^{i\lambda x}\mathds{1}_{\R^+})*\mathfrak{F}(x)
    \) as $\varepsilon \to 0$. 
    Since $u\in L^2\cap L^1$, $u\Pi(\overline u \phi_\lambda^+)\in L^2\cap L^1$ and  \eqref{Gkc+c-}, we have \begin{equation}\label{G-conv-point}
    G_{\lambda+i\varepsilon}*u\to  G_{\lambda}^+*u,\quad G_{\lambda+i\varepsilon}*(u\Pi( \overline u \phi))\to G_{\lambda}^+*(u\Pi( \overline u \phi)), \quad \varepsilon \to 0^+ .
    \end{equation}
    Here the  convergence is  pointwise.
    Next we  prove that 
\begin{equation}\label{convo}
	\begin{aligned}
		\mathcal{F}^{-1}\left(-\frac{i\varepsilon}{\xi - (\lambda + i\varepsilon)}\widehat{\phi}_{\lambda}^+\right)&=(\varepsilon e^{i(\lambda + i\varepsilon)x}\mathds{1}_{\mathbb{R}^+})*\phi_\lambda^+.
		\end{aligned}
\end{equation} 
Since $\phi_\lambda^+ \in L_{-s}^\infty$, we have $\langle x\rangle^{-2}\langle x\rangle^{-s}\phi_\lambda^+ \in L^1 \cdot L^\infty \subset L^1$. Then there exists a sequence 
\((v_n)_{n\in\mathbb{N}} \subset \mathcal{S}\) such that 
\(
\| v_n - \langle x\rangle^{-2-s}\phi_\lambda^+ \|_{L^1} \to 0
\) as $n\to +\infty$.
Set \(w_n = \langle x\rangle^{-2-s} v_n \in \mathcal{S}\). 
According to Theorem 7.19(c) in  \cite{RudinFA} and   \(e^{i(\lambda + i\varepsilon)x} \mathds{1}_{\mathbb{R}_+} \in \mathcal{S}'\), we get 
$
\left(e^{i(\lambda + i\varepsilon)x} \mathds{1}_{\mathbb{R}_+}\right) * w_n \in C^\infty \cap \mathcal{S}'
$
and 
\[
\mathcal{F}\left(\left(e^{i(\lambda + i\varepsilon)x} \mathds{1}_{\mathbb{R}_+}\right) * w_n\right) = \widehat{e^{i(\lambda + i\varepsilon)x} \mathds{1}_{\mathbb{R}_+}} \cdot \widehat{w}_n = \frac{i}{\lambda + i\varepsilon - \xi} \widehat{w}_n.
\]
That is
\begin{equation}\label{con-sch}
\mathcal{F}^{-1}\left( -\frac{i\varepsilon}{\varepsilon - (i\lambda + i\varepsilon)} \widehat{w}_n \right) = \left( i\varepsilon e^{i(\lambda + i\varepsilon)x} \mathds{1}_{\mathbb{R}_+} \right) * w_n.
\end{equation}
Let \(\gamma_{\lambda+i\varepsilon} = \frac{1}{\varepsilon - (\lambda + i\varepsilon)}\). For every \(\varphi \in \mathcal{S}\),  we have 
\begin{align*}
\langle \gamma_{\lambda+i\varepsilon} \hat{w}_n - \gamma_{\lambda+i\varepsilon} \hat{\phi}_\lambda^+, \varphi \rangle_{\mathcal{S}',\mathcal{S}}
&= \langle \hat{w}_n - \hat{\phi}_\lambda^+, \gamma_{\lambda+i\varepsilon} \varphi \rangle_{\mathcal{S}',\mathcal{S}} = \langle w_n - \phi_\lambda^+, \widehat{\gamma_{\lambda+i\varepsilon} \varphi} \rangle_{\mathcal{S}',\mathcal{S}} \\
&= \int_{\mathbb{R}} (w_n - \phi_\lambda^+) \widehat{\gamma_{\lambda+i\varepsilon} \varphi} \d x = \int_{\mathbb{R}} (v_n - \langle x\rangle^{-2-s} \phi_\lambda^+) \langle x\rangle^{2+s} \widehat{\gamma_{\lambda+i\varepsilon} \varphi}  \d x.
\end{align*}
Since \(\langle x\rangle^{2+s} \widehat{\gamma_{\lambda+i\varepsilon} \varphi} \in \mathcal{S}\), we have
$
\left| \langle \gamma_{\lambda+i\varepsilon} \hat{w}_n - \gamma_{\lambda+i\varepsilon} \hat{\phi}_\lambda^+, \varphi \rangle_{\mathcal{S}',\mathcal{S}} \right| \leq \|\langle x\rangle^{2+s} \widehat{\gamma_{\lambda+i\varepsilon} \varphi}\|_{L^\infty} \|v_n - \langle x\rangle^{-2-s} \phi_\lambda^+\|_{L^1} \to 0
$ as $n\to+\infty.$
So we obtain 
\[
-\frac{i\varepsilon}{\varepsilon - (\lambda + i\varepsilon)} \widehat{w}_n \to -\frac{i\varepsilon}{\varepsilon - (\lambda + i\varepsilon)} \widehat{\phi}_\lambda^+ \quad \text{in } \mathcal{S}', \quad n\to+\infty.
\]
For the right hand side of \eqref{con-sch}, we deduce that 
\begin{align*}
&\left| \left( e^{i(\lambda + i\varepsilon)x} \mathds{1}_{\mathbb{R}_+} \right) * (w_n - \phi_\lambda^+) (x) \right|
= \left| \int_{-\infty}^x e^{i\lambda(x-y)} e^{-\varepsilon(x-y)} (w_n(y) - \phi_\lambda^+(y)) \, dy \right| \\
&\leq \int_{-\infty}^x e^{-\varepsilon(x-y)} |w_n(y) - \phi_\lambda^+(y)| \, dy \leq e^{-\varepsilon x} \sup_{y < x} |e^{\varepsilon y} \langle x\rangle^{2+s}| \|\langle x\rangle^{-2-s} (w_n - \phi_\lambda^+)\|_{L^1}.
\end{align*}
Set  \(M := \sup_{y < 0} |e^{\varepsilon y} \langle y \rangle^{2+s}| < +\infty\). Then we have 
\[
\sup_{|x| \leq R} \left| \left( e^{i(\lambda + i\varepsilon)x} \mathds{1}_{\mathbb{R}_+} \right) * (w_n - \phi_\lambda^+) (x) \right| \leq \left( M + e^{\varepsilon R} (1+R^2)^{\frac{2+s}{2}} \right) \|w_n - \langle x\rangle^{-2-s} \phi_\lambda^+\|_{L^1} \to 0, \quad n\to+\infty.
\]
This completes the proof of \eqref{convo}. 
When $\varepsilon\to 0^+$, we deduce that
$$
(\varepsilon e^{i(\lambda + i\varepsilon)x}\mathds{1}_{\mathbb{R}^+})*\phi_\lambda^+=
\varepsilon \int_{0}^{+\infty} e^{i(\lambda+i\varepsilon )y}\phi_\lambda^+(x-y)\d y=\int_{0}^{+\infty} e^{i\lambda \frac{y_1}\varepsilon}e^{-y_1}\phi_\lambda^+(x-\frac{y_1}\varepsilon)\d y_1
.$$
For any $x\in\R$, we recall that $\phi^+_\lambda\in L^\infty(-\infty,x)$, then we have $|e^{i\lambda \frac{y_1}\varepsilon}e^{-y_1}\phi_\lambda^+(x-\frac{y_1}\varepsilon)|\leq \|\phi_\lambda^+\|_{L^{\infty}(-\infty,x)}e^{-y_1}$. Since $\phi^+_\lambda\to 0$ as $x\to-\infty$,   by dominated convergence theorem, 
 we obtain 
\begin{equation}\label{eps-e-conv-to-0}
    \lim_{\varepsilon\to 0^+} \Big((\varepsilon e^{i(\lambda + i\varepsilon)x}\mathds{1}_{\mathbb{R}^+})*\phi_\lambda^+\Big)(x)=0.
\end{equation}
Finally, from \eqref{three_cv}, \eqref{G-conv-point}, \eqref{convo} and \eqref{eps-e-conv-to-0}, we obtain \eqref{phi_lambda_conv}.

Next we prove the inverse arguement. Using \eqref{glambda} and \eqref{phi_lambda_conv}, we get
\begin{equation}\label{philamint}
    \phi_\lambda^+=ie^{i\lambda x}\int_{-\infty}^x e^{-i\lambda y} \big( u(y)+ u(y)\Pi(\overline u \phi_\lambda^+) (y) \big) \ \mathrm{d} y-\widetilde G_\lambda * (u+u\Pi(\overline u \phi_\lambda^+)).
\end{equation}
Since $\widetilde{G}_\lambda\in L^2_-$, $u\in L^2_+$ and $u\Pi(\overline u \phi_\lambda^+)\in L^2_+$, we have $\widetilde G_\lambda * (u+u\Pi(\overline u \phi_\lambda^+))=0$.
Take the derivative with respect to $x$, we have 
\begin{equation}\label{Dphilam}
    \tfrac 1 i \partial_x \phi_\lambda^+=i\lambda e^{i\lambda x} \int_{-\infty}^x e^{-i\lambda y} \big( u(y)+ u(y)\Pi(\overline u \phi_\lambda^+) (y) \big) \ \mathrm{d} y+u+u\Pi(\overline u \phi_\lambda^+).
\end{equation}
From \eqref{Dphilam}, we have 
\begin{equation}
    \tfrac 1 i \partial_x \phi_\lambda^+ - u\Pi (\overline u \phi_\lambda^+)= u+ \lambda G_\lambda^+ * (u+u\Pi(\overline u \phi_\lambda^+))=u+\lambda \phi_\lambda^+.
\end{equation}
By dominated convergence theorem, \eqref{philamint} and $\widetilde G_\lambda * (u+u\Pi(\overline u \phi_\lambda^+))=0$, we have  $\phi_\lambda^+\to 0$ as $x\to-\infty$.
\end{proof}

  \begin{definition}\label{defpsilam}
 Assume that $u\in H^1_+\cap L^2_s$ for some $s> \frac 12$. For every $\lambda\in (0,+\infty)\setminus \sigma_{pp}(\mathbb{L}_u)$, define 
\begin{equation}
\psi_\lambda^\pm=(\mathrm{id}-T_\lambda^\pm)^{-1}e^{i\lambda x},
\end{equation}
i.e.,
\begin{align} 
\psi_\lambda^\pm(x)=e^{ix\lambda}+G_\lambda^\pm*\big(u\Pi(\overline u \psi_\lambda^\pm )\big).\label{psieintegral}
		 \end{align}
       \end{definition}
\begin{lemma}\label{psi-equi}
 Assume that $u\in H^1_+\cap L^2_s$ for some $s> \frac 12$.  For every $\lambda\in (0,+\infty)\setminus \sigma_{pp}(\mathbb{L}_u)$,
     the equation \eqref{psieintegral} is equivalent to 
  	\begin{align}     &\mathbb{L}_u\psi_{\lambda}^\pm=\lambda\psi_{\lambda}^\pm,\label{1.2}\\
        &\psi_\lambda^+-e^{ix\lambda} \to 0, \quad x\to -\infty,\label{11.6a}\\
    &\psi_\lambda^--e^{ix\lambda} \to 0, \quad x\to +\infty\label{11.6b},
 	\end{align}
    for every $\psi_\lambda^\pm\in L_{-s}^\infty$. 
\end{lemma}

\begin{proof}
We only work on the plus sign \(  \psi_\lambda^+(x)  \).
For $\varepsilon>0$, by the Fourier transform for \eqref{1.2}, 
 we get
\[
	\widehat{\psi}_{\lambda}^+=\frac{\mathds{1}_{\mathbb{R}^+}}{\xi - (\lambda + i\varepsilon)}\widehat{ u\Pi(\overline u \psi_\lambda^+)}
	-\frac{i\varepsilon}{\xi - (\lambda + i\varepsilon)}	\widehat{\psi}_{\lambda}^+.
\]
From the inverse Fourier transform, we deduce that
\begin{equation}\label{1.10}
	\psi_\lambda^+=G_{\lambda + i\varepsilon}*(u\Pi(\overline{u}\psi_\lambda^+ ))+\mathcal{F}^{-1}\left(-\frac{i\varepsilon}{\xi - (\lambda + i\varepsilon)}\widehat{\psi}_{\lambda}^+\right).
\end{equation}
Similar as the proof of Lemma \ref{philamequi}, we have 
$
\lim_{\varepsilon\to 0}G_{\lambda + i\varepsilon}*(u\Pi(\overline{u}{\psi_\lambda^+})) (x)= G_{\lambda  }^+*(u\Pi(\overline{u}{\psi_\lambda^+}))(x)
$ in the sense of pointwise convergence. 
In addition, we obtain 
\begin{equation}\label{F-1}
\small
	\begin{aligned}
		\mathcal{F}^{-1}\left(-\frac{i\varepsilon}{\xi - (\lambda + i\varepsilon)}\widehat{\psi}_{\lambda}^+\right)=(\varepsilon\mathds{1}_{\mathbb{R}^+}e^{i(\lambda + i\varepsilon)x})*\psi_\lambda^+
		=e^{i(\lambda + i\varepsilon)x}\int_{-\infty}^{\varepsilon x}e^{y_1}e^{-i\lambda y_1/\varepsilon}\psi_\lambda^+(y_1/\varepsilon)         \ \mathrm{d} y_1.
		\end{aligned}
\end{equation}
Since $\psi_\lambda^+$ is bounded on $(-\infty, a)$ for a fixed $a\in\R$, by dominated convergence theorem and $\psi_\lambda^+-e^{ix\lambda}\to 0$ as $x\to-\infty$, we obtain $\mathcal{F}^{-1}\left(-\frac{i\varepsilon}{\xi - (\lambda + i\varepsilon)}\widehat{\psi}_{\lambda}^+\right)\to e^{ix\lambda}$ as $\varepsilon\to 0$. So we obtain \eqref{psieintegral}. 
Taking the derivative with respect to $x$ of \eqref{psieintegral} yields \eqref{1.2}.
\end{proof}
As per Lemma 2.3 in  \cite{2}, the CMDNLS equation \eqref{21.1} admits a Lax pair structure $(\mathbb{L}_u, \mathbb{B}_u)$, with $\mathbb{L}_u$ defined by \eqref{21.4}. 
For the explicit form of the operator $\mathbb{B}_u$, see  \cite[eq. (2.4)]{2}
\begin{equation}\label{Budef}
    \mathbb{B}_u := \mathbb{T}_u \mathbb{T}_{\partial_x \bar{u}} - \mathbb{T}_{\partial_x u} \mathbb{T}_{\bar{u}} + i \left( \mathbb{T}_u \mathbb{T}_{\bar{u}} \right)^2 : H_+^1 \to L_+^2.
\end{equation}
Additionally, the operator $\widetilde{\mathbb{B}}_u$ is defined in  \cite[eq. (2.6)]{2}. Specifically, we have 
\begin{equation}\label{tildBudef1}
 \widetilde{\mathbb{B}}_u:=   \mathbb{B}_u - i \mathbb{L}_u^2 = i \partial_x^2 + 2 \mathbb{T}_u  \partial_x  \mathbb{T}_{\overline{u}}: H^2_+ \to L^2_+,
\end{equation}
by \eqref{21.4} and \eqref{Budef}.
If \( u \in H_+^2 \), then  $i\mathbb{B}_u$ is a {bounded self-adjoint operator on $L^2_+$.}
\begin{corollary}\label{Buto0}
 Assume that
    $u \in  H^2_s\cap H^3_+$ 
     for some   \( s > \frac 12 \).  For every $\lambda\in (0,+\infty)\setminus \sigma_{pp}(\mathbb{L}_u)$, let \( \psi^\pm_{\lambda} (t,x)\in L^\infty_{-s} \) denote the solution to \eqref{1.2}, \eqref{11.6a} and \eqref{11.6b}. 
Then we have 
\begin{equation}
    \mathbb{B}_u(\psi_\lambda^\pm) \in H^1_+ \quad \mathrm{and}\quad 
    \lim_{|x| \to +\infty} \bigl| \mathbb{B}_u(\psi_\lambda^\pm) (x) \bigr| = 0.
\end{equation}
\end{corollary}
\begin{proof}
Taking the derivative with respect to $x$ for $(\partial_x \overline u)\psi_\lambda^+$ yields
\begin{equation*}
    \begin{split}
        &\partial_{x} \left( (\partial_x \overline u)\psi_\lambda^+ \right) = (\partial_{x}^2  \overline{u}) \psi_\lambda^+ + i (\partial_{x}  \overline{u}) \left( -i \partial_{x} \psi_\lambda^+ \right) = (\partial_{x}^2  \overline{u}) \psi_\lambda^+ + i (\partial_{x}  \overline{u}) \left( u\Pi(  \overline{u} \psi_\lambda^+) + \lambda \psi_\lambda^+ \right)
        \\
        =& (\partial_{x}^2  \overline{u}) \psi_\lambda^+ + i (u \partial_{x}  \overline{u}) \Pi( \overline{u} \psi_\lambda^+) + i\lambda (\partial_{x}  \overline{u}) \psi_\lambda^+\in  L_{s}^2 \cdot L_{-s}^\infty + H^1 \cdot H^1 \cdot \Pi\left( L_{s}^2 \cdot L_{-s}^\infty \right) + L_{s}^2 \cdot L_{-s}^\infty \subset L^2.
    \end{split}
\end{equation*}
Then we obatin $(\partial_x \overline u)\psi_\lambda^+\in H^1$ and 
 \( \mathbb{T}_{\partial_x  \overline{u}} \left( \psi_\lambda^+ \right) = \Pi\left( (\partial_x  \overline{u}) \psi_\lambda^+ \right) \in \Pi(H^1) = H_+^1 \). Therefore, we have
\( \mathbb{T}_u T_{\partial_x  \overline{u}} \left( \psi_\lambda^+ \right) \in \mathbb{T}_u(H_+^1) \subset H_+^1 \).
Taking the derivative with respect to $x$ for $ \overline u \psi_\lambda^+$ yields
\begin{equation*}
    \begin{split}
        &\partial_x ( \overline{u} \psi_\lambda^+) = (\partial_x  \overline{u}) \psi_\lambda^+ + i \overline{u} \left( -i\partial_x\psi_\lambda^+ \right)= (\partial_x  \overline{u}) \psi_\lambda^+ + i \overline{u} \left( u\Pi( \overline{u} \psi_\lambda^+) + \lambda \psi_\lambda^+ \right)
        \\
        =&(\partial_x  \overline{u}) \psi_\lambda^+ + i|u|^2 \Pi( \overline{u} \psi_\lambda^+) + i\lambda  \overline{u} \psi_\lambda^+\in L_{s}^2 \cdot L_{-s}^\infty + H^1 \cdot H^1 \cdot \Pi\left( L_{s}^2 \cdot L_{-s}^\infty \right) + L_{s}^2 \cdot L_{-s}^\infty\subset L^2.
    \end{split}
\end{equation*}
Then  we obtain $ \overline{u} \psi_\lambda^+ \in H^1$ and $ \mathbb{T}_{ \overline{u}}(\psi_\lambda^+) = \Pi( \overline{u} \psi_\lambda^+) \in H_+^1 $. Therefore, we get $$\mathbb{T}_{\partial_x u} \mathbb{T}_{ \overline{u}} \left( \psi_\lambda^+ \right) \in \mathbb{T}_{\partial_x u}(H_+^1) \subset H_+^1.$$ 
Since  \( \mathbb{T}_u \mathbb{T}_{ \overline{u}} \left( \psi_\lambda^+ \right) \in \mathbb{T}_u(H_+^1) \subset H_+^1 \),  we obtain 
\[
 \left( \mathbb{T}_u \mathbb{T}_{ \overline{u}} \right)^2 \left( \psi_\lambda^+ \right) = \left( \mathbb{T}_u \mathbb{T}_{ \overline{u}} \right)\left( \mathbb{T}_u \mathbb{T}_{ \overline{u}} \left( \psi_\lambda^+ \right) \right) \in \mathbb{T}_u \mathbb{T}_{ \overline{u}}(H_+^1) \subset H_+^1 .
 \]

As a result, we obtain
\(  \mathbb{B}_u\left( \psi_\lambda^+ \right) = \mathbb{T}_u \mathbb{T}_{\partial_x  \overline{u}} \left( \psi_\lambda^+ \right) - \mathbb{T}_{\partial_x u} \mathbb{T}_{ \overline{u}} \left( \psi_\lambda^+ \right) + i \left( \mathbb{T}_u \mathbb{T}_{ \overline{u}} \right)^2 \left( \psi_\lambda^+ \right)\in H_+^1 + H_+^1 + H_+^1 \subset H_+^1 \). Then 
\( \mathbb{B}_u\left( \psi_\lambda^+ \right)(x) \to 0\) as \(\ |x| \to +\infty \).

\end{proof}

\begin{corollary}\label{B(phi)inH1}
Assume that
    $u \in  H^2_s\cap H^3_+$ 
     for some   \( s > \frac 12 \).
       For every $\lambda\in (0,+\infty)\setminus \sigma_{pp}(\mathbb{L}_u)$, 	let $\phi_\lambda^\pm\in L_{-s}^\infty$ denote the solution to \eqref{Luphilam}, \eqref{phi_xto-inf} and \eqref{phi_xto+inf}.
    Then we have 
    \begin{equation}\label{limtildBuphi+-}
        \lim_{x\to \mp\infty}(\widetilde{\mathbb{B}}_{u} (\phi_\lambda^\pm))(x)=0.
    \end{equation} 
    Moreover, for any $a\in\R$, we have 
\begin{equation}\label{suotildBuphi}
    \sup_{x>a} |\widetilde{\mathbb{B}}_u(\phi_\lambda^-)(x)|<+\infty,\quad \sup_{x<-a} |\widetilde{\mathbb{B}}_u(\phi_\lambda^+)(x)|<+\infty.
    \end{equation}
\end{corollary}
\begin{proof}
Fix $s > \frac{1}{2}$. 
Since $\mathbb{L}_u\phi_{\lambda}^\pm=\lambda\phi_{\lambda}^\pm+u$ and  $u\in H^2_s\cap H^3_+$, it follows  that
\begin{small}
\begin{align*}
&\partial_x((\partial_x \bar{u})\phi_\lambda^-)
= (\partial_x^2 \bar{u})\phi_\lambda^- + (\partial_x \bar{u})\left( i u \Pi (\bar{u} \phi_\lambda^-) + i u + i \lambda \phi_\lambda^- \right) 
= (\partial_x^2 \bar{u})\phi_\lambda^- + i u (\partial_x \bar{u}) \Pi (\bar{u} \phi_\lambda^-) + i u \partial_x \bar{u} + i \lambda (\partial_x \bar{u})\phi_\lambda^- \\
\in &L_s^2 \cdot L_{-s}^\infty + H^2 \cdot H^1 \cdot \Pi \left( L_s^2 \cdot L_{-s}^\infty \right) + H^2 \cdot H^1 + L_s^2 \cdot L_{-s}^\infty \subset L^2.
\end{align*}
\end{small}
So we have $(\partial_x \bar{u})\phi_\lambda^- \in H^1 $ and $T_u T_{\partial_x \bar{u}} (\phi_\lambda^-)
\in \Pi \left( H^1 \cdot \Pi (H^1) \right) \subset H^1$. 
Similarly, we deduce that that
\begin{align*}
\partial_x(\bar{u}\phi_\lambda^-)
& = (\partial_x \bar{u})\phi_\lambda^- + \bar{u} \left( i u \Pi (\bar{u} \phi_\lambda^-) + i u + i \lambda \phi_\lambda^- \right) 
= (\partial_x \bar{u})\phi_\lambda^- + i |u|^2 \Pi (\bar{u}\phi_\lambda^-) + i |u|^2 + i \lambda \bar{u} \phi_\lambda^- \\
&\in L_s^2 \cdot L_{-s}^\infty + H^2 \cdot H^2 \cdot \Pi \left( L_s^2 \cdot L_{-s}^\infty \right) + H^2 + L_s^2 \cdot L_{-s}^\infty \subset L^2.
\end{align*}
This implies  that $\bar{u}\phi_\lambda^- \in H^1$ and $T_{\partial_x \bar{u}} T_{\bar{u}} (\phi_\lambda^-)
\in \Pi \left( H^1 \cdot \Pi (H^1) \right) \subset H^1_+ $. Furthermore, we have  $\mathbb{T}_u \mathbb{T}_{\bar{u}} (\phi_\lambda^-)
\in \Pi \left( H^1 \cdot \Pi (H^1) \right) \subset H^1_+ $ and $\left( \mathbb{T}_u \mathbb{T}_{\bar{u}} \right)^2 (\phi_\lambda^-)
\in \mathbb{T}_u \mathbb{T}_{\bar{u}} (H^1) \subset H^1_+$. So we get
\begin{equation}\label{BuphiinH1}
\small
    \mathbb{B}_u (\phi_\lambda^-)=\mathbb{T}_u \mathbb{T}_{\partial_x \bar{u}} (\phi_\lambda^-)+\mathbb{T}_{\partial_x \bar{u}} \mathbb{T}_{\bar{u}} (\phi_\lambda^-)+\left( \mathbb{T}_u \mathbb{T}_{\bar{u}} \right)^2 (\phi_\lambda^-)\in H^1_+ \quad \mathrm{and} \quad - \partial_x u + i u \Pi (|u|^2) - i \lambda u \in H^1_+.
\end{equation}As a consequence, we have
\begin{equation}\label{limBuphi=0}
     \lim_{x\to\pm \infty} \left( \mathbb{B}_u(\phi_\lambda^-)  \right) (x)=  \lim_{x\to\pm \infty} \left( - \partial_x u + i u \Pi (|u|^2) - i \lambda u \right) (x) =0.
\end{equation}Formula \eqref{tildBudef1} yields that
\begin{equation}\label{tildBuphi-expBu}   \widetilde{\mathbb{B}}_u(\phi_\lambda^-)
= \mathbb{B}_u(\phi_\lambda^-) - i \mathbb{L}_u^2(\phi_\lambda^-)  
 =  \mathbb{B}_u(\phi_\lambda^-) - \partial_x u + i u \Pi (|u|^2) - i \lambda u- i \lambda^2 \phi_{\lambda}^-.
\end{equation}Since $\mathop{{\rm lim}}\limits_{x \to +\infty} \phi_{\lambda}^{-}(x) = 0$, formula \eqref{limtildBuphi+-} follows from  \eqref{limBuphi=0}.  By \eqref{BuphiinH1}, \eqref{tildBuphi-expBu} and $\phi_\lambda^-\in L_{-s}^\infty$, we obtain $\widetilde{\mathbb{B}}_u(\phi_\lambda^-)\in L_{-s}^\infty$. 
 So formula \eqref{suotildBuphi} is a direct consequence of \eqref{limtildBuphi+-}.
\end{proof}

\section{Scattering coefficients}\label{sec3}
In this section, we introduce  the scattering coefficients $\Gamma(\lambda)$ and $\beta(\lambda)$ using Jost functions $\phi_\lambda^\pm$ and $\psi_\lambda^\pm$.    Additionally, we derive 
 a relation between $\phi_\lambda^-$ and $\psi_\lambda^-$.  The relation is  fundamental in formulating the associated Riemann–Hilbert problem.
\begin{lemma}
    Assume that $u\in H^1_+\cap L_s^2$ for some $s>\frac12$. For every $k\in \C^+\cup \C^-\cup  \big( (-\infty, 0)\setminus \sigma_{pp}(\mathbb{L}_u)\big)$ and every $\lambda\in (0,+\infty)\setminus \sigma_{pp}(\mathbb{L}_u)$, let $\phi_k$, $\phi_\lambda^\pm$ and $\psi_\lambda^\pm$ be defined as Definition \ref{def3.7}, \ref{defphilam} and \ref{defpsilam}. Then we have 
        \begin{align} 
			&\psi_\lambda^+(x)=\Gamma(\lambda)	\psi_{\lambda}^-(x),\label{3.2}\\
&\phi_{\lambda}^+-\phi_{\lambda}^-=\beta(\lambda)\psi_{\lambda}^-,\label{24}
	\end{align}
    where 
    \begin{align}
&\Gamma(\lambda):=1 + i\int_\R \left( u\Pi(\overline{u}	\psi_\lambda^+)\right)(y)e^{-i\lambda y}
\mathrm{d} y,\label{gamma}\\
&	\label{3.5}
\beta(\lambda):=i\int_{\mathbb{R}}(u+u\Pi(\overline{u}\phi_{\lambda }^+))(y)e^{-i\lambda y}         \ \mathrm{d} y.
\end{align}
And it follows that
\begin{align} 
	&|\Gamma(\lambda)| = 1,\quad	|\beta(\lambda)|^2 = 2\mathrm{Im}\int_{\mathbb{R}}  2\mathrm{Im}\int_\R  \overline u(x)\phi_{\lambda}^+(x)  \ \mathrm{d}x.  \ 
	\end{align} 
\end{lemma}
\begin{proof}
Recall the definition   \eqref{glambda} and \eqref{glambda-}, it follows that
$
G_{\lambda  }^+-G_{\lambda }^-=ie^{i\lambda x} .
$
Then we deduce that 
\begin{small}
    \begin{equation}\label{psie2}
		\begin{aligned}
		\psi_\lambda^+&=e^{i\lambda x}+(G_\lambda^-
        +ie^{i\lambda x})*(u\Pi(\overline{u}	\psi_\lambda^+ ))=\left(1 + i\int_\R e^{-i\lambda y}u\Pi(\overline{u}	\psi_\lambda^+) \ \mathrm{d} y\right)e^{i\lambda x}+G_\lambda^- *(u\Pi(\overline{u}	\psi_\lambda^+)).
	\end{aligned} 
\end{equation}
\end{small}
Since the operator $\mathrm{id}-T_\lambda^-$ is bijective, equation \eqref{3.2} follows. 
Using the same argument, we have
\begin{align*}
	\phi_{\lambda}^+-\phi_{\lambda}^-
    &=G_\lambda^+*(u+u\Pi(\overline u \phi_\lambda^+))-G_\lambda^-*(u+u\Pi(\overline u \phi_\lambda^-))
    \\
    &=(  G_\lambda^+-G_\lambda^-  )*u+(G_\lambda^-+ie^{i\lambda x})*(u\Pi(\overline u \phi_\lambda^+))-G_\lambda^-*(u\Pi(\overline u \phi_\lambda^-))
    \\
    &=ie^{i\lambda x}\int_\R (u+u\Pi(\overline{u}\phi_{\lambda}^+))e^{-i\lambda y}         \ \mathrm{d} y + G_{\lambda}^-*(u\Pi(\overline{u}(\phi_\lambda^+-\phi_\lambda^-))).
\end{align*}
Again by the bijectivity of the operator $\mathrm{id}-T_\lambda^-$, we conclude \eqref{24}.

Next, applying the formula \eqref{Glambda^+*F,F}, we compute
\begin{equation}\label{innerpsiuPiupsi}
\begin{split}
\langle\psi_{\lambda }^+, u\Pi(\overline{u}\psi_{\lambda }^+)\rangle_{L^\infty, L^1}
&
=
\langle    e^{ix\lambda}+  G_\lambda^+*(u\Pi(\overline u \psi_\lambda^+)) ,    u\Pi(\overline{u}\psi_{\lambda }^+)            \rangle_{L^\infty, L^1}
\\
&
=
\overline{\mathcal{F}(u\Pi(\overline u\psi_\lambda^+))}+i|{\mathcal{F}(u\Pi(\overline u\psi_\lambda^+))}|^2-{\mathcal{F}(u\Pi(\overline u\psi_\lambda^+))}+\langle u\Pi(\overline{u}\psi_{\lambda }^+),\psi_{\lambda }^+,\rangle_{ L^1,L^\infty}.
\end{split}
\end{equation}
Note that  $\langle\psi_{\lambda }^+, u\Pi(\overline{u}\psi_{\lambda }^+)\rangle_{L^\infty, L^1}=\| \Pi(\overline{u}\psi_{\lambda }^+) \|_{L^2}^2=\langle u\Pi(\overline{u}\psi_{\lambda }^+),\psi_{\lambda }^+,\rangle_{ L^1,L^\infty}$. Substituting \eqref{innerpsiuPiupsi} and simplifying   leads to $|\Gamma(\lambda)|=|1+{\mathcal{F}(u\Pi(\overline u\psi_\lambda^+))}|=1$.

Finally, from formula \eqref{phi_lambda_conv} and \eqref{Glambda^+*F,F}, we have 
$$
\langle\phi_{\lambda }^+, u+u\Pi(\overline{u}\phi_{\lambda }^+)\rangle_{L^\infty, L^1}=i|\mathcal{F}(u+u\Pi(\overline{u}\phi_{\lambda }^+))|^2+\langle u+u\Pi(\overline{u}\phi_{\lambda }^+), \phi_{\lambda }^+\rangle_{L^1,L^\infty}.
$$
Since $\langle\phi_{\lambda }^+, u\Pi(\overline{u}\phi_{\lambda }^+)\rangle_{L^\infty, L^1}=\| \Pi(\overline{u}\phi_{\lambda }^+) \|_{L^2}^2=\langle u\Pi(\overline{u}\phi_{\lambda }^+),\phi_{\lambda }^+,\rangle_{ L^1,L^\infty}$  is real,  and from formula \eqref{3.5} we have $\beta(\lambda)=i\mathcal{F}(u+u\Pi(\overline{u}\phi_{\lambda }^+))$,  it follows that
	\[
	 |\beta(\lambda)|^2 = 2\mathrm{Im}\int_\R  \overline u\phi_{\lambda}^+  \ \mathrm{d}x.
	\]
\end{proof}

To  eliminate  $\psi_{\lambda}^-$ in \eqref{24}  and proceed with the construction of the Riemann–Hilbert problem, we now  derive the relation 
\begin{equation}\label{phi-psi-relation}
e^{i\lambda x} \partial_{\lambda} \left( e^{-i\lambda x}\psi_{\lambda}^- \right) = \left( \frac{1}{2\pi} \int_{\mathbb{R}} \overline{u} \psi_{\lambda}^- \,      \ \mathrm{d} x \right) \phi_{\lambda}^-.
\end{equation}

Recall that $\psi_\lambda^-=e_\lambda+G_\lambda^- * \left( u \Pi (\overline{u} \psi_{\lambda}^-)\right) $.  Using  (18) in  \cite{12} and \eqref{glambda-}, we obtain that 
	\[
		\begin{aligned}
	 \partial_{\lambda} \left( e_{-\lambda} \psi_{\lambda}^- \right) 
		&= -i x e_{-\lambda} \left( G_\lambda^- * \left( u \Pi (\overline{u} \psi_{\lambda}^-) \right) \right)+ e_{-\lambda}\left( \partial_{\lambda} G_\lambda^- \right) * \left( u \Pi (\overline{u} \psi_{\lambda}^-) \right)+ e_{-\lambda}G_\lambda^- * \left( u \Pi (\overline{u} \partial_{\lambda} \psi_{\lambda}^-) \right) \\
		&= -i x e_{-\lambda} \left( G_\lambda^- * \left( u \Pi (\overline{u} \psi_{\lambda}^-) \right) \right) + e_{-\lambda}\left( {-\frac{1}{2\pi\lambda} + i x G_\lambda^-} \right) * \left( u \Pi (\overline{u} \psi_{\lambda}^-) \right) 
        \\
		 &\quad + e_{-\lambda}G_\lambda^- * \left( u \Pi \left( \overline{u} e_\lambda \partial_{\lambda} \left( e_{-\lambda} \psi_{\lambda}^- \right) \right) \right) + i e_{-\lambda} G_\lambda^- * \left( u \Pi (x \overline{u} \psi_{\lambda}^-) \right).
	\end{aligned}
	\]
    The formula $-ix(G_\lambda^-*(u\Pi(\overline u \psi_\lambda^-)))+(ixG_\lambda^-)*(u\Pi(\overline u \psi_\lambda^- ))=-iG_\lambda^-*(xu\Pi(\overline u \psi_\lambda^-))$ allows us to simplify the expression. Multiplying both sides by $e_\lambda$ yields 
    \begin{equation}\label{eparepsi123}
    \small
    \begin{split}
		e_\lambda \partial_{\lambda} \left( e_{-\lambda} \psi_{\lambda}^- \right)&= -\frac{1}{2\pi\lambda} \int_{\mathbb{R}} u \Pi (\overline{u} \psi_{\lambda}^-) \,      \ \mathrm{d} x - i \big( {G_\lambda^-} * \left( x u \Pi (\overline{u} \psi_{\lambda}^-) \right)
		 - G_\lambda^- * \left( u \Pi (x \overline{u} \psi_{\lambda}^-) \right) \big)\\
		&\quad +G_\lambda^- * \left( u \Pi \left( \overline{u} e_\lambda \partial_{\lambda} \left( e_{-\lambda} \psi_{\lambda}^- \right) \right) \right) = \mathrm{I} + \mathrm{II} + \mathrm{III}.
    \end{split}\end{equation}
    Given that  $u\in H^1_+\cap L_{s}^\infty$ and $\psi_\lambda^-\in L_{-s}^\infty$, it follows  \( u \Pi (\overline{u} \psi_{\lambda}^-) \in L^1 \cap L^2_+ \). { Thus,  we have }
	\begin{equation}\label{=0}
	 \mathrm{I}=-\frac{1}{2\pi\lambda}\int_{\mathbb{R}} u \Pi (\overline{u} \psi_{\lambda}^-) \,      \ \mathrm{d} x= -\frac{1}{2\pi\lambda}\mathcal{F}({u \Pi (\overline{u} \psi_{\lambda}^- )})(0^-) = 0.
	\end{equation}
 For any $f \in L_1^2$, let $Xf: \mathbb{R} \to \mathbb{C}$ be defined by
\begin{equation}
    (Xf)(x) = xf(x), \quad \forall x \in \mathbb{R}.
\end{equation}Then $X : L^2_1 \to L^2$ is a linear mapping. The following lemma gives a commutator formula that is used to simplify term $\mathrm{II}$ in \eqref{eparepsi123}.
\begin{lemma}\label{xPi-Pix}
If $m\in L_1^2$, then we have  $X\Pi(m)-\Pi(Xm)=\frac{i}{2\pi} \hat{m}(0^+)$.
\end{lemma}

\begin{proof}
Since $m\in L_1^2\subset L^1\cap L^2\subset \mathcal{S}'$, 
we have  $\mathds{1}_{\mathbb{R}^+} \hat{m}\in L^2\in \mathcal{S}'$.
	 For any test function \(  n \in \mathcal{S} \), we deduce that  
     \[
     \small
	\begin{aligned}
		&- \langle \partial_{\xi} (\mathds{1}_{\mathbb{R}^+} \hat{m}), n \rangle_{\mathcal{S}', \mathcal{S}} = \langle \mathds{1}_{\mathbb{R}^+} \hat{m}, \partial_{\xi} n \rangle_{\mathcal{S}', \mathcal{S}} = \int_{0}^{\infty} \hat{m}(\xi) \partial_{\xi} n(\xi) \, d\xi 
		= -\hat{m}(0^+) n(0^+) - \langle \mathds{1}_{\mathbb{R}^+} \partial_{\xi} \hat{m}, n \rangle_{\mathcal{S}', \mathcal{S}} \\
		\quad=& -\langle \hat{m}(0^+) \delta_0 + \mathds{1}_{\mathbb{R}^+} \partial_{\xi} \hat{m}, n \rangle_{\mathcal{S}', \mathcal{S}}.
	\end{aligned}
	\]
	This implies 
	$$
	\partial_{\xi} (\mathds{1}_{\mathbb{R}^+} \hat{m}) = \hat{m}(0^+) \delta_0 + \mathds{1}_{\mathbb{R}^+} \partial_{\xi} \hat{m}.
	$$
	Now consider the Fourier transform:
	$$
		\mathcal{F}(x \Pi m - \Pi (x m))  = i \partial_{\xi} (\mathds{1}_{\mathbb{R}^+} \hat{m})  -i\mathds{1}_{\mathbb{R}^+}  \partial_{\xi} \hat{m} = i \hat{m}(0^+) \delta_0 = i \hat{m}(0^+) \frac{\hat{\mathds 1}}{2\pi} = \left( \frac{i}{2\pi} \hat{m}(0^+) \right)^{\wedge}.
	$$
    Taking the inverse Fourier transform gives the desired identity.
\end{proof}
Applying Lemma \ref{xPi-Pix} with $m=\overline{u} \psi_{\lambda}^-\in L^2_1$, we obtain that 
\begin{equation}
    Xu\Pi(\overline{u} \psi_{\lambda}^-)-u\Pi(X\overline{u} \psi_{\lambda}^-)=\left(\frac i {2\pi} \int_{\mathbb{R}} \overline{u} \psi_{\lambda}^- \,       \right) u.
\end{equation}Then term  
 $\mathrm{II}$ simplifies to
	\[
	\mathrm{II} = \left( \frac{1}{2\pi} \int_{\mathbb{R}} \overline{u} \psi_{\lambda}^- \,      \ \mathrm{d} x \right) \left( G_\lambda^- * u \right).
	\]
	Substituting into \eqref{eparepsi123}, we conclude
	\[
	e_\lambda \partial_{\lambda} \left( e_{-\lambda} \phi_{\lambda,-} \right) = \left( \frac{1}{2\pi} \int_{\mathbb{R}} \overline{u} \psi_{\lambda}^- \,      \ \mathrm{d} x \right) \left( G_\lambda^- * u \right) + G_\lambda^- * \left( u \Pi \left( \overline{u} e_{\lambda} \partial_{\lambda} \left( e_{-\lambda} \psi_{\lambda}^- \right) \right) \right).
	\]
	Thanks to the bijectivity  of the operator $\mathrm{id}-T_\lambda^-$, it follows that
	\[
	e_\lambda \partial_{\lambda} \left( e_{-\lambda} \psi_{\lambda}^- \right) = \left( \frac{1}{2\pi} \int_{\mathbb{R}} \overline{u} \psi_{\lambda}^- \,      \ \mathrm{d} x \right) \phi_{\lambda}^-.
	\]

    \section{Spectral asymptotics}\label{sec4}

In this section, we aim to investigate the spectral asymptotic behavior of the Jost functions and scattering coefficients at $0$ and $\infty$.   To study the asymptotics at $0$, we need to regularize the singularity $\frac{1}{\xi-k}$ arising in the Green functions $G_k$ and $\widetilde{G}_{\lambda}$ in \eqref{gk}. Following the singularity cancellation technique used in Section 2 of Coifman–Wickerhauser  \cite{13} and Section 5 of Wu  \cite{16}, we introduce the modified Jost functions.

First, choose a smooth cut-off function $\chi$ on $[0, +\infty)$ satisfying $\chi(\xi)=1$ for $0\le\xi\le 1$ and $\chi(\xi)=0$ for $\xi\ge 2$. For each $k\in\mathbb{C}^+\cup\mathbb{C}^-\cup\big((-\infty,0)\setminus\sigma_{pp}(\mathbb{L}_u)\big)$ and $\lambda\in(0,+\infty)\setminus\sigma_{pp}(\mathbb{L}_u)$, we define the constants as follows,
\begin{equation*}
    \begin{split}
        &l_k : =\frac{1}{2\pi}\int_{0}^{\infty}\frac{\chi(\xi)}{\xi - k}d\xi \in \mathbb{C}
        , \quad l_\lambda : =\frac{1}{2\pi}\int_{-\infty}^0\frac{\chi(-\xi)}{\xi - \lambda}d\xi \in \mathbb{C}.
    \end{split}
\end{equation*} 
The modified Green functions $M_k$ and $M_\lambda^\pm$ are then defined by
\begin{align}
	M_k(x):=G_{k}(x)-l_k=\frac{1}{2\pi}\int_{0}^{\infty}\frac{e^{ix\xi}- \chi(\xi)}{\xi - k}d\xi,  \quad M_\lambda^\pm(x):=\pm ie^{i\lambda x}\mathds{1}_{\R^\pm}-(\widetilde G_\lambda-l_\lambda)=G_\lambda^\pm(x) +l_\lambda.\label{4.4}
	\end{align} 
 
According to Section 2 of  \cite{13} and Lemma 5.2 of  \cite{16}, the modified Green function  $M_k$ has no singularity at $k=0$. 
Similarly, $M_\lambda^{\pm}$ has no singularity at $\lambda=0$.   
The modified operators $\widetilde T_k: L_{-s}^{\infty} \to L_{-s}^{\infty}$ and  $\widetilde T_\lambda^\pm: L_{-s}^{\infty} \to L_{-s}^{\infty}$ are defined as follows,
\begin{align}
\widetilde T_k:=M_k\ast(u\Pi(\overline{u}\hspace{0.1cm}\cdot\hspace{0.1cm}))=T_{k}-l_ku\Pi(\overline{u}\hspace{0.1cm}\cdot\hspace{0.1cm}), \quad  \widetilde T_\lambda^\pm:=M_\lambda^\pm \ast (u\Pi(\overline{u}\hspace{0.1cm}\cdot\hspace{0.1cm}))=T_{\lambda}^\pm+l_\lambda u\Pi(\overline{u}\hspace{0.1cm}\cdot\hspace{0.1cm})  .\label{4.3}
	\end{align}

\begin{lemma}\label{lemma3.2}
	   Assume that $u\in H^1_+\cap L_s^2$ for some $s>\frac12$. 
       Then there exist two constants $\mathfrak{C}>0$ and   $a_1>0$ such that for every $0<\varepsilon<\min\{1, s\}$, we have 
     \begin{equation}\label{TkT0}
         \|  \widetilde T_k -\widetilde T_0   \|_{\mathcal{B}(L_{-s}^\infty)}\leq \mathfrak{C}|k|^\varepsilon,\quad \text{if} \quad |k|<a_1.
     \end{equation}
Therefore, the operator $\widetilde T_k$  is compact from  $L_{-s}^{\infty} $ to $L_{-s}^{\infty} $, $\forall k\in \C^+\cup \C^- \cup \{0\}\cup  \big((-\infty,0)\setminus \sigma_{pp}(\mathbb{L}_u)\big)$.
	\end{lemma}

\begin{proof}
Since $T_k$ is compact, and   $l_ku\Pi(\overline{u}~\cdot)$ is a rank one operator,   $\widetilde T_k$  is compact as $k\neq 0$.
According to Lemma 5.1 of  \cite{16}, for every $f\in L_{-s}^\infty$, there exists $a_1>0
$ such that for $|k|<a_1$ and  every $0<\varepsilon<\min\{1, s\}$, we have 
\begin{equation*}\small
\begin{split}
	&\lVert 
	\widetilde T_k f -\widetilde T_0 f
	\rVert_{L_{-s}^\infty}
	=\lVert   (M_k-M_0) \ast (u\Pi(\overline u f))      \rVert_{L_{-s}^\infty}
    \leq{\langle x\rangle}^{-s} \left|  \int_\R (M_k(x-y)-M_0(x-y))u(y)\Pi(\overline u f) \ \mathrm{d} y   \right|
    \\
    &\leq 
	C  |k|^{\varepsilon}  {\langle x\rangle}^{-s} \int_\R {\langle x-y\rangle}^{\varepsilon}   u \Pi(\overline u f)\ \mathrm{d} y
    \leq C  |k|^{\varepsilon}  {\langle x\rangle}^{-s} \int_\R {\langle x\rangle}^{\varepsilon} {\langle y\rangle}^{\varepsilon}   {\langle y\rangle}^{-s} {\langle y\rangle}^{s} u \Pi(\overline u f)\ \mathrm{d} y
    \leq C  |k|^{\varepsilon}
    \lVert  u  \rVert^2_{L_s^2}  \lVert f  \rVert_{L_{-s}^\infty} .
    \end{split}
	\end{equation*}
Therefore, the operator 
\(
\widetilde T_k\to \widetilde T_0
\)
as $k\to 0$ 
in  $\mathcal{B}(L_{-s}^\infty  )$. Since the space of all compact linear operators is closed
,   the operator $\widetilde T_0
$ is compact.
	\end{proof}

\begin{proposition}\label{identity-prop}
    Let $s > 1$. Suppose $u \in H^1_+ \cap H^1_s$. Suppose further that $g \in L^{\infty}_{-s}$ satisfies $\widetilde{T}_0 g = g$.
    Then $g \in H^2_+$, and the following identities hold:
    \begin{equation}\label{innprodugiden}
       \mathbb{L}_u(g) = 0, \quad |\langle u, g\rangle_{L^2_+}|^2 = 2 \pi \|g\|_{L^2_+}^2.
    \end{equation}
    As a result, if $0$ is not an eigenvalue of $\mathbb{L}_u$, the operator $\mathrm{id} - \widetilde{T}_0$ is injective from $L^{\infty}_{-s}$ to itself.
\end{proposition}

\begin{remark}
    The second formula in \eqref{innprodugiden} was first established by G\'erard-Lenzmann~ \cite{2}. We refer the readers to Proposition~5.1 in their work for a different derivation.
\end{remark}

\begin{proof}
Let $g\in L_{-s}^{\infty} $ and $g = \widetilde T_0 g$. 
For every $h\in\mathcal{S}$, using $g = \widetilde T_0 g$ and Fubini's theorem, we obtain that 
	\begin{equation} 
		\begin{split}
		\langle -\partial_xg, h\rangle_{\mathcal{S}^\prime,\mathcal{S}}=
        \langle g, \partial_xh\rangle_{\mathcal{S}^\prime,\mathcal{S}}
        =&\int_\R \bigg(  \int_\R M_0(x - y)u(y)\Pi(\overline{u}g)(y)         \ \mathrm{d} y \bigg) \partial_x h     \ \mathrm{d} x
        \\
		=&
        \frac{1}{2\pi}\int_\R  \bigg(  \int_\R \bigg( \int_{0}^{\infty}\frac{e^{i(x - y)\xi}-\chi(\xi)}{\xi}d\xi\bigg) \  u(y)\Pi(\overline{u}g)(y)          \ \mathrm{d} y\bigg) \partial_xh     \ \mathrm{d} x
\\
=&
	\frac{1}{2\pi}\int_{\mathbb{R}}\left(\int_{0}^{\infty}\bigg( \int_{\mathbb{R}}\frac{e^{i(x - y)\xi}-\chi(\xi)}{\xi}\partial_x h     \ \mathrm{d} x\bigg) d\xi\right)u(y)\Pi(\overline{u}g)(y)         \ \mathrm{d} y.
	\end{split}
		\end{equation}
	Integrating by part and using the Fubini's theorem, we obtain that 
    \begin{equation*}
        \begin{split}
            \langle -\partial_xg, h\rangle_{\mathcal{S}^\prime,\mathcal{S}}
            =&-\frac{i}{2\pi}\int_{\mathbb{R}}\bigg( \int_{0}^{\infty} \bigg(\int_{\mathbb{R}}e^{i\xi(x - y)}h(x)     \ \mathrm{d} x \bigg)d\xi\bigg) u(y)\Pi(\overline{u}g)(y)         \ \mathrm{d} y
            \\
            =&
            	-\frac{i}{2\pi}\int_{\mathbb{R}}h(x)\bigg(\int_{0}^{\infty}e^{i\xi x}\left(\int_{\mathbb{R}}e^{-i\xi y}u(y)\Pi(\overline{u}g)(y)         \ \mathrm{d} y\right)d\xi \bigg
                )dx.
        \end{split}
    \end{equation*}
The function $u\Pi(\overline u g)$ belongs to $L^1 \cap L^2_+$. Hence, its Fourier transform is a continuous function; moreover, it is supported on the non-negative real axis.   So we have
\[
\langle -\partial_xg, h\rangle_{\mathcal{S}^\prime,\mathcal{S}}=-i\int_{\mathbb{R}}h(x)u\Pi(\overline{u}g)     \ \mathrm{d} x.
\]
Therefore, we have 
\begin{equation}\label{tempered}
-i\partial_{x}g - u\Pi(\overline{u}g)=0\in\mathcal{S}^\prime\implies \mathbb{L}_ug=0.
\end{equation}

Next we claim that 
\begin{equation}\label{ginL2}
    g\in L^2.
\end{equation}

\begin{proof}[Proof of formula \eqref{ginL2}]
    Since  \( g = \widetilde{T}_0 g \), we have 
\(
g(x) = \int_{\mathbb{R}} M_0(x-y) u \Pi(\overline ug)(y) \, \d y.
\)
Referring to equation (5.22) in  \cite{16}, we decompose $2\pi g(x)$ as follows:
\[
\begin{split}
   2\pi g(x) =
   &-\int_x^\infty \log(y-x) u(y) \Pi(\overline{u}g)(y) \, dy
-\int_{-\infty}^x \log(x-y) u(y) \Pi(\overline{u}g)(y) \, dy 
\\
&+  (c_2 - c_1) \int_x^\infty u \Pi(\overline u g) \, dy=I+II+III,
\end{split}
\]
where $c_1$ and $c_2$ are constants, and  the terms are defined by 
\begin{align}
    &I=-\int_x^\infty \log(y-x) u(y) \Pi(\overline{u}g)(y) \, dy,\label{xto+infty}
    \\
    &II=-\int_{-\infty}^x \log(x-y) u(y) \Pi(\overline{u}g)(y) \, dy,\label{-inftytox}
    \\
    & III=(c_2 - c_1) \int_x^\infty u \Pi(\overline u g) \, dy= (c_1 - c_2) \int_{-\infty}^x u \Pi(\overline{u}g) \, dy,\label{3-x-inf-inf-x}
\end{align}
and the last equality holds because $\int_\R u \Pi(\overline u g) \, dy=0$ in \eqref{=0}.

We begin by  considering the case   $x>0$. Given that $u\in L^2_s$ and $g\in L_{-s}^\infty$, we estimate term $III$:
\begin{equation}
|III(x)| \leq |c_1-c_2| \int_x^\infty |u \Pi(\overline u g)| \d y \leq C \langle x\rangle^{-s}\|u\|_{L^2_{s}} \|\overline u g\|_{L^2}.
\end{equation}
This implies that $III\in L^2$. Next we analyze  term $I$ by splitting the integral  
\begin{equation*}
\begin{split}
    I&=-\int_x^\infty \log(y-x) u(y) \Pi(\overline{u}g)(y) \, dy
    \\
    &=-\int_x^{x+1} \log(y-x) u(y) \Pi(\overline{u}g)(y) \, dy-\int_{x+1}^\infty \log(y-x) u(y) \Pi(\overline{u}g)(y) \, dy
    =I_a+I_b,
\end{split}
\end{equation*}
where $I_a=-\int_x^{x+1} \log(y-x) u(y) \Pi(\overline{u}g)(y) \, dy$ and $I_b=-\int_{x+1}^\infty \log(y-x) u(y) \Pi(\overline{u}g)(y) \, dy$. We bound $I_a$ as:
\begin{small}
\begin{equation}
     |I_a|= \left|\int_x^{x+1} \log(y-x) u(y) \Pi(\overline u g)(y) \, dy\right|\leq \bigg(\int_0^1|\log \theta|^2\d \theta \bigg)^\frac12 \|u\|_{L^2_s}\|\Pi(\overline u g)\|_{L^\infty}\langle x\rangle^{-s}.
\end{equation}
\end{small}
For $I_b$, we use the fact that any $\kappa>0$, the quantity $C_\kappa$ is defined as 
$$C_\kappa:=\sup_{t>1}\frac{\log t}{t^\kappa}<+\infty.$$  Then, 
\begin{equation*}
\begin{split}
    |I_b|
    &=\bigg|\int_{x+1}^\infty \log(y-x) u(y) \Pi(\overline{u}g)(y) \, dy\bigg|\leq C_\kappa \int_{x+1}^\infty |y-x|^\kappa |u(y) \Pi(\overline{u}g)(y) | \d y
    \\
    &\leq C_\kappa \langle  x \rangle^\kappa \int_{x+1}^\infty \langle  y \rangle^\kappa \langle  y \rangle^{-b} \langle  y \rangle^{b}|u(y) \Pi(\overline{u}g)(y) |\d y\leq C_\kappa \| u \|_{L^2_{\kappa+b}}\| \Pi(\overline u g)\|_{L^2}\langle  x \rangle^{-b+\kappa}.
\end{split}
\end{equation*}
By choosing $\kappa<\frac 12(s-\frac 12)$ and $b=\frac 12(s+\frac12)$, we ensure $-b+\kappa<-\frac 12$ and $\kappa+b<s$. Consequently,  $I\in L^2$. We now turn to  term $II$. For $0<x<1$, we decompose  
\begin{equation*}
\begin{split}
    II
    &=-\int_{-\infty}^{x} \log(x-y) u(y) \Pi(\overline u g)(y) \, dy
    \\
    &=-\int_{-\infty}^{-1} \log(x-y) u(y) \Pi(\overline u g)(y) \, dy-\int_{-1}^{x} \log(x-y) u(y) \Pi(\overline u g)(y) \, dy={II}_A^{<1}+{II}_B^{<1},
\end{split}
\end{equation*}
where ${II}_A^{<1}=-\int_{-\infty}^{-1} \log(x-y) u(y) \Pi(\overline u g)(y) \, dy$ and ${II}_B^{<1}=-\int_{-1}^{x} \log(x-y) u(y) \Pi(\overline u g)(y) \, dy$. 
The terms are bounded by:
\begin{equation*}
\begin{split}
    |{II}_A^{<1}|
    &\leq \int_{-\infty}^{-1}\log(|y|+1)|u(y)\Pi(\overline u g)(y)| \d y
    \\
    &\leq  \int_{-\infty}^{-1}\frac{\log(|y|+1)}{\langle y \rangle^s}|\langle y \rangle^su(y)\Pi(\overline u g)(y)| \d y\leq \sup_{|y|>1} \bigg(\frac{\log(|y|+1)}{\langle y \rangle^s}\bigg)\| u\|_{L^2_s}\| \Pi(\overline u g) \|_{L^2}.
\end{split}
\end{equation*}
Then ${II}_A^{<1} \in L^2(0,1)$. By  \eqref{tempered}, we get $-i\partial_xg=u\Pi(\overline u g)\in L^2$. Since $\partial_x u 
\in L^2_s$,  we have $\partial_x (\overline u g)\in L^2$ and $\overline u g\in H^1$. So $\Pi(\overline u g)\in H^1$. Then we have 
\begin{equation*}
    |{II}_B^{<1}|\leq \bigg(\int_0^{x+1} |\log t|^2 \d t \bigg)^{\frac 12}\|u \|_{L^2}\|\Pi(\overline u  g) \|_{L^\infty}.
\end{equation*}
Thus, ${II}\in L^2(0,1)$. 

For $x>1$, we further decompose $II$:
\begin{equation*}
    \begin{split}
        II
        &=-\int_{-\infty}^x \log(x-y) u(y) \Pi(\overline u g)(y) \, dy
        \\
        &=-\int_{-\infty}^x \big(\log(x-y)-\log x\big) u(y) \Pi(\overline u g)(y) \, dy-\int_{-\infty}^x \log x \ u(y) \Pi(\overline u g)(y) \, dy={II}_\alpha^{>1}+{II}_\beta^{>1},
    \end{split}
\end{equation*}
where ${II}_\alpha^{>1}=-\int_{-\infty}^x \big(\log(x-y)-\log x\big) u(y) \Pi(\overline u g)(y) \, dy$ and ${II}_\beta^{>1}=-\int_{-\infty}^x \log x \ u(y) \Pi(\overline u g)(y) \, dy$. For ${II}_\beta^{>1}$, with $\epsilon>0$ and $C_\epsilon^\prime=\sup_{x>1}\frac{\log x}{\langle x\rangle^\epsilon}<+\infty$,  we have
\begin{equation*}
    |{II}_\beta^{>1}|\leq \frac{|\log x|}{\langle x\rangle^\epsilon}\langle x\rangle^{-\frac 12-\epsilon} \| u \|_{L^2_{\frac12+2\epsilon}}\| \Pi(\overline u g)\|_{L^2}\leq C_\epsilon^\prime \langle x\rangle^{-\frac 12-\epsilon} \| u \|_{L^2_{\frac12+2\epsilon}}\| \Pi(\overline u g)\|_{L^2}.
\end{equation*}
Choosing $0<\epsilon<\frac12(s-\frac12)$ such that $s>\frac12+2\epsilon$,  we find ${II}_\beta^{>1}\in L^2$. Finally, we split ${II}_\alpha^{>1}$:
\begin{equation*}
    {II}_\alpha^{>1}=-\int_{-\infty}^{\frac x 2} \log\big(1-\frac yx\big) u(y) \Pi(\overline u g)(y) \, dy-\int_{\frac x 2}^x \log\big(1-\frac yx\big) u(y) \Pi(\overline u g)(y) \, dy=Y_1+Y_2,
\end{equation*}
where $Y_1=-\int_{-\infty}^{\frac x 2} \log\big(1-\frac yx\big) u(y) \Pi(\overline u g)(y) \, dy$ and $Y_2=-\int_{\frac x 2}^x \log\big(1-\frac yx\big) u(y) \Pi(\overline u g)(y) \, dy$. The term $Y_2$ satisfies 
\begin{equation*}
    |Y_2|\leq  \int_{\frac x 2}^x |\log\big(1-\frac yx\big)| |u(y) \Pi(\overline u g)(y)| \, dy\leq \bigg(\int_0^\frac12 |\log w|\d w\bigg)^\frac12\|u\|_{L^2_s}\| \Pi(\overline u g)\|_{L^\infty} x^{\frac12}\langle \frac x2 \rangle^{-s}.
\end{equation*}
Since $s>1$, we get $x^{\frac12}\langle  x \rangle^{-s}\in L^2$ and hence $Y_2\in L^2$. For $Y_1$,  we have 
\begin{equation*}
    |Y_1|=\bigg|\int_{-\infty}^{\frac x 2} \log\big(1-\frac yx\big) u(y) \Pi(\overline u g)(y) \, dy\bigg|\leq \int_{-\infty}^{\frac x2}|\frac y x|^p |u\Pi(\overline u g)| \d y\leq |x|^{-p} \| u\|_{L_p^2}\| \Pi(\overline u g)\|_{L^2},
\end{equation*}
where we choose $\frac12<p<1$ such that $Y_1\in L^2(1,+\infty)$. Hence, ${II}_\alpha^{>1}\in L^2(1, +\infty)$. The case $x<0$ is similar using the second expression in \eqref{3-x-inf-inf-x}. Having shown that all terms $I, II, III$ belong to $L^2$, we conclude that $g\in L^2$. And there exists a constant $C$ depending on $s$, $\| u\|_{L^2_s}$ and $\|\Pi(\overline u g)\|_{H^1}$ such that 
\begin{equation}\label{gsmacja}
    |g(x)|\leq C\langle x\rangle^{-\frac12-\delta},\quad 0<\delta<\min(p-\frac 12, \epsilon).
\end{equation}
\end{proof}
The claim $g\in L^2$ and $-i\partial_x g=u\Pi(\overline u g)$ implies that $g\in H^1$, and further, $g\in H^2$. 
Recall that  $\Pi$ is a bounded self-adjoint operator on $L^2$. Using Lemma \ref{xPi-Pix} and the fact $\overline u g\in L_1^2 $, we obtain that
\begin{align*}
&\int_{\mathbb{R}} xu(x) \, \overline{g(x)} \, \Pi(\overline{u}g)(x) \, \d x
= \langle \Pi(\overline{u}g), \, X(\overline{u}g) \rangle_{L^2}
= \langle \overline{u}g, \, \Pi X(\overline{u}g) \rangle_{L^2}
\\
= 
&
\int_{\mathbb{R}} \overline{u}g \, \overline{[\Pi, X](\overline{u}g)} \, \d x
+ \int_{\mathbb{R}} \overline{u}g \, X\overline{\Pi (\overline{u}g)} \, \d x 
= -\int_{\mathbb{R}} \overline{u}g \, \overline{\left( \frac{i}{2\pi} \, \widehat{\overline{u}g}(0^+) \right)} \, \d x
+ \langle X\overline{u}g, \, \Pi(\overline{u}g) \rangle_{L^2} 
\\
=
&
\frac{i}{2\pi} \, \langle g, u \rangle_{L^2} \, \overline{\langle g, u \rangle_{L^2}}
+ \langle X\overline{u}g, \, \Pi(\overline{u}g) \rangle_{L^2} 
= 
\frac{i}{2\pi} \, |\langle g, u \rangle_{L^2}|^2
+ \overline{\langle \Pi(\overline{u}g), \, X(\overline{u}g) \rangle_{L^2}}.
\end{align*}
This implies that
\begin{align*}
&\operatorname{Im} \int_{\mathbb{R}} (X u)(x) \, \overline{g(x)} \, \Pi(\overline{u}g)(x) \, \d x
=
\operatorname{Im} \langle \Pi(\overline{u}g), \, X(\overline{u}g) \rangle_{L^2} 
\\
= 
&
\frac{1}{2i} \left( \langle \Pi(\overline{u}g), \, X(\overline{u}g) \rangle_{L^2}
- \overline{\langle \Pi(\overline{u}g), \, X(\overline{u}g)} \rangle_{L^2} \right)
= \frac{1}{2i} \cdot \frac{i}{2\pi} \, |\langle g, u \rangle_{L^2}|^2
= \frac{1}{4\pi} \, |\langle g, u \rangle_{L^2}|^2.
\end{align*}
Then by Plancherel formula, we have
\begin{equation}\label{Re-int-overline-widehat{g}}
\begin{split}
\operatorname{Re} \int_0^\infty \overline{\widehat{g}(\xi)} \, \partial_\xi \left( \widehat{u \Pi(\overline{u}g)} \right)(\xi) \, \d \xi
&= \operatorname{Re} \left( -2\pi i \int_{\mathbb{R}} x  u(x) \, \overline{g(x)} \, (\Pi(\overline{u}g))(x) \, \d x \right) \\
&= 2\pi \, \operatorname{Im} \int_{\mathbb{R}} x  u(x) \, \overline{g(x)} \, (\Pi(\overline{u}g))(x) \, \d x
= \frac{1}{2} \, |\langle g, u \rangle_{L^2}|^2.
\end{split}
\end{equation}
Set \( h := \Pi(\overline{u}g) \in H^1_+ \hookrightarrow L^\infty \). Then it follows that
\begin{align*}
\| \widehat{u h} \|_{H^1}^2
&= \frac{1}{2\pi} \int_{\mathbb{R}} |\widehat{\widehat{u h}}(x)|^2 \, (1+x^2)  \d x 
=
2\pi \int_{\mathbb{R}} |u h(x)|^2 \, (1+x^2) \, dx 
\\
&
\leq
2\pi \, \| h \|_{L^\infty}^2 \, \| u \|_{L^2_1}^2
\leq 2\pi \, \| h \|_{L^\infty}^2 \, \| u \|_{L^2_s}^2 < +\infty.
\end{align*}
Consequently, we obtain 
\( \widehat {u \Pi(\overline{u}g)} \in H^1\).
Formula \eqref{tempered} yields 
\begin{align*}
\partial_\xi \left( \widehat{u \Pi(\overline{u}g)} \right)(\xi)
&= \partial_\xi \left( \widehat{-i \partial_x {g}} \right)
= \widehat{g} + \xi \, \partial_\xi \widehat{g}. 
\end{align*}
Then we obtain
\begin{equation} \label{mul-overline-widehat{g}}
\overline{\widehat{g}(\xi)} \, \partial_\xi \left( \widehat{u \Pi(\overline{u}g)} \right)(\xi)
= \overline{\widehat{g}(\xi)} \, \left( \widehat{g}(\xi) + \xi \, \partial_\xi \widehat{g}(\xi) \right)
= |\widehat{g}(\xi)|^2 + \overline{\widehat{g}(\xi)} \, \xi \, \partial_\xi \widehat{g}(\xi).
\end{equation}
From Plancherel formula and $-i\partial_x g=u\Pi(\overline u g)\in L_1^2$, it follows that
\begin{align*}
&
\int_0^\infty \overline{\widehat{g}(\xi)} \, \xi \, (\partial_\xi \widehat{g})(\xi) \, \d \xi
=
\int_0^\infty \overline{\widehat{g}(\xi)} \, (-i) \, \widehat{D X g}(\xi) \, \d\xi
= -2\pi i \int_{\mathbb{R}} (D X g) \, \overline{g} \, dx \quad 
\\
=
&
-2\pi \int_{\mathbb{R}} (\partial_x (xg)) \, \overline{g} \, dx
= -2\pi \big( \int_{\mathbb{R}} |{g}|^2 \, dx
+ \int_{\mathbb{R}} \overline{g(x)} \, x \, (\partial_x g)(x) \, dx\big).
\end{align*}
Then we take the real part of the above equation and integrate by part,
\begin{equation}\label{Re-int-over-hat-g-xi}
\begin{split}
&
\operatorname{Re} \int_0^\infty \overline{\widehat{g}(\xi)} \, \xi \, (\partial_\xi \widehat{g})(\xi) \, d\xi
=
(-2\pi) \operatorname{Re} \left( \int_{\mathbb{R}} |{g}|^2 \, dx
+ \int_{\mathbb{R}} \overline{g(x)} x (\partial_x g)(x) \, \d x \right) 
\\
=
&
-2\pi \| g \|_{L^2}^2 - 2\pi \operatorname{Re} \left( \int_{\mathbb{R}} x\overline{g(x)} \, (\partial_x g)(x) \, dx \right) 
=
-2\pi \| g \|_{L^2}^2 - \pi \int_{\mathbb{R}} x \ (\partial_x |g|^2) (x) \, \d x 
\\
=
&
-2\pi \| g \|_{L^2}^2 + \pi \| g \|_{L^2}^2
= -\pi \| g \|_{L^2}^2.
\end{split}
\end{equation}
By \eqref{Re-int-overline-widehat{g}}, \eqref{mul-overline-widehat{g}} and \eqref{Re-int-over-hat-g-xi},  we conclude that
\begin{align*}
\frac{1}{2} \, |\langle g, u \rangle_{L^2}|^2
&{=} \operatorname{Re} \int_0^\infty \overline{\widehat{g}(\xi)} \, \partial_\xi \left( \widehat{u \Pi(\overline{u}g)} \right)(\xi) \, d\xi 
=
\int_{\mathbb{R}} |\widehat{g}|^2 \, dx - \pi \| g \|_{L^2}^2 
= \pi \| g \|_{L^2}^2.
\end{align*}
So we have  $ |\langle g, u \rangle_{L^2}|^2 = 2\pi \| g \|_{L^2}^2$.
\end{proof}

  \begin{corollary}\label{tildeT0bi}
       Assume that $u\in H^1_+\cap H^1_s$ for some $s>1$ and $0$ is not an eigenvalue of $\mathbb{L}_u$. 	Then the operator $\mathrm{id}-\widetilde T_0$ is bijective from $L_{-s}^\infty$ to $L_{-s}^\infty$.
    \end{corollary}
    \begin{proof}
        By Fredholm alternative theorem and Lemma \ref{lemma3.2} and Proposition \ref{identity-prop}, we obtain the bijectivity of $\mathrm{id}-\widetilde T_0$.
    \end{proof}    
    \begin{lemma}\label{id-mod-bije}
  Assume that $u\in H^1_+\cap H_s^1$ for some $s>1$ and $0$ is not an eigenvalue of $\mathbb{L}_u$.  Then there exists a constant ${a}>0$ such that for every $k\in \C^+\cup \C^- \cup \{0\} \cup  \big((-\infty,0)\setminus \sigma_{pp}(\mathbb{L}_u)\big)$ and $|k|<{a}$, the operator $\mathrm{id}-\widetilde T_k$ is bijective from $L_{-s}^\infty$ to $L_{-s}^\infty$.
    \end{lemma}
    \begin{proof}
Note that the space $L^{\infty}_{-s}$ is a Banach space, then the the space $\mathcal{B}(L^{\infty}_{-s})$ is a Banach algebra. According to Theorem 10.12 of { \cite{RudinFA} }, the group $G(\mathcal{B}(L^{\infty}_{-s}))=\{A\in \mathcal{B}(L^{\infty}_{-s}): A \text{ is bijective} \}$  consisting of all bounded invertible operators on $L^{\infty}_{-s}$ is an open set in $\mathcal{B}(L^{\infty}_{-s})$. Recall that the operator $\mathrm{id}-\widetilde T_0$ is invertible by Corollary \ref{tildeT0bi}. So there exists a constant $r>0$ such that if an operator $A\in \mathcal{B}(L_{-s}^\infty)$ and $\|A-(\mathrm{id}-\widetilde T_0) \|_{L^{\infty}_{-s}}<r$, then  $A\in G(\mathcal{B}(L^{\infty}_{-s}))$. 
Let $a_2=\sqrt[\varepsilon]{\frac r C}$ and $a=\min(a_1, a_2)$ where $C$ and $a_1$ are given by Lemma \ref{lemma3.2}. For every $|k|<a$, we deduce 
\[
\| (\mathrm{id}-\widetilde T_k)-(\mathrm{id}-\widetilde T_0) \|_{\mathcal{B}(L_{-s}^\infty)}=\| \widetilde T_k-\widetilde T_0 \|_{\mathcal{B}(L_{-s}^\infty)}\leq C|k|^\varepsilon < Ca_2^\varepsilon < r.
\]
Therefore, for every $|k|<a$, the operator $\mathrm{id}-\widetilde T_k\in G(\mathcal{B}(L^{\infty}_{-s}))$. Then  $\mathrm{id}-\widetilde T_k$ is bijective.
    \end{proof}
    Similarly, we have the compactness for $\mathrm{id}-T_\lambda^\pm$.
    \begin{lemma}\label{id-mod-lam-bije}
       Assume that $u\in H^1_+\cap H_s^1$ for some $s>1$ and $0$ is not an eigenvalue of $\mathbb{L}_u$.  Then  there exists a constant $b>0$ such that for every  $\lambda<b$ and $\lambda\in \{0\}\cup ((0,+\infty)\setminus \sigma_{pp}(\mathbb{L}_u) )$, the operator $\mathrm{id}-\widetilde T_\lambda^\pm$ is bijective from $L_{-s}^\infty$ to $L_{-s}^\infty$.
    \end{lemma}

     Since $u\in L^2$ and $G_k\in L^2$, $M_k*u\in L^2*L^2\subset L^\infty \subset  L_{-s}^\infty$. Since $\pm i e^{i\lambda x} \mathds{1}_{x\geq 0}\in L^\infty$ and $\widetilde{G}_\lambda\in L^2$, we have $M_\lambda^\pm *u\in C_0 \subset L_{-s}^\infty$. Now we can define the modified Jost functions $\widetilde{\phi}_k$, $\widetilde \phi_\lambda^\pm$ and $\widetilde \psi_\lambda^\pm$.
       \begin{definition}\label{defwidetilde}
        Assume that $u\in H^1_+\cap H_s^1$ for some $s>1$ and $0$ is not an eigenvalue of $\mathbb{L}_u$.   The constants $a>0$ and $b>0$ are given in Lemma \ref{id-mod-bije} and \ref{id-mod-lam-bije}. For every  $|k|<a$ and  $k\in \C^+\cup \C^- \cup \{0\} \cup  \big((-\infty,0)\setminus \sigma_{pp}(\mathbb{L}_u)\big)$,     	for  every $\lambda\in 0\cup ((0,+\infty)\setminus \sigma_{pp}(\mathbb{L}_u) )$ and $|\lambda|<b$, we define 
\begin{align}
		&\widetilde{\phi}_k=(\mathrm{id}-\widetilde T_k)^{-1} (M_k*u) =M_{k}\ast(u+u\Pi(\overline{u}\widetilde{\phi}_k)), \label{4.1}\\
	&\widetilde{\phi}_\lambda^\pm=(\mathrm{id}-\widetilde T_\lambda^\pm)^{-1} (M_\lambda^\pm*u)=M_\lambda^\pm\ast(u+u\Pi(\overline{u}\widetilde{\phi}_\lambda^\pm)),\label{mod-phi-lamb-def}\\
    &\widetilde{\psi}_{\lambda}^\pm=(\mathrm{id}-\widetilde T_\lambda^\pm)^{-1} e^{ix\lambda}=e^{ix\lambda}+M_\lambda^{\pm}*\big(u\Pi(\overline u \widetilde{\psi}_{\lambda}^\pm )\big).\label{4.2}
	\end{align}
    \end{definition}
\begin{lemma}
 Assume that $u\in H^1_+\cap H_s^1$ for some $s>1$ and $0$ is not an eigenvalue of $\mathbb{L}_u$.   The constants $a>0$ and $b>0$ are given in Lemma \ref{id-mod-bije} and \ref{id-mod-lam-bije}. For every   $k\in \C^+\cup \C^- \cup \{0\}\cup  \big((-\infty,0)\setminus \sigma_{pp}(\mathbb{L}_u)\big)$ such that $|k|<a$,     	for  every $\lambda\in \{ 0\}\cup ((0,+\infty)\setminus \sigma_{pp}(\mathbb{L}_u) )$ such that $|\lambda|<b$, 
 we have
    \begin{equation}\label{tilderelation}
        \begin{split}
            &\phi_k= \widetilde \phi_k + l_k\int_\R (u+u\Pi (\overline u \widetilde \phi_ k))  \ \mathrm{d} x,\\
            &\phi_\lambda^\pm= \widetilde \phi_\lambda^\pm -l_\lambda\int_\R (u+u\Pi (\overline u \widetilde \phi_\lambda^\pm))  \ \mathrm{d} x,\\
             &\psi_\lambda^\pm = \widetilde \psi_\lambda^\pm-l_\lambda  \int_\R u\Pi (\overline u \widetilde \psi_\lambda^\pm) \ \mathrm{d} x.
        \end{split}
    \end{equation}
\end{lemma}
\begin{proof}
 By \eqref{4.1}, we have
 $
	\widetilde \phi_k =G_k\ast (u+u\Pi (\overline u \widetilde \phi_ k))- l_k\int_\R (u+u\Pi (\overline u \widetilde \phi_ k))  \ \mathrm{d} x$.
    Since the second term does not depend on $x$, and $\Pi(\overline u)=0$, we get 
    \[
    \widetilde \phi_k + l_k\int_\R (u+u\Pi (\overline u \widetilde \phi_ k))  \ \mathrm{d} x = G_k \ast \bigg(u+u\Pi\Big(\overline u  ( \widetilde \phi_k + l_k\int_\R (u+u\Pi (\overline u \widetilde \phi_ k))   \mathrm{d} x) \Big)  \bigg).
    \]
	Thanks to   the bijectivity of $\mathrm{id}-T_k$, 
	 we obtain 
	$
\phi_k= \widetilde \phi_k + l_k\int_\R (u+u\Pi (\overline u \widetilde \phi_ k))  \ \mathrm{d} x
	$.
From \eqref{mod-phi-lamb-def}, we have 
$
\widetilde \phi_\lambda^\pm =(G_\lambda^\pm +l_\lambda)* (u+u\Pi(\overline u \widetilde \phi_\lambda^\pm)).
$
Thus, we get 
\[
\widetilde \phi_\lambda^\pm-l_\lambda \int_\R (u+u\Pi(\overline u \widetilde \phi_\lambda^\pm)) \ \mathrm{d} x=G_\lambda^\pm * \bigg(u+u\Pi\Big(\overline u  ( \widetilde \phi_\lambda^\pm - l_\lambda\int_\R (u+u\Pi (\overline u \widetilde \phi_\lambda^\pm))   \mathrm{d} x) \Big)  \bigg).
\]
Thanks to the the bijectivity of $\mathrm{id}-T_\lambda^\pm$, we obtain
$
\phi_\lambda^\pm= \widetilde \phi_\lambda^\pm - l_\lambda\int_\R (u+u\Pi (\overline u \widetilde \phi_\lambda^\pm))  \ \mathrm{d} x.
$
In addition, it follows from \eqref{4.2} that 
$
    \widetilde \psi_\lambda^\pm  
    = e^{ix\lambda} + (G_\lambda^\pm +l_\lambda) * (u\Pi(\overline u \widetilde \psi_\lambda^\pm)).
$
Then we get
\[
 \widetilde \psi_\lambda^\pm-l_\lambda  \int_\R u\Pi (\overline u \widetilde \psi_\lambda^\pm) \ \mathrm{d} x=e^{ix\lambda} + G_\lambda^\pm * \bigg(  u\Pi \Big(   \overline u  \big(  \widetilde \psi_\lambda^\pm - l_\lambda \int_\R u\Pi (\overline u \psi_\lambda^\pm) \ \mathrm{d} x \big)  \Big)  \bigg).
\]
The bijectivity of $\mathrm{id}-T_\lambda^\pm$ yields  that 
		$
        \psi_\lambda^\pm = \widetilde \psi_\lambda^\pm+l_\lambda  \int_\R u\Pi (\overline u \widetilde \psi_\lambda^\pm) \ \mathrm{d} x.
		$
\end{proof}

\begin{proposition}\label{jost0asy}  
 Assume that $u\in H^1_+\cap H_s^1$ for some $s>1$ and $0$ is not an eigenvalue of $\mathbb{L}_u$.  Let $\widetilde\phi_k$, $\widetilde\phi_\lambda^\pm$ and $\widetilde\psi_\lambda^\pm$ be given by Definition \ref{defwidetilde}. Let $\phi_k$, $\phi_\lambda^\pm$ and $\psi_\lambda^\pm$ be given by Definition \ref{def3.7}, \ref{defphilam} and \ref{defpsilam}. 
Then we have

\begin{equation}\label{phi-k-asy}
\sup_{k \in \mathbb{C} \setminus (0,+\infty) \setminus \sigma_{pp}(\mathbb{L}_u)} \frac{\left\| \phi_k - c_1 \log k - \widetilde{\phi}_0 \right\|_{L^\infty_{-s}}}{\left| k^\varepsilon \log k \right|} < +\infty,
\end{equation}

\[
\sup_{\lambda \in [0,+\infty)\setminus \sigma_{pp}(\mathbb{L}_u)} \frac{\left\| \phi_\lambda^{\pm} - c_2 \log \lambda - \widetilde{\phi}_0^{\pm} \right\|_{L^\infty_{-s}}}{\left| \lambda^\varepsilon \log \lambda \right|} < +\infty,
\]

\[
\sup_{\lambda \in [0,+\infty)\setminus \sigma_{pp}(\mathbb{L}_u)} \frac{\left\| \psi_\lambda^{\pm} - c_3 \log \lambda - \widetilde{\psi}_0^{\pm} \right\|_{L^\infty_{-s}}}{\left| \lambda^\varepsilon \log \lambda \right|} < +\infty,
\]
where the constants are given by \small {$c_1= - \frac{1}{2\pi }\int_\R \left( u + u \Pi \left(  \overline{u} \widetilde{\phi}_0 \right) \right) dx$, $c_2=- \frac{1}{2\pi }\int_\R \left( u + u \Pi \left(  \overline{u} \widetilde{\phi}_0^\pm \right) \right) dx$ and $c_3=- \frac{1}{2\pi } \int_\R \left( u + u \Pi \left(  \overline{u} \widetilde{\psi}_0^\pm \right) \right) dx$. }
\end{proposition}

\begin{proof}
    According to Definition \ref{defwidetilde}, we have
    \begin{small}
    \[
    \widetilde\phi_k-\widetilde\phi_0=M_k*u+\widetilde T_k(\widetilde \phi_k)-M_0*u-\widetilde T_0(\widetilde\phi_0)=(M_k-M_0)*u+\widetilde T_k(\widetilde \phi_k)-\widetilde T_k(\widetilde \phi_0)+\widetilde T_k(\widetilde \phi_0)-\widetilde T_0(\widetilde \phi_0).
    \]
    \end{small}
    Then we obtain
    \begin{equation}
    (\mathrm{id}-\widetilde T_k)(\widetilde \phi_k-\widetilde \phi_0)=(M_k-M_0)*u+ (\widetilde T_k -\widetilde T_0)(\widetilde \phi_0)=\mathrm{I}+\mathrm{II},
    \end{equation}
    where $\mathrm{I}=(M_k-M_0)*u$ and $\mathrm{II}=(\widetilde T_k -\widetilde T_0)(\widetilde \phi_0)$. By Lemma 5.1 in \cite{16}, 
    for every $x\in \R$, there exist constants $a>0$ and $C>0$ such that for $|k|<a$ and  $\varepsilon =\min (1, \frac{1}{2}(s-\frac 1 2))$, we have
    \[
    \begin{aligned}
        &|\langle x \rangle^{-s} \mathrm{I}(x)|=   \bigg| \langle x \rangle^{-s} \int_{\mathbb{R}} (M_k(x-y)-M_0(x-y))u(y) \ \mathrm{d} y \bigg| \leq  \bigg| \langle x \rangle^{-s}  \int_{\mathbb{R}} C|k|^\varepsilon (1+|x-y|)^\varepsilon u(y) \ \mathrm{d} y \bigg| 
        \\
        \leq 
        &
        C|k|^\varepsilon \bigg|   \int_\mathbb{R}\langle y \rangle ^\varepsilon  \langle y \rangle ^{-s} \langle y \rangle ^{s} u(y) \ \mathrm{d} y \bigg|\leq C|k|^\varepsilon \| \langle y \rangle^{\varepsilon-s} \|_{L^2} \| u\|_{L^2_s} \leq \mathfrak{C}_1|k|^\varepsilon,
    \end{aligned}
    \]
   where $\mathfrak{C}_1=C\| \langle y \rangle^{\varepsilon-s} \|_{L^2} \| u\|_{L^2_s}$. Then 
   \(
   \|\mathrm{I} \|_{L^\infty_{-s}}\leq  \mathfrak{C}_1|k|^\varepsilon
   \). Using \eqref{TkT0}, we obtain that   
\begin{equation}
\| \mathrm{II}\|_{L_{-s}^\infty}\leq \left\| (\widetilde{T}_k - \widetilde{T}_0)(\widetilde{\phi}_0) \right\|_{L_{-s}^{\infty}} 
\leq \left\| \widetilde{T}_k - \widetilde{T}_0 \right\|_{\mathcal{B}(L_{-s}^{\infty})} \left\| \widetilde{\phi}_0 \right\|_{L_{-s}^{\infty}} 
\leq \mathfrak{C}_2|k|^{\varepsilon}
,
\end{equation}
where $\mathfrak{C}_2$ is a constant. 
Since $I - \widetilde{T}_k : L_{-s}^{\infty} \to L_{-s}^{\infty}$ is bijective, the open mapping theorem yields that $(I - \widetilde{T}_k)^{-1}$ is a bounded operator on $L^\infty_{-s}$. 
Then we have 
\begin{equation}\label{phik-0-k-varep}
   \left\| \widetilde{\phi}_k - \widetilde{\phi}_0 \right\|_{L_{-s}^{\infty}} \leq \left\| (I - \widetilde{{T}}_k)^{-1} \right\|_{\mathcal{B}(L_{-S}^{\infty} )} \|\mathrm{I}+\mathrm{II} \|_{L_{-s}^\infty} \leq    (\mathfrak{C}_1+\mathfrak{C}_2)|k|^{\varepsilon}
. 
\end{equation}
Formula (5.78) of  \cite{16} gives   
\begin{equation}
    l_k=-\frac{1}{2\pi}\log k-h(k).
\end{equation} 
The half-line  $[0, +\infty)$ is the branch cut, and the function $h$ is  analytic on $k$.  Formulas \eqref{tilderelation} and \eqref{phik-0-k-varep} yield that 
\[
\begin{aligned}
\phi_k(x) &= \widetilde{\phi}_k(x) + l_k \int \left( u + u \Pi\left( \overline{u} \widetilde{\phi}_k \right) \right) \ \mathrm{d}x
\\
&
= \widetilde{\phi}_0(x) + (\widetilde{\phi}_k - \widetilde{\phi}_0)(x) - \left( \frac{1}{2\pi } \log k + h(k) \right) \int_\R u(x) dx - \left( \frac{1}{2\pi } \log k + h(k) \right) \int_\R u \Pi \left(  \overline{u} \widetilde{\phi}_0 \right) dx 
\\
&\qquad \qquad- \left( \frac{1}{2\pi } \log k + h(k) \right) \int_\R u \Pi\left(  \overline{u} \left( \widetilde{\phi}_k - \widetilde{\phi}_0 \right) \right) dx \\
&= \widetilde{\phi}_0(x)  - \frac{1}{2\pi } \log k \int_\R \left( u + u \Pi \left(  \overline{u} \widetilde{\phi}_0 \right) \right) dx+ \mathcal{O}\left( k^\varepsilon \log k \right),\quad k\to 0.
\end{aligned}
\]
Here $\mathcal{O}$ holds in the sense of $L_{-s}^\infty$ norm. Then \eqref{phi-k-asy} follows.
\end{proof}

Now we investigate the spectral asymptotics at $k=\infty$. 

\begin{proposition}
   Assume that  $u\in H^1_+ \cap L_s^2 $   for some $s>\frac{1}{2}$. Let $\phi_k$ be given by Definition \ref{def3.7}. 
If   there exists $\alpha\in(0,\frac{\pi}{2})$ such that 
	\(
	|\mathrm{Im}\ k|\geq(\tan\alpha)\mathrm{Re}\ k
\)
(the shaded area in Fig. \ref{fig1}), then we have
\begin{equation}
\lim_{k \to \infty}\| \phi_k\|_{L^\infty}= 0.\label{417}
\end{equation}
\end{proposition}

 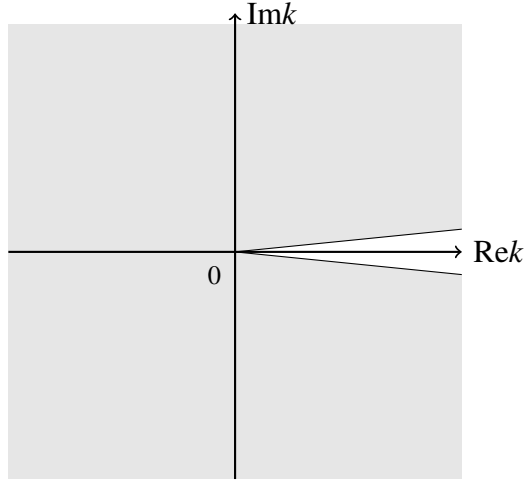
\begin{figure}
	\centering
	\begin{tikzpicture}[scale=1.5]
		
		\draw[->][thick](-2,0)--(2,0)

		 node[right]{$\mathrm{Re} k$};
		 
		 		\draw[->][thick](0,-2)--(0,2.1)
		 
		 node[right]{$\mathrm{Im} k$};

\node at (0,0) [below left=1pt] {\footnotesize $0$};
             
     \draw[-](0,0)--(2,0.2);
     \draw[-](0,0)--(2,-0.2);
     
     \pgfmathsetmacro{\angleA}{atan(0.2/2)} 
     \pgfmathsetmacro{\angleB}{atan(-0.2/2)} 
     
     \begin{scope}[on background layer]
     	\fill[gray!20] (-2,-2) rectangle (2,2);
     	
     	\fill[white] (0,0) 
     	-- (\angleA:3) arc (\angleA:\angleB:3) 
     	-- cycle;
     \end{scope}
	\end{tikzpicture}
	\caption{ \footnotesize The complex k-plane where the shaded area is the domain  away form the positive r eal line.}
	\label{fig1}
\end{figure}

\begin{proof}
Let $\xi\in \R^+$ and $k=a+ib$. When $a>0$, $b>0$ and $\mathrm{Im}\ k\geq(\tan\alpha)\mathrm{Re}\ k$ with $\alpha\in (0,\frac \pi 2)$,  we introduce a spherical coordinate
	\begin{align*}
		\xi=\rho \sin \omega,\quad
		a=\rho \cos\omega \cos \eta,\quad
		b=\rho \cos\omega \sin \eta,
		\end{align*}
	with $\omega\in [0,\frac \pi 2)$ and $\eta\in [\alpha, \frac \pi 2)$. Then we have 
	\begin{equation}\label{44.5}
	\frac{|\xi-k|^2}{\xi^2+|k|^2}=(\sin \omega- \cos \omega \cos \eta)^2+(\cos \omega\sin\eta)^2={1-\sin(2\omega)\cos \eta}
	\geq 1-\cos \alpha.
	\end{equation}
So we obtain 
   $ \lVert G_{k}\rVert_{L^{2}}=
   \frac{1}{\sqrt{2\pi}}\int_0^\infty \frac{1}{|\xi-k|^2}\  \mathrm{d} \xi\leq C(\alpha)\left( \int_0^\infty \frac{1}{(\xi+ |k|)^{2}}d\xi\right)^{\frac{1}{2}}=\frac{C(\alpha)}{\sqrt{|k|}},   $    
 where $C(\alpha)=\frac{2}{\sqrt{2\pi}(1-\cos \alpha)}$. Thus, from Young's inequality, we deduce that
    \begin{equation}\label{gk*u}
        \|G_k*u \|_{L^\infty}\leq \lVert G_{k}\rVert_{L^{2}} \| u \|_{L^2}\leq \frac{C(\alpha)}{\sqrt{|k|}}  \| u \|_{L^2}.
    \end{equation}
Therefore, for  every $f\in L^\infty$, we have 
\begin{equation} 
			\lVert T_{k}f\rVert_{L^{\infty}}\leq\lVert G_{k}\rVert_{L^{2}}\lVert u\Pi(\overline{u}f)\rVert_{L^{2}}\leq \lVert G_{k}\rVert_{L^{2}}\lVert u\rVert_{L^\infty}\lVert u\rVert_{L^2}\lVert f\rVert_{L^{\infty}}\leq \frac{\mathfrak{C}(\alpha,\| u\|_{H^{1}})}{\sqrt{|k|}}\lVert f \rVert_{L^{\infty}},
		\end{equation}
        where $\mathfrak{C}(\alpha,\| u\|_{H^{1}})$ is a constant depending on $\alpha$ and $\|u \|_{H^1}$.  
  When  $|k|$ is large, we have $\lVert T_{k}\rVert_{L^{\infty}\to L^{\infty}} \ll 1$.   By formula $\phi_k=(\mathrm{id}-T_k)^{-1}(G_k*u)$ and \eqref{gk*u}, we obtain \( \lim_{k \to \infty} \| \phi_k\|_{L^\infty}= 0\).
\end{proof}

\begin{proposition}\label{lam-inf-jost}
 Assume that  $u\in H^1_+ \cap L_s^2 $   for some $s>\frac{1}{2}$. Let  $\phi_\lambda^\pm$ and $\psi_\lambda^\pm$  be given by Definition  \ref{defphilam} and \ref{defpsilam}. Then we have 
 \begin{equation}
     \lim_{\lambda \to +\infty} \phi_\lambda^\pm(x)= 0,\quad 
     \lim_{\lambda \to +\infty} \big(\psi_\lambda^\pm(x)-e^{i\lambda x}\big)=0.
     \end{equation}
\end{proposition}
\begin{proof}
    Since $u\in L^1\cap L^2_s $ and $\phi_\lambda^+\in L_{-s}^\infty$, by Riemann-Lebesgue lemma,  we have 
     $\mathop{\rm lim}\limits_{\lambda\to +\infty}\mathcal{F}(u+u\Pi(\overline u \phi_\lambda^+)(\lambda)=0$. Then we obtain $\mathop{\rm lim}\limits_{\lambda\to +\infty}|ie^{i\lambda x}\mathds{1}_{R^+}*(u+u\Pi(\overline u \phi_\lambda^+))|=0$.
    Since $\| \widetilde{G}_\lambda\|_{L^2}=\frac{1}{\sqrt{2\pi \lambda}}$ and $u\in H^1$, we get 
    \begin{equation*}
    \begin{split}
         \| \widetilde{G}_\lambda*(u+u\Pi(\overline u \phi_\lambda^+))\|_{L^\infty} 
        \leq \| \widetilde{G}_\lambda\|_{L^2}\|u+u\Pi(\overline u \phi_\lambda^+)\|_{L^2}\leq 
        \frac{1}{\sqrt{2\pi \lambda}}(\|u \|_{L^2}+\|u\|_{L^\infty}\|u\|_{L^2}\|\phi_\lambda^+\|_{L^\infty}).
    \end{split}
\end{equation*}
Then from $\phi_\lambda^+=(ie^{i\lambda x}\mathds{1}_{R^+}-\widetilde{G}_\lambda)*(u+u\Pi(\overline u \phi_\lambda^+))$, we obtain $\mathop{\rm lim}\limits_{\lambda \to +\infty} \phi_\lambda^+= 0$.
  By $\psi_\lambda^+=e^{i\lambda x}+(ie^{i\lambda x}\mathds{1}_{R^+}-\widetilde{G}_\lambda)*(u\Pi(\overline u \psi_\lambda^+))$, we obtain $\mathop{\rm lim}\limits_{\lambda \to +\infty} \big(\psi_\lambda^\pm(x)-e^{i\lambda x}\big)=0$.
\end{proof}
 By the definitions \eqref{gamma} and \eqref{3.5} of the scattering coefficients,  Riemann-Lebesgue lemma and Proposition  \ref{jost0asy}, we obtain the following lemma:
\begin{proposition}
 Assume that  $u\in H^1_+ \cap L_s^2 $   for some $s>\frac{1}{2}$. Let  $\phi_\lambda^\pm$ and $\psi_\lambda^\pm$  be given by Definition  \ref{defphilam} and \ref{defpsilam}. Let $\Gamma(\lambda)$ aand $\beta(\lambda)$ be given by \eqref{gamma} and \eqref{3.5}. Then we have 
 \begin{equation}
 \begin{split}
     &\lim_{\lambda \to +\infty} \Gamma(\lambda)= 1, \quad \lim_{\lambda \to +\infty} \beta(\lambda)= 0,
     \\
     &\lim_{\lambda\to 0^+}\Gamma(\lambda)=1, \quad \lim_{\lambda\to 0^+}\beta(\lambda)=0.
     \end{split}
 \end{equation}
\end{proposition}

\section{The invariant subspace}\label{sec5}
In this section, we will prove the following theorem:
\begin{theorem}\label{invariance}
Let \( u \in C^0((T^-, T^+); H_+^4) \)  be the unique solution to the 
CMDNLS equation \eqref{21.1}
with initial datum \( u(0) = u_0 \in H_+^4 \). 
If \( u_0, \partial_x u_0, \partial_x^2 u_0 \in L^2_1 \), then 
we have 
\begin{equation}
u(t), \partial_x u(t), \partial_x^2 u(t) \in L_1^2,\quad \forall t\in (T^-, T^+).
\end{equation}
\end{theorem}

Before proving Theorem \ref{invariance}, we establish two lemmas that provide an equivalence between different regularity conditions for the function $u$. Specifically, we show that for any $u\in H^2_+$
 and $s>0$, the conditions $u, \partial_x u, \partial_x^2u\in L_s^2$
 and $u, \mathbb{L}_u(u), \mathbb{L}_u^2(u)\in L_s^2$
 are equivalent. This equivalence is useful for the proof of the theorem.

\begin{lemma}\label{utoLu}
Assume that \( u \in H_+^2 \) such that
$
u,\ \partial_x u,\ \partial_x^2 u \in L_s^2
$ for some $s > 0$.
Then we have
$
\mathbb{L}_u(u),\ \mathbb{L}_u^2(u) \in L_s^2.
$
\end{lemma}
\begin{proof}
Note that $\Pi(|u|^2) \in \Pi(H^2) = H_+^2 \hookrightarrow L^\infty$, it follows that 
\[
\mathbb{L}_u(u) 
= -i \partial_x u - u \Pi(|u|^2) \in L_s^2\cap H^1_+.
\]
Define $g:=\mathbb{T}_u \mathbb{T}_{\bar{u}}(u) = u \Pi(|u|^2) \in L_s^2 \cap H_+^2$. Consequently, We have $\Pi(\bar{u} g) \in   \Pi(H^2) = H_+^2 \hookrightarrow L^\infty$. 
Given that 
$\Pi(\bar{u} \partial_x u)  \subset \Pi(H^1) = H_+^1 \hookrightarrow L^\infty$, it follows that $u \Pi(\bar{u} \partial_x u) \in  L_s^2$. 
 Furthermore, from  
$\Pi(|u|^2) \in H_+^2$, we deduce that $\partial_x \Pi(|u|^2) \in H_+^1 \hookrightarrow L^\infty$. This implies that 
$(\partial_x u) \Pi(|u|^2) \in L_s^2$ and $u \partial_x \Pi(|u|^2) \in L_s^2$. 
Then computing $ \mathbb{L}_u^2(u)$ explicitly yields 
\begin{equation}\label{Lu2u}
    \begin{split}
        \mathbb{L}_u^2(u) &=-\partial_x^2 u + i u \Pi(\bar{u} \partial_x u) + u \Pi(\bar{u} g) + i \partial_x\left( u \Pi(|u|^2) \right)
        \\
        &= -\partial_x^2 u + i u \Pi(\bar{u} \partial_x u) + u \Pi(\bar{u} g) + i (\partial_x u) \Pi(|u|^2) + i u \partial_x \Pi(|u|^2) \in L_s^2 .
    \end{split}
\end{equation}
\end{proof}

\begin{lemma}\label{Lutou}
Assume that \( u \in H_+^2 \) such that
$
u,\ \mathbb{L}_u(u),\ \mathbb{L}_u^2(u) \in L_s^2
$ for some $s>0$. 
Then we have 
$
\partial_x u,\ \partial_x^2 u \in L_s^2.
$
\end{lemma}

\begin{proof}
From the definition \eqref{21.4} and the given fact $\mathbb{L}_u(u)\in L^2_s$, it follows that  
$$
\mathbb{L}_u(u) 
= -i \partial_x u - u \Pi(|u|^2) \in L_s^2. 
$$
Since 
$u \Pi(|u|^2) \in L_s^2 \cdot L^\infty \subset L_s^2,$
 we have 
\[
 -i \partial_x u = \mathbb{L}_u(u) + u \Pi(|u|^2) \in L_s^2.
\]
Set  $h:=-i \partial_x u - u \Pi(|u|^2) \in H_+^1$.  Formula  \eqref{Lu2u} and $\mathbb{L}_u^2(u)\in L^2_s$ imply  that  
$$\mathbb{L}_u^2(u) 
= -\partial_x^2 u + i (\partial_x u) \Pi(|u|^2) + i u \partial_x \Pi(|u|^2) - u \Pi(\bar{u} h) \in L_s^2.$$
Given $u\in H^2_+$, 
 we have  $\Pi(|u|^2) \in H_+^2$ and consequently $ \partial_x \Pi(|u|^2) \in H_+^1 \hookrightarrow L^\infty$. This implies $(\partial_x u) \Pi(|u|^2) \in  L_s^2 $ and $u \partial_x \Pi(|u|^2) \in  L_s^2$. Moreover, because $h\in H^1_+$, we obtain  $\Pi(\bar{u} h) \in \Pi(H^1) = H_+^1 \hookrightarrow L^\infty$, and hence, $u \Pi(\bar{u} h) \in  L_s^2$. 
So we conclude that
\[
\partial_x^2 u = -\mathbb{L}_u^2(u) + i (\partial_x u) \Pi(|u|^2) + i u \partial_x \Pi(|u|^2) - u \Pi(\bar{u} h)
\in L_s^2.
\]
\end{proof}

Assume that \( u \in C^0\left( (T^-, T^+); H_+^2 \right) \) is the solution to the  CMDNLS equation \eqref{21.1}. Recall $\mathbb{B}_u$ is given by \eqref{Budef}.
Let \( W \in C^1\left( (T^-, T^+); \mathcal{B}_\C(L_+^2) \right) \) denote the unique solution to
\begin{equation}\label{UniPropaW'=BW}
    \partial_t W(t) = \mathbb{B}
    _{u(t)} W(t) \in \mathcal{B}_{\mathbb{C}}(L_+^2),\quad W(0) = \text{id}_{L_+^2}. 
\end{equation}Let $W(t)^*$ denote the $L^2_+$-adjoint of the operator $W(t)$, $\forall t\in (T^-, T^+)$. 
Since $\mathbb{B}_{u(t)}^*=-\mathbb{B}_{u(t)}$, we have 
\begin{equation}\label{dtWtstar=-BWst}
    \partial_t W(t)^*=-W(t)^*\mathbb{B}_{u(t)} \quad \forall t\in (T^-, T^+).
\end{equation}
Furthermore, we have
\(
\frac{\mathrm{d}}{\mathrm{d} t}(W(t)^*W(t))=0,
\)
and 
$W(t)W(t)^*$ solves the differential equation 
\(
\partial_t Y(t)=[\mathbb{B}_{u(t)},  Y(t)].
\)
The above differential equation  admits  the uniqueness of the solution. So  we obtain
\begin{equation}
    W(t)W(t)^* =W(t)^* W(t)=\mathrm{id}_{L^2_+}, \quad\forall t \in (T^-, T^+).
\end{equation}
So $W(t)$ is a unitary operator on $L^2_+$. In other words, we have
\begin{align}\label{unit}
	W(t)^*=W(t)^{-1},\quad 
	\|   W(t) \phi  \|_{L^2}= \|  \phi  \|_{L^2}, \quad \forall \phi\in L^2_+.
	\end{align}
According to formulas \eqref{UniPropaW'=BW} and \eqref{dtWtstar=-BWst},   it follows  from the Lax pair structure \eqref{LaxEq} that
\begin{equation}\label{uniequiLL2}
\begin{split}      
&\mathbb{L}_{u(t)} = W(t) \mathbb{L}_{u_0} W(t)^* : H_+^1 \to L_+^2 , \\
&\mathbb{L}_{u(t)}^2 = W(t) \mathbb{L}_{u_0}^2 W(t)^* : H_+^2 \to L_+^2.
\end{split}
\end{equation}

\begin{lemma}\label{gronwall}
Assume that \( u \in C^0((T^-, T^+), H^2_+) \) solves CMDNLS equation \eqref{21.1}
with initial datum \( u(0) = u_0 \in H^2_+ \), for some \( T^- < 0 < T^+ \).
Assume that 
\( \Psi \in C^0((\mathfrak{T}^-, \mathfrak{T}^+), H^2_+) \) solves 
\begin{equation}\label{dtpsiEq=iDelPsi+2u}
    \partial_t \Psi(t) = i\partial_x^2 \Psi(t) + 2u(t) \partial_x \Pi(\overline{u(t)} \Psi(t))
\end{equation}
with initial datum \( \Psi(0) = \Psi_0 \in H^2_+ \), for some $\mathfrak{T}^-\in (T^-, 0)$ and $\mathfrak{T}^+\in (0, T^+)$.
If $\   \Psi_0, u_0 \in L_1^2 $,
then
\begin{itemize}
    \item \( \Psi(t) \in L_1^2 \) for any \( t \in (\mathfrak{T}^-, \mathfrak{T}^+) \).
    \item \( u(\tau) \in L_1^2 \) for any \( \tau \in (T^-, T^+) \).
\end{itemize}

\end{lemma}

\begin{proof}

We choose a cut-off function  \( \chi \in C_c^\infty(\mathbb{R}; \mathbb{R}) \) such that
   \( 0 \leq \chi \leq 1 \),
     \( \chi = 1 \) on \( [-1, 1] \),      \( \mathrm{supp}\  \chi \subset [-2, 2] \),
    and \( \chi \) decreases on \( \mathbb{R}_+ = [0, +\infty) \). Set $\rho(x) := x\chi(x) \geq 0$.
For any \( R > 1 \), we define \[ I(R, t)=I_{\Psi}(R,t) := \int_{\R} \left| \chi\left(\frac{x}{R}\right) \right|^2 x^2 |\Psi(t, x)|^2 \d x ,\]
and  \( \rho_R(x) := R\rho\left(\frac{x}{R}\right) = x\chi\left(\frac{x}{R}\right) \), then \( \operatorname{supp} \rho_R \subset [-2R, 2R] \).
So we have
\[
\begin{aligned}
\partial_t I(R, t) 
&= 2\operatorname{Re} \int_{\R} x^2 \left| \chi\left(\frac{x}{R}\right) \right|^2 \overline{\Psi}(t, x) \partial_t \Psi(t, x) \d x 
=2\operatorname{Re} \int_{\R}  (\left| \rho_R(x) \right|^2  \partial_t \Psi(t, x))\overline{\Psi(t, x)} \d x\\
&= 2\operatorname{Re} \langle \left| \rho_R(x) \right|^2 \partial_t \Psi(t, x),   \Psi(t, x) \rangle_{L^2},
\end{aligned}
\]
Since \( \partial_t \Psi(t) = i\partial_x^2 \Psi(t) + 2u(t)\partial_x \Pi\left( \overline{u(t)} \Psi(t) \right) \), we have
\begin{equation}\label{partialtIA1A2}
\begin{split}
\partial_t I(R, t) 
&= -2\operatorname{Im}  \langle \left| \rho_R \right|^2 \partial_x^2 \Psi(t),  { \Psi(t) }\rangle_{L^2} + 4\operatorname{Re} \langle  \left| \rho_R \right|^2 u(t)\partial_x \Pi\left( \overline{u(t)} \Psi(t) \right) , { \Psi(t) }\rangle_{L^2} = (A_1 + A_2)(t).
\end{split}
\end{equation}
where $A_1(t):=-2\operatorname{Im}  \langle \left| \rho_R \right|^2 \partial_x^2 \Psi(t),  { \Psi(t) }\rangle_{L^2} $ and $A_2(t):=4\operatorname{Re} \langle  \left| \rho_R \right|^2 u(t)\partial_x \Pi\left( \overline{u(t)} \Psi(t) \right) , { \Psi(t) }\rangle_{L^2}$.
Note \( |\rho_R|^2 \Psi \in H^2\), then \(\partial_x(|\rho_R|^2 \Psi) \in H^1 \).
As a consequence, we have
\begin{align*}
&A_1
= -2\cdot\frac{1}{2i}\left(\left\langle |\rho_R|^2\partial_x^2\Psi,\,\Psi\right\rangle_{L^2} - \overline{\left\langle |\rho_R|^2\partial_x^2\Psi,\,\Psi\right\rangle_{L^2}}\right) 
= i\left\langle |\rho_R|^2\partial_x^2\Psi,\,\Psi\right\rangle_{L^2} - i\langle |\rho_R|^2\Psi,\,\partial_x^2\Psi\rangle_{L^2},\\
&= i\left\langle |\rho_R|^2\partial_x^2\Psi,\,\Psi\right\rangle_{L^2} - i\left\langle \partial_x^2(|\rho_R|^2\Psi),\,\Psi\right\rangle_{L^2}= i\left\langle \left[|\rho_R|^2,\,\partial_x^2\right]\Psi,\,\Psi\right\rangle_{L^2}.
\end{align*}
For every $f,\,g\in H^2 $ and $\rho_R\geq 0$ and $\rho_R\in C_c^\infty$, we deduce that 
\begin{equation*}\small
\left\langle \left[\rho_R,\,\partial_x^2\right]f,\,g\right\rangle_{L^2} 
=\int_{\mathbb{R}} \partial_x^2f \cdot \overline{\rho_R g} \d x- \int_{\mathbb{R}} \rho_R\cdot f \overline{\partial_x^2 g} \d x
=
 \int_{\mathbb{R}} f\left(\partial_x^2(\overline{\rho_R g}) - \rho_R(\overline{\partial_x^2 g})\right)
= \left\langle f,\,\left[\partial_x^2,\, \rho_R\right]g\right\rangle_{L^2}.
\end{equation*}
So $\left[\rho_R,\,\partial_x^2\right]$ is an unbounded anti-symmetric operator on $L^2$. 
Recall that 
\[
\left[\rho_R^2,\,\partial_x^2\right] = \rho_R\left[\rho_R,\,\partial_x^2\right] + \left[\rho_R,\,\partial_x^2\right]\rho_R: H^2\to L^2.
\] 
\[
\left[\rho_R,\,\partial_x^2\right] = \left[\rho_R,\,\partial_x\right]\partial_x + \partial_x\left[\rho_R,\,\partial_x\right]:H^2\to L^2.
\]
Then the following identity holds,
\[
\begin{aligned}
   &A_1
  =i\left\langle \rho_R \left[ \rho_R, \partial_x^2 \right] \Psi, \Psi \right\rangle_{L^2} + i\left\langle \left[ \rho_R, \partial_x^2 \right] \rho_R \Psi, \Psi \right\rangle_{L^2} 
   =
   i\left\langle \left[ \rho_R, \partial_x^2 \right] \Psi, \rho_R \Psi \right\rangle_{L^2} + i\left\langle \rho_R \Psi, \left[ \partial_x^2, \rho_R \right] \Psi \right\rangle_{L^2}
   \\=&
  2\operatorname{Re}\left( i\left\langle \left[ \rho_R, \partial_x^2 \right] \Psi, \rho_R \Psi \right\rangle_{L^2} \right)
   =
-2\operatorname{Im}\left\langle \left[ \rho_R, \partial_x \right] \partial_x \Psi, \rho_R \Psi \right\rangle_{L^2} - 2\operatorname{Im}\left\langle \partial_x \left[ \rho_R, \partial_x \right] \Psi, \rho_R \Psi \right\rangle_{L^2}.
\end{aligned}
\]
For any $R>1$, we have 
\[
\| \partial_x \rho_R \|_{L^1} = \| \partial_x \rho \|_{L^1} \cdot R, \quad \| \partial_x^3 \rho_R \|_{L^1} = R^{-1} \| \partial_x^3 \rho \|_{L^1}, \quad 
 \| \partial_x^5 \rho_R \|_{L^1} = R^{-3} \| \partial_x^5
\rho \|_{L^1}.
 \]

By Lemma 4.13 in  \cite{SunCMP2021}, there exist two universal constants $ \ C_1, \ C_2 > 0$ such that 
\[
\begin{aligned}
&\left\| \left[ \partial_x, \rho_R \right] g \right\|_{L^2}^2 = \left\| \left[ \rho_R, \partial_x \right] g \right\|_{L^2(\mathbb{R}; \mathbb{C})}^2 \leq C_1 \left( \| \partial_x \rho_R \|_{L^1} \| \partial_x^3 \rho_R \|_{L^1} \right)^{\frac{1}{2}} \| g \|_{L^2}, \\
&\left\| \partial_x \left[ \partial_x, \rho_R \right] h \right\|_{L^2}^2 \leq C_2 \left( \| \partial_x \rho_R \|_{L^1} \| \partial_x^3 \rho_R \|_{L^2} \right)^{\frac{1}{2}} \| \partial_x h \|_{L^2} + \left( \| \partial_x \rho_R \|_{L^1} \| \partial_x^5 \rho_R \|_{L^1} \right)^{\frac{1}{2}} \| h \|_{L^2},
\end{aligned}
\]
for any $ g \in H^1$ and  $ h \in H^1$. 
Therefore, we obtain 
\begin{equation}\label{A1}
\small
\begin{aligned}
    \left| A_1  \right| 
    \leq 
    & 2\left| \left\langle [\partial_x, \rho_R] \partial_x \Psi, \rho_R \Psi \right\rangle_{L^2} \right| + 2\left| \left\langle \partial_x [\partial_x, \rho_R] \Psi, \rho_R \Psi \right\rangle_{L^2} \right|
    \\
    \leq & 2\bigl\| \rho_R \Psi \bigr\|_{L^2} \cdot C_1 \left( \| \partial_x \rho_R \|_{L^1} \| \partial_x^3 \rho_R \|_{L^1} \right)^{\frac{1}{2}} \| \partial_x \Psi \|_{L^2} 
    \\  &\quad  +2\bigl\| \rho_R \Psi \bigr\|_{L^2} C_2 \left( \left( \| \partial_x \rho_R \|_{L^1} \| \partial_x^3 \rho_R \|_{L^1} \right)^{\frac{1}{2}} \| \partial_x \Psi \|_{L^2} + \left( \| \partial_x \rho_R \|_{L^1} \| \partial_x^5 \rho_R \|_{L^1} \right)^{\frac{1}{2}} \| \Psi \|_{L^2} \right)
\\ \leq &  2(C_1+C_2)\bigl\| \rho_R \Psi \bigr\|_{L^2} \left( \| \partial_x \rho \|_{L^1} \| \partial_x^3 \rho \|_{L^1} \right)^{\frac{1}{2}} \| \partial_x \Psi \|_{L^2} + \frac{2C_2}{R} \bigl\| \rho_R \Psi \bigr\|_{L^2} \left( \| \partial_x \rho \|_{L^1} \| \partial_x^5 \rho \|_{L^1} \right)^{\frac{1}{2}} \| \Psi \|_{L^2}.
\end{aligned}
\end{equation}
Recall that
\(
I(R, t) = \int_{\mathbb{R}} \left| \chi\left( \frac{x}{R} \right) \right|^2 x^2 \left| \Psi(t, x) \right|^2 \d x = \bigl\| \rho_R \Psi(t) \bigr\|_{L^2}^2
\),
we define
\[
J(R,t) := \left. I_{\Psi}(R,t) \right|_{\Psi=u} = \int_{\mathbb{R}} \left| \chi\left(\frac{x}{R}\right) \right|^2 x^2 |u(t,x)|^2 \d x  = \|\rho_R u(t)\|_{L^2}^2.
\]
According to formula \eqref{partialtIA1A2}, for every $t\in (T^-, T^+)$, we have
\[
\partial_t J(R,t) = -2\operatorname{Im}\left\langle |\rho_R|^2 \partial_x^2 u(t), u(t) \right\rangle_{L^2} + 4\Re \left\langle |\rho_R|^2 u(t) \partial_x \Pi(|u(t)|^2), u(t) \right\rangle_{L^2}= B_1(t) + B_2(t)
\]
with 
\(
B_1(t) := -2\operatorname{Im}\left\langle |\rho_R|^2 \partial_x^2 u(t), u(t) \right\rangle_{L^2}
\)
and 
\(
B_2(t) := 4\Re \left\langle |\rho_R|^2 u(t) \partial_x \Pi(|u(t)|^2), u(t) \right\rangle_{L^2}.
\)
From \eqref{A1},  we have
\[
\begin{split}
|B_1| 
 \leq 2(C_1+C_2)\|\rho_R u\|_{L^2} \left( \|\partial_x \rho\|_{L^1} \|\partial_x^3 \rho\|_{L^1} \right)^{\frac{1}{2}} \|\partial_x u(t)\|_{L^2}
+ \tfrac{2C_2}{R} \|\rho_R u\|_{L^2} \left( \|\partial_x \rho\|_{L^1} \|\partial_x^5 \rho\|_{L^1} \right)^{\frac{1}{2}} \|u\|_{L^2}.
\end{split}
\]
Since  $u \in H^2_+$, we have $|u|^2  \in H^2$. By \eqref{inequalityCspri} and \eqref{ineqC1pripri},
 we have 
 \[
 \begin{split}
 &|B_2| \leq 4 \|\rho_R u \partial_x \Pi(|u|^2)\|_{L^2} \|\rho_R u\|_{L^2}
 \leq 
 4 C_1'' \|\rho_R u\|_{L^2}^2 \|u^2\|_{H^2}\ \leq 4C_2' C_1'' \|\rho_R u\|_{L^2}^2 \|u\|_{H^2}^2.
 \end{split}
 \]
For all $R>1$, we have
\[\small
\begin{split}
 |B_1+B_2| 
&\leq 2(C_1+C_2)\|\rho_R u\|_{L^2} \left( \|\partial_x \rho\|_{L^1} \|\partial_x^3 \rho\|_{L^1} \right)^{\frac{1}{2}} \|\partial_x u\|_{L^2}
\\ &\hspace{1cm}  +
 \frac{2C_2}{R} \|\rho_R u\|_{L^2} \left( \|\partial_x \rho\|_{L^1} \|\partial_x^5 \rho\|_{L^1} \right)^{\frac{1}{2}} \|u\|_{L^2} +4C_2' C_1'' \|\rho_R u\|_{L^2}^2 \|u\|_{H^2}^2
\\ & \leq 
\left(2+4C_1'' C_2' \|u\|_{H^2}^2 \right) \|\rho_R u\|_{L^2}^2
+ (C_1+C_2)^2 \|\partial_x \rho\|_{L^1} \left( \|\partial_x^3 \rho\|_{L^1} + \|\partial_x^5 \rho\|_{L^1} \right) \left( \|u\|_{L^2}^2 + \|\partial_x u\|_{L^2}^2 \right).
\end{split}
\]
Now we turn back to $A_2$  in \eqref{partialtIA1A2}. From \eqref{ineqC1pripri} and \eqref{inequalityCspri}, we get
\[\small
\begin{aligned}
|A_2| 
\leq 
4 \| \rho_R u\|_{L^2} \| \partial_x \Pi(\overline{u}\Psi)\|_{L^\infty} \|\rho_R \Psi\|_{L^2} 
\leq
4 C_1'' C_2' \|\rho_R u\|_{L^2} \|\rho_R \Psi\|_{L^2} \|u\|_{H^2} \|\Psi\|_{H^2}.
\end{aligned}
\]
Formula \eqref{A1} implies that
\[
\begin{aligned}
|A_1| &
\leq 2\|\rho_R \Psi\|_{L^2}^2 + (C_1+C_2)^2 \|\partial_x \rho\| _{L^1} \left(\| \partial_x^3 \rho\|_{L^1} + \| \partial_x^5 \rho\| _{L^1}\right) \| \Psi \| _{H^1}^2.
\end{aligned}
\]
Thus, we have
\[
\small
\begin{aligned}
&|A_1+A_2| 
\leq 3 \| \rho_R \Psi\|_{L^2}^2 + 4 (C_1'' C_2')^2 \|u\|_{H^2}^2 \|\Psi\|_{H^2}^2 \|\rho_R u\|_{L^2}^2 
 + (C_1+C_2)^2 \| \partial_x \rho\|_{L^1} \left(\|\partial_x^3 \rho\|_{L^1} + \|\partial_x^5 \rho\|_{L^1}\right) \|\Psi\|_{H^1}^2.
\end{aligned}
\]
Fix an arbitary compact interval $[T_1, T_2]\subset (\mathfrak{T}^-, \mathfrak{T}^+)\subset (T^-, T^+)$ such that $T_1<0<T_2$. 
Set 
\[
\begin{aligned}
K(R,t) &:= I(R,t) + J(R,t) 
= \int_{\mathbb{R}} \left|\chi\left(\frac{x}{R}\right)\right|^2 x^2 \left(|\Psi(t,x)|^2 + |u(t,x)|^2\right) dx \,,\quad \forall R>1.
\end{aligned}
\]
Then we have
\[ 
\small
\begin{aligned}
|\partial_t K(R,t) |
&\leq 3 \|\rho_R \Psi(t)\|_{L^2}^2 + 4 (C_1'' C_2')^2 \|u(t)\|_{H^2}^2 \|\Psi(t)\|_{H^2}^2 \|\rho_R u(t)\|_{L^2}^2 
\\
 & \quad \quad
+ (C_1+C_2)^2 \|\partial_x \rho\|_{L^1} \left(\|\partial_x^3 \rho\|_{L^1} + \|\partial_x^5 \rho\|_{L^1}\right) \|\Psi(t)\|_{H^1}^2 
 + \|\rho_R u(t)\|_{L^2}^2 \left(2 + 4 C_1'' C_2' \|u(t)\|_{H^2}^2\right) 
 \\  &   \quad \quad
 + (C_1+C_2)^2 \|\partial_x \rho\|_{L^1} \left(\|\partial_x^3 \rho\|_{L^1} + \|\partial_x^5 \rho\|_{L^1}\right) \|u(t)\|_{H^1}^2 
\leq AK(R,t)+\gamma,
\end{aligned}
\]
where
\[
\small
\begin{aligned}
A &= A(T_1, T_2, u, \Psi):= 5 + 4 C_1'' C_2' \sup_{T_1 \leq t \leq T_2} \|u(t)\|_{H^2}^2 + 4 (C_1'' C_2')^2 \left(\sup_{T_1 \leq t \leq T_2} \|u(t)\|_{H^2}^2\right) \sup_{T_1 \leq t \leq T_2} \|\Psi(t)\|_{H^2}^2, \\
\gamma &=\gamma(T_1, T_2, u, \Psi, \chi):= (C_1+C_2)^2 \|\partial_x \rho\|_{L^1} \left(\|\partial_x^3 \rho\|_{L^1} + \|\partial_x^5 \rho\|_{L^1}\right) \left(\sup_{T_1 \leq t \leq T_2} \|\Psi(t)\|_{H^1}^2 + \sup_{T_1 \leq t \leq T_2} \|u(t)\|_{H^1}^2\right) .
\end{aligned}
\]
According to Gronwall's inequality, we have 
\[
\begin{aligned}
K(R,t) \leq e^{A|t|} \left(K(R,0) + \frac{\gamma(T_1, T_2, u, \Psi, \chi)}{A(T_1, T_2, u, \Psi)}\right) \,,\quad \forall t \in [T_1, T_2].
\end{aligned}
\]
Since 
$
K(R,0) 
= \int_{\mathbb{R}} x^2 \left|\chi\left(\frac{x}{R}\right)\right|^2 \left(|\Psi_0(x)|^2 + |u_0(x)|^2\right) \d x \leq \|\Psi_0\|_{L_1^2}^2 + \|u_0\|_{L_1^2}^2 < +\infty.
$
Then we have
\[
K(R,t) \leq e^{A|t|} \left(\|\Psi_0\|_{L_1^2}^2 + \|u_0\|_{L_1^2}^2 + \frac{\gamma(T_1, T_2, u, \Psi, \chi)}{A(T_1, T_2, u, \Psi)}\right).
\]
By Lebesgue monotone convergence theorem,  we have
\[
\begin{aligned}
& \int_{\mathbb{R}} x^2 \left(|\Psi(t,x)|^2 + |u(t,x)|^2\right) \d x =\lim_{R \to +\infty} K(R,t) 
\leq e^{A|t|} \left(\|\Psi_0\|_{L_1^2}^2 + \|u_0\|_{L_1^2}^2 + \frac{\gamma}{A}\right) < +\infty.
\end{aligned}
\]

So  we have
\(
 \Psi(t), u(t) \in L_1^2 
\) for every $t\in [T_1, T_2]$.
Since $[T_1, T_2]$ is arbitarily chosen, the proof is completed.

\end{proof}

\begin{lemma}\label{wexp}
If \( u \in C^0\left( (T^-, T^+); H_+^3 \right) \) solves CMDNLS equation \eqref{21.1}
and \( u_0 \in L_1^2 \), then for every $t \in (T^-, T^+)$, we have 
\begin{equation}\label{L21capH2invunderWe*}
    W(t) \exp\left( -it \mathbb{L}_{u_0}^2 \right) \left( L_1^2 \cap H_+^2 \right) \subset L_1^2 \cap H_+^2.
\end{equation} 
\end{lemma}
\begin{proof}
Recall that equation \eqref{21.1} is locally well-posed in $H^3_+$, according to Proposition 2.1 and Theorem 2.1 of  \cite{2}. The operator $\mathbb{L}_{u_0}^2$ is a positive self-adjoint operator on the Hilbert space $L^2_+$, whose domain of definition is $H^2_+$. According to Stone theorem, the family  $(e^{-it\mathbb{L}_{u_0}^2})_{t\in\R}$ is a strongly continuous group of unitary operators on $L^2_+$. For any $h \in H^2_+ = \mathrm{Dom}(\mathbb{L}_{u_0})$, we have
\begin{equation}\label{InfiGeneLu2}
    \lim_{\theta \to 0} \big\|\tfrac{1}{\theta}\left(e^{-i \theta  \mathbb{L}_{u_0}^2}(h) - h \right)  + i\mathbb{L}_{u_0}^2 (h)   \big\|_{L^2_+}=0.
\end{equation}Since \( u(t) \in H_+^3 \), we have \( B_{u(t)} \in    \mathcal{B}_{\mathbb{C}}(H_+^2)  \cap \mathcal{B}_{\mathbb{C}}(L_+^2) \), $\forall t \in (T^-, T^+)$. As a consequence, the unitary propagator $W$ satisfies that $    W \in \  C^1\left( (T^-, T^+); \mathcal{B}_{\mathbb{C}}(L_+^2) \right) \bigcap  C^1\left( (T^-, T^+); \mathcal{B}_{\mathbb{C}}(H_+^2) \right)$. We thus have
\begin{equation}\label{H2invbyW(t)e-Lu2}
     W  (t) (H^2_+) + e^{-it \mathbb{L}_{u_0}^2} (H^2_+) \subset H^2_+  \; \Longrightarrow \;  W  (t)   e^{-it \mathbb{L}_{u_0}^2} (H^2_+) \subset H^2_+ .
\end{equation}For any $\Psi_0 \in L^2_1 \cap H^2_+$, we define $\Psi(t) = W  (t)   e^{-it \mathbb{L}_{u_0}^2}  (\Psi_0) \in L^2_+$, $\forall t \in (T^-, T^+)$. Thanks to formulas \eqref{tildBudef1}, \eqref{UniPropaW'=BW} and \eqref{InfiGeneLu2}, the function \( \Psi \in C^0\left( (T^-, T^+); H_+^2 \right) \cap C^1\left( (T^-, T^+); L_+^2 \right) \) satisfies    
\begin{equation}
   \lim_{\delta \to 0} \| \tfrac{1}{\delta} (\Psi(t+\delta)-\Psi(t))-\widetilde{\mathbb{B}}_{u(t)}(\Psi(t))\|_{L^2_+}= 0.
\end{equation}In other words, $ \Psi$ is the unique solution to equation \eqref{dtpsiEq=iDelPsi+2u}. By Lemma $\ref{gronwall}$, it follows that
\begin{equation*}
    W(t)\exp(-it\mathbb{L}_{u_0}^2)(\Psi_0) = \Psi(t) \in L_1^2. 
\end{equation*}So we obtain \eqref{L21capH2invunderWe*} and finish the proof.
\end{proof}

\begin{proof}[\noindent \textbf{Proof of Theorem \ref{invariance}}]
If \( u_0, \partial_x u_0, \partial_x^2 u_0 \in L_1^2 \), by Lemma \ref{utoLu}, we have \( \mathbb{L}_{u_0}(u_0), \mathbb{L}_{u_0}^2(u_0) \in L_1^2 \).
For every $T^- < t < T^+$, we have 
\[
\partial_t u(t) = \left( i \partial_x^2 u + 2u \partial_x \Pi(|u|^2) \right)(t) = \widetilde{\mathbb{B}}_{u(t)}(u(t)).
\]
Formula (4.25) in   \cite{11} expresses $u(t)$ in terms of the time $t$, the unitial datum $u_0$ and the unitary propagator $   W \in \  C^1\left( (T^-, T^+); \mathcal{B}_{\mathbb{C}}(L_+^2) \right) \bigcap  C^1\left( (T^-, T^+); \mathcal{B}_{\mathbb{C}}(H_+^2) \right)$:
\begin{equation}\label{expresu(t)WLu0}
    u(t) = W(t)\exp(-it\mathbb{L}_{u_0}^2)(u_0)\in W(t) \exp(-it\mathbb{L}_{u_0}^2)(L_1^2 \cap H_+^4) \subset L_1^2 \cap H_+^2.
\end{equation}
By formulas \eqref{expresu(t)WLu0},  \eqref{uniequiLL2} and \eqref{L21capH2invunderWe*}, we can deduce that
\begin{align*}
\mathbb{L}_{u(t)}(u(t)) &= (\mathbb{L}_{u(t)} W(t) ) \exp(-it\mathbb{L}_{u_0}^2)(u_0) = W(t)\mathbb{L}_{u_0} \exp(-it\mathbb{L}_{u_0}^2)(u_0)
= W(t) \exp(-it\mathbb{L}_{u_0}^2)(\mathbb{L}_{u_0}(u_0))
\\
&\in W(t) \exp(-it\mathbb{L}_{u_0}^2)(L_1^2 \cap H_+^3) \subset L_1^2 \cap H_+^2.
\end{align*}
Similarly, we use the same process to show that 
\begin{align*}
\mathbb{L}_{u(t)}^2(u(t)) &= \mathbb{L}_{u(t)} ^2W(t)\exp(-it\mathbb{L}_{u_0}^2)(u_0) = W(t)\mathbb{L}_{u_0}^2 \exp(-it\mathbb{L}_{u_0}^2)(u_0)
= W(t)\exp(-it\mathbb{L}_{u_0}^2)(\mathbb{L}_{u_0}^2(u_0))
\\
&\in W(t) \exp(-it\mathbb{L}_{u_0}^2)(L_1^2 \cap H_+^2) \subset L_1^2 \cap H_+^2.
\end{align*}
As a consequence,  Lemma \ref{Lutou} yields that  \( \partial_x u(t), \partial_x^2 u(t) \in L_1^2 \) for any $t\in (T^-,T^+)$.
\end{proof}

\section{Time evolution}\label{sec6}

\begin{theorem}\label{eigen_time}
 For any $T^- < 0  $ and $T^+>0$, let $u\in C^0((T^-, T^+),H^1_+)$ be the unique solution to the CMDNLS equation \eqref{21.1} with initial datum $u(0)=u_0\in H^1_+$.  For any $t \in (T^-, T^+)$, the Lax operator $\mathbb{L}_{u(t)}$ has finitely many eigenvalues, listed in increasing order  $\lambda_1^{u(t)}<\lambda_2^{u(t)} <\cdots < \lambda_N^{u(t)}$, each of which is simple. For any $j \in \{1,2, \cdots, N\}$, let $\varphi^{u(t)}_j\in \mathrm{Ker}(\lambda_j^{u(t)} - \mathbb{L}_{u(t)})$ be the normalized and reoriented eigenfunction defined in \eqref{defEigenVec}.  Then $\varphi_j^u$ satisfies the following: 
\begin{equation}\label{tim-evo-eigen}
     \partial_t \varphi_j^u = i \partial_x^2 \varphi_j^u + 2 u \partial_x \Pi (\overline{u}\varphi_j^u ),\quad \forall t\in (T^-,T^+).
\end{equation}

\end{theorem}

\begin{proof}
    The eigenvalues $\{\lambda_j^u\}_{j=1}^n$  are invariant under the flow of \eqref{21.1}. Recall that $\mathbb{B}_{u(t)}$ is a  bounded operator on $L^2_+$ if $u\in H^1_+$. Let $W\in C^1\big(   (T^-, T^+);  \mathcal{B}_{\C}(L^2_+)    \big) $ denote the unique solution to the  ordinary differential equation \eqref{UniPropaW'=BW}. 
     By formulas \eqref{unit} and \eqref{uniequiLL2}, it follows that  

     \begin{equation}
         \mathbb{L}_{u(0)}(W(t)^*\varphi_j^{u(t)}) = W(t)^*\left( W(t) \mathbb{L}_{u(0)}(W(t)^* \right) \varphi_j^{u(t)})=\lambda_j^{u(0)} (W(t)^*\varphi_j^{u(t)}).
     \end{equation}
Due to the simplicity of the eigenvalues $\{\lambda_j^u\}_{j=1}^n$, there exists a $c(t)\in \C$ such that 
\begin{equation}\label{6.8}
W(t)^*\varphi_j^{u(t)}=c(t)\varphi_j^{u(0)}.
\end{equation}
Using \eqref{defEigenVec} and \eqref{unit}, we obtain that 
\begin{equation}\label{5.8}
	1=\| \varphi_j^{u(t)}\|_{L^2}= \|  W(t)^*\varphi_j^{u(t)} \|_{L^2}= \|  c(t)\varphi_j^{u(0)} \|_{L^2}=|c(t)|.
	\end{equation}
Formula \eqref{unit} yields
\begin{equation}\label{5.9}
	\begin{split}
	\langle u(t), \varphi_j^{u(t)} \rangle_{L^2}=	\langle W(t)^*u(t), W(t)^*\varphi_j^{u(t)} \rangle_{L^2}=\overline{c(t)} 	\langle W(t)^*u(t), \varphi_j^{u(0)} \rangle_{L^2}.
		\end{split}
	\end{equation}
By formula \eqref{expresu(t)WLu0} and Theorem \uppercase\expandafter{\romannumeral8}.5 in  \cite{Reed Simon book 1},  formula \eqref{5.9} becomes
\begin{equation}\label{6.11}
\begin{split}
	 \overline{c(t)} 	\langle W(t)^*u(t), \varphi_j^{u(0)} \rangle_{L^2}  
     = & \overline{c(t)} 	\langle e^{-it\mathbb{L}_{u(0)}^2}u(0) , \varphi_j^{u(0)} \rangle_{L^2}
     = \overline{c(t)} 	\langle u(0) , e^{it\mathbb{L}_{u(0)}^2}\varphi_j^{u(0)} \rangle_{L^2} \\
    = & \overline{c(t)} 	\langle u(0) , \varphi_j^{u(0)} \rangle_{L^2}e^{-it|\lambda_j^{u(0)}|^2}.
    \end{split}
	\end{equation}We use formulas \eqref{5.9},  \eqref{6.11} and
    \eqref{defEigenVec} to deduce that 
    \begin{equation}\label{cbarinS1}
        \overline{c(t)} e^{-it|\lambda_j^{u(0)}|^2} =\frac{\langle u(t), \varphi_j^{u(t)} \rangle_{L^2}}{\langle u(0) , \varphi_j^{u(0)} \rangle_{L^2}}>0.
    \end{equation}
Together with  \eqref{5.8}, this yields 
\begin{equation}\label{6.13}
	c(t)=e^{-it(\lambda_j^{u(0)})^2}.
	\end{equation}
Substituting \eqref{6.13} into \eqref{6.8} yields
\(
W(t)^* \varphi_j^{u(t)} = c(t) \varphi_j^{u(0)} = e^{-it(\lambda_j^{u(0)})^2} \varphi_j^{u(0)}= e^{-it\mathbb{L}_{u(0)}^2} \varphi_j^{u(0)} 
\).
In other words, we have 
\[
\varphi_j^{u(t)} = W(t) e^{-it\mathbb{L}_{u(0)}^2} \varphi_j^{u(0)}.
\]
According to Stone theorem, we differentiate with respect to $t$ to obtain that
\begin{equation}
	\begin{split}
\partial_t \varphi_j^{u(t)} =&\partial_t  (W(t) e^{-it\mathbb{L}_{u(0)}^2} \varphi_j^{u(0)})
= (\partial_t W(t)) e^{-it\mathbb{L}_{u(0)}^2} \varphi_j^{u(0)} + W(t) \partial_t (e^{-it\mathbb{L}_{u(0)}^2}\varphi_j^{u(0)})
\\
=&\mathbb{B}_{u(t)} W(t) e^{-it\mathbb{L}_{u(0)}^2} \varphi_j^{u(0)} + W(t) (-i\mathbb{L}_{u(0)}^2) e^{-it\mathbb{L}_{u(0)}^2} \varphi_j^{u(0)}
=\mathbb{B}_{u(t)} \varphi_j^{u(t)}-i\mathbb{L}_{u(t)}^2(W(t)e^{-i\mathbb{L}_{u(0)}^2} \varphi_j^{u(0)}).
\end{split}
\end{equation}
Hence we obtain
\begin{equation}
			\partial_t \varphi_j^{u(t)} = \mathbb{B}_{u(t)} \varphi_j^{u(t)} - i\mathbb{L}_{u(t)}^2 \varphi_j^{u(t)} 
		 = \widetilde{\mathbb{B}}_{u(t)}(\varphi_j^{u(t)}),
\end{equation}
which is the equation \eqref{tim-evo-eigen}, i.e.,  $
     \partial_t \varphi_j^u = i \partial_x^2 \varphi_j^u + 2 u \partial_x \Pi (\overline{u}\varphi_j^u )$.
\end{proof}

\begin{theorem}
 Assume that
    $u \in C^0\left( (T^-, T^+); H^2_s\cap H^3_+\right)$ solves the CMDNLS equation \eqref{21.1} with initial datum $u(0)=u_0 \in H^2_s \cap H^3_+$, 
     for some $T^- < 0 <T^+$ and  \( s > \frac 12 \).    For every $$k\in \C^+\cup \C^- \cup  \big((-\infty,0)\setminus \sigma_{pp}(\mathbb{L}_{u_0})\big),$$	let $\phi_k(t)\in L_{-s}^\infty$ denote the solution to 
    $\mathbb{L}_{u(t)} \phi_k(t) = k \phi_k (t)+ u(t).$
    Then we have  
    \begin{equation}
        \partial_t \phi_k = i \partial_x^2 \phi_k + 2 u \partial_x \Pi (\overline{u}\phi_k ),\quad \forall t\in (T^-,T^+).
    \end{equation}
\end{theorem}

\begin{proof}
   For every fixed $k\in \C^+\cup \C^- \cup  \big((-\infty,0)\setminus \sigma_{pp}(\mathbb{L}_u)\big)$,	
 differentiating with respect to $t$ of 
\(
\mathbb{L}_{u(t)} \phi_k(t) = k \phi_k(t) + u(t)
\) gives
\(
(\frac{\partial}{\partial t} \mathbb{L}_{u(t)}) \phi_k(t) + \mathbb{L}_{u(t)} \left( \partial_t \phi_k (t)\right) = k \, \partial_t \phi_k(t) + \partial_t u(t).
\)
Utilizing the  Lax equation
\begin{equation}
{\partial_t} \mathbb{L}_{u(t)} = \left[ \mathbb{B}_{u(t)}, \mathbb{L}_{u(t)} \right] =\left[ \widetilde{\mathbb{B}}_{u(t)}, \mathbb{L}_{u(t)} \right]
\end{equation}
and 
\(
\partial_t u (t) = \widetilde{\mathbb{B}}_{u(t)}  u(t) , 
\)
we have 
\begin{equation}\label{Luf=kf}
\mathbb{L}_{u(t)} \left( \partial_t \phi_k(t) - \widetilde{\mathbb{B}}_{u(t)} \phi_k(t) \right) = k \left( \partial_t \phi_k(t) - \widetilde{\mathbb{B}}_{u(t)} \phi_k (t)\right).
\end{equation}
	Let
	 \(
	  f(t) = \partial_t \phi_k (t)- \widetilde{B}_{u(t)} \phi_k(t),
	  \)
	  then formula \eqref{Luf=kf} can be written as \( \mathbb{L}_{u(t)} f (t) = k f (t)\).	
      
Next  we claim that  \( f(t) \in L^\infty\). 
\begin{proof}[Proof of the claim]
According to Lemma 
\ref{phikequi}, we obtain that $\phi_k=G_k*(u+u\Pi(\overline u \phi_k))$
	and 
	 \[
	  \partial_t \phi_k = G_k * \left( \partial_t u+ (\partial_t u) \Pi (\overline{u} \phi_k) + u \Pi (({\partial_t \overline u} ) \phi_k) + u \Pi (\overline{u} \partial_t \phi_k) \right), \quad (t,x) \in  (T^-, T^+) \times\mathbb{R}
	  .
      \]
		This can be written as 
	\(
	 (I - T_k) \partial_t \phi_k  =G_k*A ,
	 \) where we define $A:= \partial_t u + (\partial
     _t u)\Pi (\overline{u} \phi_k) +  u \Pi (\overline{\partial_t u}  \phi_k )  $. 
Since  $ u\in C^0((T^-, T^+),H^2_s\cap H^3_+)$, we get $\partial_tu(t)=\widetilde{\mathbb{B}}_{u(t)}u(t)=\left(i\partial_x^2 u +2u\Pi(\partial_x|u|^2)\right)(t)\in L^2_s\cap H^1_+$. Consequently $A\in C((T^-, T^+); L^2_+)$. Because $G_k\in L^2$, Young's inequality implies  $(I-T_k)\partial_t\phi_k(t)\in L^\infty$.  Given that  the operator $I-T_k$ is bijective  from $L^\infty$ to $L^\infty$ , it follows that $\partial_t \phi_k(t) \in L^\infty $. 

From Lemma \ref{phikequi}, we have $\phi_k(t)\in H^1_+$. 
Together with $u(t)\in H^2_s\cap H^3_+$,  we have 
	$ -i \partial_x \phi_k = u \Pi (\overline{u} \phi) + k \phi_k + u\in C((T^-, T^+); H^1_+)$. And hence $\phi_k\in C((T^-, T^+); H^2_+)$. Iterating this argument gives  $-i \partial_x \phi_k = u \Pi (\overline{u} \phi) + k \phi_k + u\in C((T^-, T^+); H^2_+)$ and consequently $\phi_k\in C((T^-, T^+); H^3_ +)$. For any $t \in (T^-, T^+)$, we have 
    \begin{equation}
(\widetilde{\mathbb{B}}_u\phi_k)(t)=(i\partial_x ^2\phi_k+2\mathbb{T}_u\partial_x \mathbb{T}_{\overline u}\phi_k)(t)=(i\partial_x ^2\phi_k+2u\partial_x \Pi(\overline u \phi_k))(t)\in H^1_+ \subset L^\infty .
\end{equation}
Therefore, we obtain $f(t)\in L^\infty$. \end{proof}
	Formula \eqref{Luf=kf} yields that
	\( -i \partial_x f - k f = u \Pi (\overline{u} f) \)
	. Using  the Fourier transform, we obtain that
	\begin{equation}
	    (\xi - k) \hat{f} = \mathds{1}_{\R^+} \widehat{(u \Pi (\overline{u} f) ) }\Longrightarrow f=G_k*u\Pi(\overline u f),\quad (t,x) \in  (T^-, T^+) \times\mathbb{R}.
	\end{equation}
 Young's inequality implies that $f(t)\in L^2$.
	By a similar proof of \eqref{phik01} and \eqref{22.14}, we have 
	\( f (t)\in H^1 \).
		Since $k\in \C^+\cup \C^- \cup  \big((-\infty,0)\setminus \sigma_{pp}(\mathbb{L}_u)\big)$ is not the eigenvalues, we obtain  
	\(
	f(t)=\partial_t \phi_k(t) - \widetilde{\mathbb{B}}_{u(t)} \phi_k(t) = 0
	\), for any $t \in (T^-, T^+)$.
\end{proof}

\begin{lemma}\label{partialtpsito0}
     Assume that
    $u \in C^0\left( (T^-, T^+); H^2_s\cap H^3_+\right)$ solves the CMDNLS equation \eqref{21.1} with initial datum $u(0)=u_0 \in H^2_s \cap H^3_+$, 
     for some $T^- < 0 <T^+$ and  \( s > \frac 12 \).  For every $\lambda\in (0,+\infty)\setminus \sigma_{pp}(\mathbb{L}_{u_0})$, let \( \psi^\pm_{\lambda} (t)\in L^\infty_{-s} \) denote the solution to
\begin{align} \nonumber    &\mathbb{L}_{u(t)}\psi_{\lambda}^\pm(t)=\lambda\psi_{\lambda}^\pm(t)\in L^{\infty}_{-s},\\
        &\psi_\lambda^+(t,x)-e^{ix\lambda} \to 0, \quad x\to -\infty,\\
    &\psi_\lambda^-(t,x)-e^{ix\lambda} \to 0, \quad x\to +\infty.
 	\end{align}
Then we have $\mathop{{\rm lim}}\limits_{x\to \mp\infty} \partial_t \psi_\lambda^\pm (t,x)=0$, for any $t \in (T^-, T^+)$.
\end{lemma}
\begin{proof}
According to Lemma \ref{psi-equi}, we have 
\begin{equation}\label{6.15}
    \begin{split}
        &\partial_t \psi_\lambda^+ =G_\lambda^+ * \left( u \Pi\left(  \overline{u} \left( \partial_t \psi_\lambda^+ \right) \right) \right)+G_\lambda^+ * \left( (\partial_t u) \Pi( \overline{u} \psi_\lambda^+) + u \Pi\left( (\partial_t  \overline{u}) \psi_\lambda^+ \right) \right).
    \end{split}
\end{equation}
This gives 
$(\mathrm{id}-T_\lambda^+) \left( \partial_t \psi_\lambda^+ \right) = G_\lambda^+ * \left( (\partial_t u) \Pi( \overline{u} \psi_\lambda^+) + u \Pi\left( (\partial_t  \overline{u}) \psi_\lambda^+ \right) \right) = \left( \mathfrak{h} - \widetilde{G}_\lambda \right) * \mathfrak{g}$,
where
\( 
 \mathfrak{h}(x) := i e^{i\lambda x} \mathds{1}_{x \geq 0}
\)
and 
$
\mathfrak{g} := (\partial_t u) \Pi( \overline{u} \psi_\lambda^+) + u \Pi\left( (\partial_t  \overline{u}) \psi_\lambda^+ \right).
$
Since $u\in H^3_+\cap H^2_s$, we have 
\(
\partial_t u(t) = \widetilde{\mathbb{B}}_{u(t)}u(t) = (i \partial_x^2 u + 2u \Pi \partial_x( |u|^2))(t) \in L_{s}^2 \cap L^\infty.
\)
It follows that
\begin{equation}
    \mathfrak{g} = (\partial_t u) \Pi( \overline{u} \psi_\lambda^+) + u \Pi\left( (\partial_t  \overline{u}) \psi_\lambda^+ \right) \in C((T^-, T^+); L^2_+\cap L^1)
\end{equation}
 and 
\(
\widetilde{G}_\lambda * (\mathfrak{g}(t)) \in L^2 * L^2 \subset L^\infty.
\)
Since 
\(
\left(\mathfrak{h} * (\mathfrak{g}(t)) \right)(x) = i e^{i\lambda x} \int_{-\infty}^x e^{-i\lambda y} g(y) dy
\), 
we have 
\begin{equation}
|\left(\mathfrak{h} * (\mathfrak{g}(t)) \right)(x)| \leq \int_{-\infty}^x |e^{-i\lambda y} \mathfrak{g}(t, y)| dy
\leq \|\mathfrak{g}(t)\|_{L^1},\quad  \forall x \in \mathbb{R}.
\end{equation}
Then we get
\(
\|\mathfrak{h} * (\mathfrak{g}(t))\|_{L^\infty} \leq \|\mathfrak{g}(t)\|_{L^1} < +\infty
\) and 
\(
\mathfrak{h} * (\mathfrak{g}(t)) \in L^\infty.
\)
Therefore, we obtain
\begin{equation}
    (\mathrm{id}-T_\lambda^+) (\partial_t \psi_\lambda^+(t)) = (\mathfrak{h} - \widetilde{G}_\lambda) * (\mathfrak{g}(t))\in L^\infty\subset L^\infty_{-s}.
\end{equation}
The operator $\mathrm{id}-T_\lambda^+: L_{-s}^\infty \to L_{-s}^\infty$ is bijective, so we obtain $\partial_t\psi_\lambda^+(t)\in L_{-s}^\infty$.

From \eqref{6.15}, we get
\(
\partial_t \psi_\lambda^+  = G_\lambda^+ * \mathfrak{p}
\) 
with
\(
\mathfrak{p} := u \Pi \left( \overline{u} \, \partial_t \psi_\lambda^+ \right) + \mathfrak{g} \in C((T^-,T^+); L^2_+ \cap L^1)
\).
Then we have
\(
\partial_t \psi_\lambda^+ = \mathfrak{h} * \mathfrak{p} - \widetilde{G}_\lambda * \mathfrak{p}
\).
And 
\(
\widetilde{G}_\lambda * (\mathfrak{p}(t)) \in L^2 * L^2 \subset C_0
\).
Since $\mathop{\rm lim}\limits_{x \to -\infty}e^{-i\lambda y} \mathfrak{p}(t, y) \, \mathds{1}_{(-\infty, x)} (y)= 0$,  for every  $y \in \mathbb{R}$, $\left| e^{-i\lambda y} \mathfrak{p}(t, y) \, \mathds{1}_{(-\infty, x)} (y) \right| \leq \left| \mathfrak{p}(t, y) \right|$ and $\mathfrak{p}(t)\in L^1$, we deduce that
\begin{equation}
(\mathfrak{h} * (\mathfrak{p}(t)))  (x)
= i e^{i\lambda x} \int_{\mathbb{R}} e^{-i\lambda y} \, \mathfrak{p}(t, y) \, \mathds{1}_{(-\infty, x)} (y) \, dy\to 0
\end{equation}
as $x\to-\infty$ by  dominated convergence theorem. Therefore, we obtain 
\begin{equation}
 \partial_t \psi_\lambda^+ (t,x)  = \left(\mathfrak{h} * (\mathfrak{p}(t))  - \widetilde{G}_\lambda * (\mathfrak{p}(t))\right) (x) 
\to 0
\end{equation}
as \(x \to -\infty
\). 
\end{proof}

\begin{theorem}\label{psi_lam_time}
 Assume that
    $u \in C^0\left( (T^-, T^+); H^2_s\cap H^3_+\right)$ solves the CMDNLS equation \eqref{21.1} with initial datum $u(0)=u_0 \in H^2_s \cap H^3_+$, 
     for some $T^- < 0 <T^+$ and  \( s > \frac 12 \).   For every $\lambda\in (0,+\infty)\setminus \sigma_{pp}(\mathbb{L}_{u_0})$, let \( \psi^\pm_{\lambda} (t)\in L^\infty_{-s} \) denote the solution to
\begin{align}   &\mathbb{L}_{u(t)}\psi_{\lambda}^\pm(t)=\lambda\psi_{\lambda}^\pm(t) \in L^{\infty}_{-s},\label{Lpsi=lam-psi}\\
        &\psi_\lambda^+(t,x)-e^{ix\lambda} \to 0, \quad x\to -\infty,\\
    &\psi_\lambda^-(t,x)-e^{ix\lambda} \to 0, \quad x\to +\infty.
 	\end{align}
Then we have 
  \begin{equation}
        \partial_t \psi_\lambda^\pm = i \partial_x^2 \psi_\lambda^\pm + 2 u \partial_x \Pi (\overline{u}\psi_\lambda^\pm )+i \lambda^2 \psi_\lambda^\pm,\quad \forall t\in (T^-,T^+).
    \end{equation}
\end{theorem}

\begin{proof}
    Using $\mathbb{L}_{u(t)}\psi_{\lambda}^\pm(t)=\lambda\psi_{\lambda}^\pm(t) $   and ${\partial_t} \mathbb{L}_{u(t)}=[\widetilde{\mathbb{B}}_{u(t)}, \mathbb{L}_{u(t)}]$, we get
    \begin{equation}
        \mathbb{L}_{u(t)}(\partial_t\psi_\lambda^+(t)-\widetilde{\mathbb{B}}_{u(t)}\psi_\lambda^+(t))=\lambda(\partial_t\psi_\lambda^+(t)-\widetilde{\mathbb{B}}_{u(t)}\psi_\lambda^+(t)).
    \end{equation}
{From Lemma} \ref{partialtpsito0} and Corollary \ref{Buto0}, it follows that $\partial_t\psi_\lambda^+(t)\in L^\infty$, $\mathbb{B}_{u(t)}(\psi_\lambda^+(t))\in H^1_+$, and
\begin{equation}
\lim_{x\to-\infty}\partial_t\psi_\lambda^+(t,x)= 0, \quad \lim_{|x| \to +\infty} \bigl| \mathbb{B}_{u(t)}(\psi_\lambda^+(t)) (x) \bigr| = 0.
\end{equation}
It follows from \eqref{tildBudef1} and \eqref{Lpsi=lam-psi} that
\begin{equation*}
    \begin{split}
        &\partial_t \psi_\lambda^+(t,x) - \left(\widetilde{\mathbb{B}}_{u(t)} \left( \psi_\lambda^+ (t) \right)  \right)(x) -i \lambda^2 e^{i\lambda x} 
         \\
        = &\partial_t \psi_\lambda^+ (t,x) - \left(\mathbb{B}_{u(t)} \left( \psi_\lambda^+ (t) \right) \right)(x)+ i \lambda^2 \psi_\lambda^+ (t, x) -i \lambda^2 e^{i\lambda x}  \to 0,\quad 
        \mathrm{as} \quad x \to -\infty.
    \end{split}
\end{equation*}
 That is
\(
\frac{1}{i \lambda^2} \left( \partial_t \psi_\lambda^+ (t) - \widetilde{\mathbb{B}}_{u(t)} \left( \psi_\lambda^+ (t) \right) \right)(x) - e^{i\lambda x} \to 0\) as \(x \to -\infty
\).
Thanks to the uniqueness of the solution of $\mathbb{L}_{u(t)}\psi_\lambda^+(t)=\lambda \psi_\lambda^+(t)$ and $\psi_\lambda^+(t,x)-e^{ix\lambda} \to 0$ as $ x\to -\infty$, 
we have
\(
\psi_\lambda^+ (t) = \frac{1}{i \lambda^2} \left( \partial_t \psi_\lambda^+ (t) - \widetilde{\mathbb{B}}_{u(t)} \left( \psi_\lambda^+ (t) \right) \right)
\).
This yields
\(
\partial_t \psi_\lambda^+ (t) - \widetilde{\mathbb{B}}_{u(t)} \left( \psi_\lambda^+ (t) \right) = i \lambda^2 \psi_\lambda^+(t)
\).

\end{proof}

\begin{lemma}\label{partial_tphilam_to_0}
   Assume that
    $u \in C^0\left( (T^-, T^+); H^2_s\cap H^3_+\right)$ solves the CMDNLS equation \eqref{21.1} with initial datum $u(0)=u_0 \in H^2_s \cap H^3_+$, 
     for some $T^- < 0 <T^+$ and  \( s > \frac 12 \).  For any $t \in (T^-, T^+)$  and  for any  $\lambda\in (0,+\infty)\setminus \sigma_{pp}(\mathbb{L}_{u_0})$, 	let $\phi_\lambda^\pm(t)\in L_{-s}^\infty$ denote the solution to 
    \begin{align}
&\mathbb{L}_{u(t)}\phi_{\lambda}^\pm(t)=\lambda\phi_{\lambda}^\pm(t)+u(t),\\
&\phi_\lambda^+(t, x)\to 0,\quad x\to-\infty,\\
&\phi_\lambda^-(t, x)\to 0,\quad x\to+\infty.
\end{align}
    Then we have $$\partial_t\phi_\lambda^\pm(t)\in L^\infty,\quad \lim_{x\to\mp\infty}\partial_t\phi_\lambda^\pm(t,x)=0,\quad \forall t\in (T^-,T^+).$$
\end{lemma}
\begin{proof}
    By Lemma \ref{philamequi}, we get
\(
\phi_\lambda^- = G_\lambda^- * \left( u + u \Pi \left(  \overline{u} \phi_\lambda^- \right) \right).
\)
 Differentiating in time yields
\[\small{
\begin{aligned}
&\partial_t \phi_\lambda^- 
= G_\lambda^- * \left( u \Pi \left(  \overline{u} \partial_t \phi_\lambda^- \right) \right) + G_\lambda^- * \left( \partial_t u + (\partial_t u) \Pi \left(  \overline{u} \phi_\lambda^- \right) + u \Pi \left( (\partial_t  \overline{u}) \phi_\lambda^- \right) \right) 
\\
=&
 T_\lambda^- \left( \partial_t \phi_\lambda^- \right) + G_\lambda^- * w, \quad (t,x) \in  (T^-, T^+) \times\mathbb{R},
\end{aligned}
}
\]
with
\(
w := \partial_t u + (\partial_t u) \Pi \left(  \overline{u} \phi_\lambda^- \right) + u \Pi \left( (\partial_t  \overline{u}) \phi_\lambda^- \right) 
\).
 Given that  $\partial_x^2u(t) \in L^2_s\cap L^1$, $u(t)\in L^2_s$ and $u(t)\in H^3$, it follows that $\widetilde{\mathbb{B}}_{u(t)}(u(t))=i \partial_x^2 u(t) + 2(u  \Pi \left( \partial_x(|u|^2) \right))(t) \in L_{s}^2 \cap L^1$. Moreover, since  $u(t)\in H^3_+$, we have  $\widetilde{\mathbb{B}}_{u(t)}(u(t))=(i \partial_x^2 u + 2u \Pi \partial_x( |u|^2))(t)\in H^1_+\cap L^2_s$. Therefore, we deduce that 
\[
 w (t)= \widetilde{\mathbb{B}}_{u(t)}(u) + \widetilde{\mathbb{B}}_{u(t)}(u) \Pi \left(  \overline{u} \phi_\lambda^- \right) + u \Pi \left( \widetilde{\mathbb{B}}_{u(t)}(u) \phi_\lambda^- \right)
\in L^1 \cap L^2 + H_+^1 \cdot L^2 +  L^2 \cdot H^1 \subset L^2 \cap L^1.
\]
Since $\widetilde{G}_\lambda \in L_-^2$, Young's inequality implies  $\widetilde{G}_\lambda * (w(t)) \in L^\infty \cap C_0$. Let $g := -i e^{i\lambda x} \mathds{1}_{\mathbb{R}^-}$, then $g\in L^\infty$,  and by Young's inequality, we have $\|g * (w(t))\|_{L^\infty} \leq \|g\|_{L^\infty} \|w(t)\|_{L^1} = \|w(t)\|_{L^1} < +\infty$ . Thus, we obtain 
\begin{equation}\label{i-T-par-phi-lam}
(I - T_\lambda^-) \left( \partial_t \phi_\lambda^- \right)(t) = G_\lambda^- * (w(t)) = \left( g - \widetilde{G}_\lambda \right) * (w(t)) \in L^\infty.
\end{equation}
Because  $I - T_\lambda^-: L^\infty \to L^\infty$ is bijective,  it follows that $\partial_t \phi_\lambda^-(t) \in L^\infty$. 

Now, since $\partial_t\phi_\lambda^-(t)\in L^\infty$ and $u(t)\in H^1$, we have $(u\Pi(\overline u \partial_t\phi_\lambda^-))(t)\in L^2\cap L^1$. Combined with $w(t)\in L^2$,  this gives $\widetilde{G}_\lambda*((u\Pi(\overline u \partial_t \phi_\lambda^-)+w)(t))\in L^2*L^2\subset C_0$. By the dominated convergence theorem, it follows that  $\mathop{\rm lim}\limits_{x\to +\infty}g*((u\Pi(\overline u \partial_t \phi_\lambda^-)+w)(t))=0$. We conclude from \eqref{i-T-par-phi-lam} that
 $\partial_t\phi_\lambda^-(t,x)\to 0$ as $x\to+\infty$.
\end{proof}

\begin{theorem}\label{phi_lam_time}
    Assume that
    $u \in C^0\left( (T^-, T^+); H^2_s\cap H^3_+\right)$ solves the CMDNLS equation \eqref{21.1} with initial datum $u(0)=u_0 \in H^2_s \cap H^3_+$, 
     for some $T^- < 0 <T^+$ and  \( s > \frac 12 \).   For every $\lambda\in (0,+\infty)\setminus \sigma_{pp}(\mathbb{L}_{u_0})$, 	let $\phi_\lambda^\pm(t)\in L_{-s}^\infty$ denote the solution to 
    \begin{align}\nonumber
&\mathbb{L}_{u(t)}\phi_{\lambda}^\pm(t)=\lambda\phi_{\lambda}^\pm(t)+u(t),\\
&\phi_\lambda^+(t, x)\to 0,\quad x\to-\infty,\\
&\phi_\lambda^-(t, x)\to 0,\quad x\to+\infty,
\end{align}
    Then we have 
    \begin{equation}
        \partial_t \phi_\lambda^\pm = i \partial_x^2 \phi_\lambda^\pm + 2 u \partial_x \Pi (\overline{u}\phi_\lambda^\pm ),\quad \forall t\in (T^-,T^+).
    \end{equation}
\end{theorem}

\begin{proof}
    Using $\mathbb{L}_{u(t)} \phi_\lambda^-(t) = \lambda \phi_\lambda^- (t)+ u(t)$, $\frac{\partial}{\partial_t}  \mathbb{L}_{u(t)}=[\widetilde {\mathbb{B}}_{u(t)}, \mathbb{L}_{u(t)}]$ and $\partial_t u(t)= \widetilde {\mathbb{B}}_{u(t)}u(t)$, we obtain that   \begin{equation}\label{Lm=lambdam}
       \mathbb{L}_{u(t)} \left( \partial_t \phi_\lambda^-(t) - \widetilde{\mathbb{B}}_{u(t)} (\phi_\lambda^-(t)) \right) = \lambda \left( \partial_t \phi_\lambda^- (t)- \widetilde{\mathbb{B}}_{u(t)} (\phi_\lambda^-(t)) \right).
   \end{equation}
   Define  $m_\lambda^-(t):=\partial_t \phi_\lambda^- (t)- \widetilde{\mathbb{B}}_{u(t)} (\phi_\lambda^-(t))$, formula \eqref{Lm=lambdam} becomes 
   \begin{equation}\label{Lum=}
   \mathbb{L}_{u(t)}m_\lambda^-(t)=\lambda m_\lambda^-(t).
   \end{equation}
   From Lemma \ref{partial_tphilam_to_0} and Corollary \ref{B(phi)inH1},  we deduce that  $\sup_{x>a} |m_\lambda^-(t,x)|<+\infty $ and $\mathop{\rm lim}\limits_{x\to+\infty}m_\lambda^-(t,x)= 0$ for any $a\in\R$.   Applying the Fourier transform with respect to the space variable $x$ to \eqref{Lum=}, we obtain the equation in frequency space:
   \[
   \hat{m}_\lambda^-=\frac{\mathds{1}_{\xi\geq 0}}{\xi-(\lambda-i\varepsilon)}\mathcal{F}(u\Pi(\overline u m_\lambda^-))+\frac{i\varepsilon }{\xi-(\lambda-i\varepsilon)}\hat{m}_\lambda^-,
   \]
   where $\varepsilon>0$. 
Taking the inverse Fourier transform gives 
\[
m_\lambda^-=G_{\lambda-i\varepsilon}*(u\Pi(\overline u m_\lambda^-))+\mathcal{F}^{-1}(\frac{i\varepsilon }{\xi-(\lambda-i\varepsilon)}\hat{m}_\lambda^-).
\]
Following arguments similar to those in the proofs of \eqref{three_cv} and \eqref{convo}, and 
{using the facts that  $\sup_{x>a} |m_\lambda^-(t,x)|<+\infty $  and $\mathop{\rm lim}\limits_{x\to +\infty}m_\lambda^-(t,x)= 0$ for any $a\in\R$, we conclude that }
$$
m_\lambda^-=G_\lambda^-*(u\Pi(\overline u m_\lambda^-)) \implies (\mathrm{id}-T_\lambda^-)m_\lambda^-=0.
$$
 Since the operator $\mathrm{id}-T_\lambda^-: L^\infty \to L^\infty$ is bijective, it follows that $$m_\lambda^-(t)=\partial_t \phi_\lambda^-(t) - \widetilde{\mathbb{B}}_{u(t)} \phi_\lambda^-(t)=0.$$
\end{proof}
Using Theorem \ref{phi_lam_time}, \ref{psi_lam_time} and \eqref{3.2},\eqref{24}, we obtain the following theorem: 

\begin{theorem}
      
    Assume that
    $u \in C^0\left( (T^-, T^+); H^2_s\cap H^3_+\right)$ solves the CMDNLS equation \eqref{21.1} with initial datum $u(0)=u_0 \in H^2_s \cap H^3_+$, 
     for some $T^- < 0 <T^+$ and  \( s > \frac 12 \).   Let  $\phi_\lambda^\pm(x,t)$ and $\psi_\lambda^\pm(x,t)$  be given by Definition  \ref{defphilam} and \ref{defpsilam}. Let $\Gamma(\lambda,t)$ and $\beta(\lambda,t)$ be given by \eqref{gamma} and \eqref{3.5}. Then we have 
    \begin{equation}
        \partial_t\Gamma(\lambda, t)=0,\quad \partial_t\beta(\lambda,t)=-i\lambda^2\beta(\lambda,t).
    \end{equation}
\end{theorem}

 \section*{Acknowledgments}
Wang acknowledges the support from
the National Natural Science Foundation of China, Grant
No. 12371247, 12431008 and Beijing Natural Science Foundation Grant No. 1262012.
Sun acknowledges the support from the China Postdoctoral Science Foundation, Grant No. 2025M773094. Zhao acknowledges the support  from the Fundamental Research Funds for the Central Universities.

\bibliographystyle{plain}   

\begin{thebibliography}{10}
\bibitem{1} Abanov A G, Bettelheim E,  Wiegmann P. Integrable hydrodynamics of Calogero–Sutherland model: bidirectional Benjamin–Ono equation. J. Phys. A, 2009, 42(13): 135201.
	 		 	 				
\bibitem{4} Badreddine R. On the global well-posedness of the Calogero–Sutherland derivative nonlinear Schr\"odinger equation. Pure Appl. Anal., 2024, 6(2): 379-414.
	 					
\bibitem{5} Badreddine R. Traveling waves and finite gap potentials for the Calogero–Sutherland derivative nonlinear Schrödinger equation. Ann. Inst. H. Poincar\'e C Anal. Non Lin\'eaire, 42(4):1037–1092, 2025.
	 											
\bibitem{13} Coifman R R, Wickerhauser M V. The scattering transform for the Benjamin-Ono equation. Inverse Probl., 1990, 6(5): 825.


													
				
\bibitem{12}Fokas A S, Ablowitz M J. The inverse scattering transform for the Benjamin‐Ono equation—a pivot to multidimensional problems. Stud. Appl. Math., 1983, 68(1): 1-10.


\bibitem{Frank} Frank R, Read L. Jost solutions and direct scattering for the continuum Calogero-Moser equation. arXiv:2510.11403.



\bibitem{2} Gérard P, Lenzmann E. The Calogero–Moser derivative nonlinear Schr\"{o}dinger equation. Commun. Pure Appl. Math., 2024, 77(10): 4008-4062.

                
\bibitem{14} Kaup D J, Matsuno Y. The inverse scattering transform for the Benjamin–Ono equation. Stud. Appl. Math., 1998, 101(1): 73-98.
	 											

	 						
\bibitem{6} Killip R, Laurens T, Vişan M. Scaling-critical well-posedness for continuum Calogero–Moser models on the line. Commun. Amer. Math. Soc., 2025, 5(07): 284-320.
	 					 							
\bibitem{8}
Kim T, Kwon S. Soliton resolution for Calogero-Moser derivative nonlinear Schr\"{o}dinger equation. \textit{arXiv preprint} arXiv:2408.12843, 2024,  to appear in J. Eur. Math. Soc.
	 	
\bibitem{7} Kim K, Kim T, Kwon S. Construction of smooth chiral finite-time blow-up solutions to Calogero-Moser derivative nonlinear Schr\" {o}dinger equation.\textit{ arXiv preprint} arXiv:2404.09603, 2024, to appear in Mem. Amer. Math. Soc.		

	 		
\bibitem{3} Matsuno Y. Multiphase solutions and their reductions for a nonlocal nonlinear Schrödinger equation with focusing nonlinearity. Stud. Appl. Math., 2023, 151(3): 883-922.
	 			
\bibitem{9}Pelinovsky D. Intermediate nonlinear Schrödinger equation for internal waves in a fluid of finite depth. Phys. Lett. A, 1995, 197(5-6): 401-406.
	 								
\bibitem{10} Pelinovsky D E, Grimshaw R H J. A spectral transform for the intermediate nonlinear Schrödinger equation. J. Math. Phys., 1995, 36(8): 4203-4219.
	 		

\bibitem{Reed Simon book 2}Reed, M., Simon, B.
\emph{Methods of Modern Mathematical Physics: Vol.: 2.: Fourier
analysis, self-adjointness},  Academic Press, 1975.\label{Reed Simon book 2} 

\bibitem{Reed Simon book 4}Reed, M., Simon, B.
\emph{Methods of Modern Mathematical Physics: Vol.: 4.: Analysis of
Operators},  Academic Press New York, 1978.\label{Reed Simon book 4} 

\bibitem{Reed Simon book 1}Reed, M., Simon, B.
\emph{Methods of Modern Mathematical Physics: Vol.: 1.: Functional
analysis},  Gulf Professional Publishing, 1980.\label{Reed Simon book 1} 

\bibitem{RudinFA} Rudin, W. \emph{Functional Analysis}, McGraw-Hill Science/Engineering/Math, 2 edition, International
Series in Pure and Applied Mathematics, 1991.


\bibitem{SunCMP2021} Sun, R. Complete integrability of the Benjamin–Ono equation on the multi-soliton manifolds. Commun. Math. Phys., 2021, 383, 1051–1092.
	 						

\bibitem{11} Sun R. The intertwined derivative Schrödinger system of Calogero–Moser–Sutherland type. Lett. Math. Phys., 2024, 114(3): 74. 

                                                
\bibitem{15}Wu Y. Simplicity and finiteness of discrete spectrum of the Benjamin--Ono scattering operator. SIAM J.  Math. Anal., 2016, 48(2): 1348-1367.
	 											
\bibitem{16} Wu Y. Jost solutions and the direct scattering problem of the Benjamin--Ono equation. SIAM J.  Math. Anal., 2017, 49(6): 5158-5206.
\end{thebibliography}

\end{document}